\documentclass[12pt]{siamart251104}

\usepackage[utf8]{inputenc}
\usepackage{amsmath, amssymb}
\usepackage{graphicx}
\usepackage{tikz-cd}
\usepackage{bm}
\usepackage{bbm}
\usepackage{mathtools}
\usepackage{dsfont}
\usepackage{algorithm}
\usepackage[noend]{algorithmic}
\usepackage{hyperref} 
\usepackage{stmaryrd}
\usepackage[OT1]{fontenc}

\newcommand{\R}{\mathbb{R}}
\newcommand{\N}{\mathbb{N}}

\newcommand{\Z}{\mathbb{Z}}
\newcommand{\eps}{\varepsilon}
\newcommand{\F}{\mathcal{F}}

\newcommand{\Ifine}{I_{\text{fine}}}
\newcommand{\Icoarse}{I_{\text{coarse}}}
\newcommand{\E}{\mathbbm{E}}

\makeatletter
\def\tableofcontents{%
  \noindent\textbf{\contentsname}\par\vspace{.25em}%
  \@starttoc{toc}%
}

\def\l@section#1#2{%
  \vspace{0.25em}\par\noindent
  \def\numberline##1{\makebox[1.5em][l]{##1}}%
  \textbf{#1}\hfill#2\par
}

\def\l@subsection#1#2{%
  \par\noindent\hspace*{1.5em}%
  \def\numberline##1{\makebox[2.5em][l]{##1}}%
  #1\dotfill#2\par
}

\def\l@subsubsection#1#2{%
  \par\noindent\hspace*{4em}%
  \def\numberline##1{\makebox[3.5em][l]{##1}}%
  #1\dotfill#2\par
}
\makeatother

\colorlet{siaminlinkcolor}{blue}

\headers{Random blob methods for diffusion}{Craig and Murphy}

\title{Random blob methods for diffusion}

\author{Katy Craig%
\thanks{UC Santa Barbara, kcraig@ucsb.edu} 
\and Claire Murphy%
\thanks{UC Santa Barbara, claire\_murphy@ucsb.edu}
}

\begin{document}

\maketitle

\begin{abstract}
Linear and nonlinear diffusion equations arise in a range of phenomena of mathematical interest, including slow and fast diffusion, sandpile dynamics, height-constrained transport, the two-dimensional Navier-Stokes equation, and dynamics for sampling probability measures.
 In recent years, blob methods have attracted significant interest as an approach for numerically simulating these types of PDEs.
To address the $O(N^2)$ computational bottleneck of blob methods, we consider stochastic discretizations of space and time.
We compare the random batch method \cite{Jin2019, Jinetal2021} with a new approach, which we call the random multirate method.
Both of these methods build on classical stochastic methods in the optimization literature, with many similarities to stochastic gradient descent and random coordinate descent.
We find that, for linear and nonlinear diffusion equations, in the tradeoff between computational complexity and accuracy, the random multirate method has the best performance. 
On one hand, we prove that the random multirate method converges to the underlying ODE system at a rate of $O(k \Delta t)$, matching forward Euler, and show by example that this rate is theoretically sharp.
On the other hand, we observe even better rates of convergence for the random multirate method when applied to ODEs arising from blob methods for diffusive PDEs.
Finally, due to its ability to simulate a wide range of nonlinear diffusion equations, including height-constrained transport and sandpile dynamics,  our method succeeds in capturing key features of  PDEs for which  few  numerical approaches exist.

\end{abstract}



\section{Introduction}

\subsection{Motivation}
\label{sec:motivation}

In this work, we explore numerical approaches for the following linear and nonlinear diffusion equations, defined in terms of a fixed convex function $f: [0, \infty) \rightarrow \R \cup \{ \infty \}$ and a velocity field $v: [0, T] \times \R^d \rightarrow \R^d$:
\begin{equation}
    \begin{cases}
        \partial_t \rho = \nabla \cdot (\nabla f^*(p) + \rho v) \\
        p \in \partial f(\rho) \\
        \rho(0) = \rho_0 . 
    \end{cases}
    \tag{PDE}
    \label{eqn:PDE}
\end{equation}
Here, $f^*$ is the convex conjugate of $f$ and $\partial f(\rho)$ is the subdifferential of $f$ at $\rho$. 
Many equations take the form of equation \eqref{eqn:PDE} for appropriate choices of $f$ and $v$:
\begin{enumerate}
    \item Fast and slow diffusion equations \cite{Ambrosio2008, Jordan1998, Otto2001, Vazquez2007}: 
    Consider
    \begin{equation}
        \label{eqn1}
        f(s) = \begin{cases}        
            s \ln(s) - s & m = 1 \\
            \frac{s^m}{m - 1} & m \neq 1
        \end{cases} \; \Rightarrow \; \partial_t \rho = \Delta \left ( \rho^m \right ) + \nabla \cdot (\rho v).
    \end{equation}
    where $m \geq 1 - \frac{1}{d}$.
    We refer to the $m > 1$ case as ``slow diffusion," and the $m < 1$ case as ``fast diffusion."
    
    \item Sandpile dynamics \cite{Aronsson1996, Bantay1992, kwon}:
    Fix $v = 0$, $r_c > 0$, and consider 
    \begin{equation}
        \label{eqn18}
        f(s) = \begin{cases}
            0 & s \in [0, r_c] \\
            s \log \left ( \frac{s}{r_c} \right ) & s > r_c 
        \end{cases} \; \Rightarrow \; \text{`` $ \partial_t \rho = \begin{cases}
            \Delta \rho & \rho > r_c \\
            0  & \rho \leq r_c . \end{cases} $ "} 
    \end{equation}

    \item Height-constrained transport \cite{Craig2016, CRAIG2020239, DePhilippis2016, masson2026stifflimitnonhomogeneousconservation}:
    Consider
    \begin{equation}
        \label{eqn12}
        f(s) = \begin{cases}
    0 & s \in [0, 1] \\
    \infty &  \text{otherwise} 
    \end{cases} \; \Rightarrow \;  
        \text{`` $\begin{cases} \partial_t \rho + \nabla \cdot(\rho v) = 0 \\
        0 \leq \rho \leq 1 . \end{cases}$ "} 
    \end{equation}
    
    \item The viscosity formulation of the 2D Navier-Stokes equation \cite{Bertozzi2002}: Fix $d = 2$, and let $K$ be the Biot-Savart kernel. Consider
    \begin{equation}
        \label{eqn16}
        f(s) = \nu(s \ln(s) - s), v = K * \rho 
        \; \Rightarrow \;  
        \partial_t \rho = \nu\Delta \rho + v \cdot \nabla \rho. 
    \end{equation}
\end{enumerate}
In (2) and (3) above, equation \eqref{eqn:PDE} does not have a simple form in terms of differential operators.
Instead, we give a heuristic form in quotes.
There are few numerical methods available for simulating this full range of nonlinear Fokker-Planck equations, and our numerical results succeed in capturing key features of height-constrained transport, which motivated recent work by Masson et al. \cite{masson2026stifflimitnonhomogeneousconservation} (see Section \ref{sec:height_constraint}), as well as sandpile dynamics, as studied by Kwon and M\'esz\'aros \cite{kwon} (see Section \ref{sec:sandpile}).

We briefly mention that if $v(t, x) = \nabla V(x)$, where $V \in W^{1, \infty}(\R^d)$ is strongly convex, and if $f$ satisfies appropriate monotonicity, convexity, and lower semicontinuity assumptions, then 
equation \eqref{eqn:PDE} is the 2-Wasserstein gradient flow of the energy $\F$ defined by $\F(\rho) = \int f(\rho(x)) dx + \int V(x) \rho(x) dx $.
Furthermore, the solution to equation \eqref{eqn:PDE} converges exponential quickly to the unique minimizer of this energy functional  \cite{Ambrosio2008}.
Due to this property, equation \eqref{eqn:PDE} has attracted substantial interest in the sampling literature, as a candidate choice of dynamics to drive particle approximations toward minimizers of $\mathcal{F}$ \cite{chen2026deterministicsamplingmethodmaximum, chewi2024statisticaloptimaltransport, Craig2024}.

Because variants of the nonlinear Fokker-Planck equation describe many PDEs of mathematical interest and because of its connection to sampling, many previous works have developed particle methods to simulate certain variants of this PDE.
As discussed by Carrillo, Craig, and Patacchini \cite{carrillo2019blobmethoddiffusion}, naive particle methods fail to capture the behavior of this system: even if $\rho_0$ is a finitely supported measure, the solution of equation \eqref{eqn:PDE} does not remain a finitely supported measure.
A natural way to circumvent this difficulty is to define a \emph{regularized energy} $\F_{\eps}$ that approximates $\F$.
The hope is that the gradient flow induced by this regularized energy approximates the linear and nonlinear diffusion terms in equation \eqref{eqn:PDE}.
In this work, we follow the approach first introduced by Craig, Jacobs, and Turanova \cite{Craig2024}, due to the fact that they succeed in proving convergence as $\eps \rightarrow 0$ to equation \eqref{eqn:PDE} for a wide range of internal energy densities $f$, including all the examples listed above.
Their regularized energy induces the following PDE: 
\begin{equation}
    \label{eqn4}
    \begin{cases}
        \partial_t \rho^{\varepsilon} = \nabla \cdot \left (\rho^{\varepsilon} \left ( \nabla p^{\varepsilon} + v \right ) \right ) \\
        p^{\varepsilon} = \varphi_{\varepsilon} * f'_{\varepsilon}(\varphi_{\varepsilon} * \rho^{\varepsilon}) \\
        \rho^{\varepsilon}(0) = \rho^{\varepsilon}_0 . 
    \end{cases} 
\end{equation}

Originally, the authors of \cite{Craig2024} chose $f_{\eps}$ to be a variant of the Moreau-Yosida regularization, but we found that other regularizations perform better numerically.
In particular, we require that $f''_{\eps}$ is locally Lipschitz, $(s, f_{\eps}'(s)) \rightarrow \partial f(s)$ for a.e. $s \in \text{int}(\text{Dom}(f))$, and $f_{\eps}(s) \nearrow \infty$ for $s \notin \text{Dom}(f)$.
If $f''$ is locally Lipschitz, we set $f = f_{\eps}$.
While the outer convolution with the mollifier $\varphi_{\eps}$ in the definition of $p_{\eps}$ was important for the analytic study \cite{Craig2024}, as the $f_{\eps}'$ we consider are already very regular, we drop this outer mollifier.
With these changes in hand, when $\rho_0^{\eps} = \sum_{j = 1}^N m_j \delta_{x_j^0}$ is a finitely supported measure, then the solution to the modified equation \eqref{eqn4} is of the form $\rho^{\eps}(t) = \sum_{j = 1}^N m_j \delta_{x_j(t)}$, where
\begin{equation}
    \tag{ODE}
    \begin{cases}
        \dot{x}_i(t) = - f''_{\varepsilon} \left(\sum \limits_{j = 1}^N m_j \varphi_{\varepsilon}(x_i(t) - x_j(t)) \right )
        \sum \limits_{j = 1}^N m_j \nabla \varphi_{\varepsilon}(x_i(t) - x_j(t)) - v(t, x_i(t)) \\
        x_i(0) = x_i^0 . 
    \end{cases} 
    \label{eqn:xdot}
\end{equation}
Thus, numerically approximating solutions of equation \eqref{eqn:PDE} is reduced to approximating $\rho_0$ by a finitely supported measure and solving the above system of ODEs.

When integrating system \eqref{eqn:xdot}, implicit methods such as the backward Euler method are costly, as they require optimization of a function over a $Nd$-dimensional space.
We found that we could obtain comparably accurate solutions at a much lower computational cost by instead implementing the forward Euler (FE) method with a finer step size.
\footnote{As noted by Xu and Li \cite{xu2025forward}, the forward Euler scheme applied directly to the Wasserstein gradient flow of an unregularized energy can fail dramatically.
These authors emphasize the importance of regularization, allowing explicit discrete-time solvers to successfully approximate the underlying PDE without losing their particle structure, as we do in the present work.}

\subsection{Main Results}
\label{sec:main}

Even with the computational savings of explicit methods, solving system \eqref{eqn:xdot} presents a challenge, as calculating the velocity of the $i$th particles requires knowledge of all $N$ particle locations. 
Thus, if we have $N$ particles and take $M$ time steps, the total computational cost of the FE method is $O(MN^2)$.
The goal of the present work is to explore stochastic numerical methods that improve this computation rate.

We first consider the random batch (RB) method introduced by Jin, Li, and Liu \cite{Jin2019, Jinetal2021}.
At each time step, particles are split into random ``batches" of size $p$.
Particles only interact within their batch, i.e. one modifies the ODE system in system \eqref{eqn:xdot} so that one only sums over batch particles and one normalizes the weights $\{ m_i \}$ so they sum to the same total over each batch.
See Section \ref{sec:alg_details} for more details and pseudocode.
The total computational cost over all time steps is $O(MNp)$.
While the RB method effectively decreases computation time, we find that for the variants of equation \eqref{eqn:PDE} that we consider, it often loses significant accuracy, even in the most accurate case that $p = \frac{N}{2}$, i.e. two batches.

To avoid some of the numerical drawbacks of the RB method, we propose an alternative stochastic numerical method, called the random multirate (RM) method.
Given a fixed integer $k$, called the time step ratio, at each $k$th time step, we select $N - p$ particles to take a coarse time step of size $k \Delta t$, while the remaining $p$ particles take $k$ fine time steps of size $\Delta t$.
At a high level, we view the RB method as a stochastic \emph{spatial} discretization, while the RM method is a stochastic \emph{temporal} discretization.
Again, see Section \ref{sec:alg_details} for more details and pseudocode.
The total computational cost of the RM method is 
\[ O \left (MN \left ( p + \frac{N - p}{k}\right ) \right ) . \]
The first term of the above expression is equal to the computational cost of the RB method.
Thus, for a given choice of $M, N$, and $p$, we expect that the RM method will be at least as slow as the RB method.
However, in many examples, we find that the RM method is significantly more accurate than the RB method, ultimately achieving the best tradeoff between accuracy and efficiency.
The benefits of the RM method are most obvious in the slow diffusion setting, and we observe the strongest performance of the RM method in the height-constraint case.
Furthermore, in practice, the computational savings of the RB and RM method exceed the theoretical predictions above, due to caching issues once the number of particles is sufficiently large.


\subsection{Related Works}

    \label{sec:related}
    Nonlocal particle approximations of PDEs have been an active area of research for many years; 
    see \cite{carrillo2019blobmethoddiffusion, CarrilloCraigYao2019, Craig2023} for a comprehensive review of the related literature and \cite{Burger2023, Carrillo2024} for some recent advancements.
    Our main focus is not to develop a new nonlocal particle method for diffusion, but to study improved discretizations of the ODE system arising from existing particle methods.
    
    Recent years have seen significant contributions to improving ODE discretizations of related problems arising in interacting particle systems, with the RB method among its most significant developments.
    Furthermore, the RB method itself draws on two fundamental stochastic approaches in optimization. 
    On the one hand, there are key similarities between the RB method and stochastic gradient descent (SGD), where a small batch of particles are chosen at each time step to form the ``noisy" gradient \cite{bottou1999online, bubeck2015convex}.
    On the other hand, the RB method is also similar to random coordinate descent (RCD), where only a small sample of randomly selected coordinates are updated at each time step \cite{nesterov2012efficiency, wright2015coordinate}.

    Our method is inspired by the \emph{multirate forward Euler method} \cite{Gunther2025} and \emph{multiple time-stepping method} \cite{martin2025multiple}.
    Traditionally, these methods are used in settings where there are known ``fast" and ``slow" components of a system, which are partitioned once at the start of the integration.
    However, static partitioning is poorly-suited to system \eqref{eqn:xdot}, since certain particles initially move fast and later move slowly, while other particles display the opposite behavior; see for example Figure 5 in \cite{CarrilloCraigYao2019}.
   
    Savcenco, Hundsdorfer, and Verwer avoid static partitioning components of the system, instead implementing a \emph{self-adjusting} scheme \cite{Savcenco2006}.
    For each particle, the coarseness of the time step is determined during time integration. 
    Integration is performed twice, using two different methods of order $p$ and $p-1$, giving two numerical solutions.
    Depending on the difference between these solutions, 
    one may reject the time step and redo the integration with a smaller step size.
    Due to the large number of particles, such an approach would be computationally prohibitive in our setting. 
    Instead, by randomly repartitioning particles rather than mathematically estimating local errors for every single trajectory, the RM method bypasses this computational bottleneck.

    \subsection{Outline} The rest of the paper is organized as follows. 
    In Section \ref{sec:properties}, we provide details and pseudocode for both the RB and RM methods, review a known error bound for the RB method, and state and prove an error bound for the RM method.
    Section \ref{sec:num_approach} describes details of our numerical approach: choice of initialization and regularization parameters and how error and runtime are measured.
    In Section \ref{sec:interacting}, we provide a numerical example to demonstrate even better rates of convergence of the RM method when applied to ODEs arising from blob methods for diffusion.
    As shown in Sections \ref{sec:diffusion_exact}, \ref{sec:diffusion_longtime}, and \ref{sec:convergence_rate}, this trend persists for a range of porous medium and fast diffusion equations, both on bounded time intervals and in the longtime limit. 
    It also persists for nonsmooth internal energy densities, as in the case of height-constrained transport (Section \ref{sec:height_constraint}) and sandpile dynamics (Section \ref{sec:sandpile}).
    We even see this trend continues in two dimensions, as explored in Section $\ref{sec:2d}$.
    Our codebase is available at \url{https://github.com/claireamurphy/Random-Blob-Methods-for-Diffusion}. 

\section{Random Blob Methods}

\label{sec:properties}

\subsection{Two Stochastic Discretizations}

\label{sec:alg_details}

We now describe the two different stochastic discretizations that we compare in this work.
First, we consider the random batch (RB) method.
As discussed in Section \ref{sec:main}, at each time step, particles are randomly split into batches of size $p$ and only interact with particles in their batch.
More specifically, if $I$ equals the indices of particles in some batch, define
\begin{equation}
    \label{eqn:F}
    F_{RB}(x, t, \tilde{x}_1, \ldots, \tilde{x}_N, I) = - f''_{\varepsilon} \left(\sum \limits_{j \in I} \tilde{m}_j \varphi_{\varepsilon}(x - \tilde{x}_j) \right )
        \sum \limits_{j \in I} \tilde{m}_j \nabla \varphi_{\varepsilon}(x - \tilde{x}_j) - v(t, x) ,
\end{equation}
where $\tilde{m}_j = m_j / \sum_{i \in I} m_i$.
A forward Euler implementation of the RB method is then given by algorithm \ref{alg:rbm}.
While the authors of \cite{Jin2019, Jinetal2021} originally consider a specific choice of $F_{RB}$, as in equation \eqref{eqn11}, we present a more general case.

\begin{algorithm}[H]
\caption{} \label{alg:rbm} 
\begin{algorithmic}[1]
\STATE \textbf{Input:} $\Delta t$, $p, N, T, F_{RB}, \{ x_1^0, \ldots, x_N^0 \}$ \;
\STATE Define number of batches $B = N / p$
\FOR{$m=0$ to $T / \Delta t - 1$ }
    \STATE Randomly divide $\{1, \ldots, N\}$ into $B$ batches, $\{ I_1, \ldots I_{B} \}$, each of size $p$
    \FOR{$b=1$ to $B$}
        \FOR{$i \in I_b$}
            \STATE $x_i^{m + 1} \gets x_i^m + \Delta t F_{RB}(x_i^m, m \Delta t, x_1^m, \ldots, x_N^m, I_b)$
        \ENDFOR
    \ENDFOR
    
\ENDFOR
\end{algorithmic}
\end{algorithm}

Next, we consider the random multirate (RM) method.
We randomly select a group of $N - p$ particles to take one time step of size $k \Delta t$ in the forward Euler scheme.
The remaining $p$ particles take $k$ fine time steps with size $\Delta t$.
When computing the trajectories of the particles with fine time step, we use a linear interpolation of the particles with coarse time step to approximate their location.
More specifically, define
\begin{equation}
    \label{eqn15}
    F_{RM}(x, t, \tilde{x}_1, \ldots, \tilde{x}_N) = F_{RB}(x, t, \tilde{x}_1, \ldots, \tilde{x}_N, \{ 1, \ldots, N \}) .
\end{equation}
We also define
\begin{equation}
    \label{eqn19}
    L(i, m, I_{\text{fine}}, I_{\text{coarse}}) = 
\begin{cases}
    x_i^m & i \in I_{\text{fine}} \\
   \left ( 1 - \frac{\ell + 1}{k}\right )x_i^{nk} + \frac{\ell + 1}{k} x_i^{(n + 1)k} & i \in I_{\text{coarse}}, 
\end{cases} 
\end{equation}
where $m = nk + \ell$ for $0 \leq \ell < k$.
The function $L$ leaves particles with fine time step unchanged, while performing a linear interpolation of particles with coarse time steps.
The RM method is then given by algorithm \ref{alg:rmm}.

\begin{algorithm}[H]
    \caption{} \label{alg:rmm} 
    \begin{algorithmic}[1]
    \STATE \textbf{Input:} $\Delta t, k, p, N, T, F_{RM}, L, \{ x_1^0, \ldots, x_N^0 \}$ 
    \FOR{$m=0$ to $T / \Delta t - 1$ }
        \IF{$m \; \text{mod} \; k = 0$}
            \STATE Let $I_{\text{fine}}$ equal $p$ randomly selected indices of $\{1, \ldots, N \}$
            
            \STATE $I_{\text{coarse}} \gets \{1, \ldots, N \} \backslash I_{\text{fine}}$
    
            \FOR{$i \in I_{\text{coarse}}$}
                \STATE $x^{m + 1}_i \gets x^m_i + k \Delta t F_{RM}(x_i^m, m \Delta t, x_1^m, \ldots, x_N^m)$
            \ENDFOR
        \ENDIF
    
        \STATE
            \FOR{$1 \leq i \leq N$}
                \STATE$\tilde{x}_i \gets L(i, m, I_{\text{fine}}, I_{\text{coarse}})$
            \ENDFOR

            \FOR{$i \in I_{\text{fine}}$}
               \STATE $x_i^{m + 1} \gets x_i^m + \Delta t F_{RM}(x_i^m, m\Delta t, \tilde{x}_1, \ldots, \tilde{x}_N)$
            \ENDFOR

        \ENDFOR
    \end{algorithmic}
    \end{algorithm}

\subsection{An Error Analysis for Random Blob Methods}
In our error analysis, for simplicity, we assume $F_{RB}$ and $F_{RM}$ do not depend on time.
Let $X_i(t)$ be the solution to the ODE system 
\begin{equation}
    \label{eqn13}
    \begin{cases}
        \dot{X_i}(t) = G(X_i(t), X_1(t), \ldots, X_N(t)) \\
    X_i(0) = x_i^0 .
    \end{cases}
\end{equation}
To start, we review the results of the convergence analysis proved for the RB method.
Following the proof of convergence in the original RB paper \cite{Jin2019}, the authors of \cite{Jinetal2021} consider the case 
\begin{align}
    G(x, \tilde{x}_1, \ldots, \tilde{x}_N) = \sum_{j = 1}^N m_j K_{ij}(x, \tilde{x}_j) - v(x) , \\
    \label{eqn11}
    F_{RB}(x, \tilde{x}_1, \ldots, \tilde{x}_N, I) = \sum_{j \in I} \tilde{m}_j K_{ij}(x, \tilde{x}_j) - v(x) .
\end{align}
Note that \cite{Jinetal2021} does not discretize the batched dynamics in time; rather, at each time step the batched trajectory is assumed to be integrated exactly.
If $x_i(t)$ is the solution to the RB method at time $t$ with initial solution $x_i^0$, then one has the following \emph{strong} convergence result:
\[ \sup_{t \leq T} \sum_{i = 1}^N m_i \E[|x_i(t) - X_i(t)|^2] \leq C \sqrt{\frac{\Delta t}{p -1} + (\Delta t)^2} . \]
Here, $C$ is a constant independent of $N$ and $p$.

Now, we analyze the error of the RM method.

\begin{theorem}
    \label{thm:error_analysis}
    Fix $G: \R^{d \times Nd} \rightarrow \R^d$ such that for any $x, y \in \R^d, \bm{u}, \bm{v} \in \R^{dN}$,    
    \[ ||G(x, \textbf{u})|| \leq L_0 \text{ and } ||G(x, \bm{u}) - G(y, \bm{v})|| \leq L_1\max(||x - y||, ||u_1 - v_1||, \ldots, ||u_N - v_N||) . \]
    Let $X_i(t)$ be the solution to the ODE system given by equation \eqref{eqn13}.
    Let $x_i^m$ be the solution to the RM method with $F_{RM} = G$ after $m$ time steps for any choices of $I_{\text{fine}}$ and $I_{\text{coarse}}$, with initial condition $x_i^0$.
    If $M$ is divisible by $k$ and $T := M \Delta t$, \[ \max_{1 \leq i \leq N} ||X_i(T) - x_i^M ||\leq \frac{L_0 e^{6L_1T}}{4} \left ( L_1(k\Delta t)^2 + k\Delta t \right ) . \]
\end{theorem}

Notice that the convergence rate is $O(k \Delta t)$, as is the error rate for the forward Euler method with time step $k \Delta t$.
On one hand, this rate is sharp for general ODEs of the form of equation $\eqref{eqn13}$ for any choice of $I_{\text{fine}}$ and $I_{\text{coarse}}$.
This can be seen by an explicit calculation in the case $G(x, \textbf{u}) = x$ and is also seen in the simulations of Section \ref{sec:toy}.
However, in many of the following numerical examples for our \eqref{eqn:xdot}, we see that the RM method outperforms the FE method for a fixed step size $k \Delta t$.
This shows that, for ODEs arising in blob methods for diffusion, the RM method often achieves convergence rates superior to these theoretical guarantees.

There are major structural differences between the strong convergence result in \cite{Jinetal2021} and Theorem \ref{thm:error_analysis}. 
First, the error considered in \cite{Jinetal2021} is only the error due to the batched trajectories; time integration at each step is assumed to be exact.
On the other hand, our result considers error due to time discretization and the trajectory approximations at each fine time step. 
Secondly, the result presented in \cite{Jinetal2021} depends on $p$, the batch size, whereas our result applies for any choice of $p$.
Finally, the result in \cite{Jinetal2021} is in-expectation, whereas our bound is deterministic, holding for every realization of the partition.
We now turn to the proof of Theorem \ref{thm:error_analysis}.

\begin{proof}[Proof of Theorem \ref{thm:error_analysis}]
    Fix $m \in \mathbb{N}$ divisible by $k$.
    Define $\bm{y}^m = \bm{X}(t_m)$.  For $\ell \in \mathbb{N}$ satisfying $0 < \ell \leq k$,
    let $\bm{y}^{m + \ell}, \bm{x}^{m + \ell}$ equal the solution to the RM method after $\ell$ time steps with initial values $\bm{y}^m$, $\bm{x}^m$, respectively, where we select the same $\Ifine$ and $\Icoarse$ for both solutions.
    Define 
    \begin{equation}
        \label{eqn:E_def}
        E^{m + \ell}_i = || x_i^{m + \ell} - X_i((m + \ell) \Delta t)|| \; \; \text{ and }\; \; E^{m + \ell} = \max_{1 \leq i \leq N} E^{m + \ell}_i . 
    \end{equation}
    Define 
    \begin{equation}
        H^{m + \ell}_i = ||x_i^{m+ \ell} - y_i^{m+ \ell}|| \; \; \text{ and }\; \;  H^{m+ \ell} = \max_{1 \leq i \leq N} H_i^{m + \ell}. 
    \end{equation}
    By Lemma \ref{propBound} (a stability result), Lemma \ref{lte} (a consistency result), and the fact that $H^m = E^m$, we obtain
    \begin{align*}
        E_i^{m + k}
        & = ||x_i^{m + k} - X_i((m + k) \Delta t)|| \\
        & \leq ||x_i^{m + k} - y_i^{m + k}|| + ||y_i^{m + k} - X_i((m + k) \Delta t)|| \\
        & \leq e^{4 L_1 k \Delta t} E^m + L_0 L_1 e^{2L_1 k \Delta t} \left ( L_1(k\Delta t)^3 + (k\Delta t)^2 \right ) .
    \end{align*}
    Now, choose $n$ such that $M = nk$.
    Using the fact that $E^0 = 0$, we invoke Lemma \ref{discrete_gronwall} (a discrete Gr\"onwall-type inequality) to obtain
    \begin{align*}
        E^M 
        & \leq L_0 L_1 e^{2L_1 k \Delta t} \left ( L_1(k\Delta t)^3 + (k\Delta t)^2 \right ) \left ( \frac{e^{4 L_1 n k \Delta t} - 1}{e^{4 L_1 k \Delta t} - 1} \right ) \\
        & \leq L_0 L_1 e^{2L_1 T} \left ( L_1(k\Delta t)^3 + (k\Delta t)^2 \right ) \left ( \frac{e^{4 L_1 T} - 1}{4 L_1 k \Delta t} \right ) \\
        & = \frac{L_0 e^{2 L_1 T} \left ( e^{4 L_1 T} - 1\right ) }{4} \left ( L_1(k\Delta t)^2 + k\Delta t \right ) \\
        & \leq \frac{L_0 e^{6L_1T}}{4} \left ( L_1(k\Delta t)^2 + k\Delta t \right ) . 
    \end{align*}
    
\end{proof}


We close with a lemma providing sufficient conditions on the internal energy density $f_{\eps}$ and external velocity $v$ in system $\eqref{eqn:xdot}$ for which the hypotheses in Theorem \ref{thm:error_analysis} hold.
The proof of this lemma is given in appendix \ref{sec:error_append}.

\begin{lemma}
    \label{lm:relax_assumptions}
    Suppose $f_{\eps}'': (0, \infty) \rightarrow [0, \infty)$, $\varphi_{\eps}: \R^d \rightarrow (0, \infty), \nabla \varphi_{\eps}: \R^d \rightarrow \R^d$, and $v: \R^d \rightarrow \R^d$ are locally Lipschitz and all $m_j$ are positive.
    Further suppose that there exists some $R > 0$ such that, for all $1 \leq i \leq N$,
    \begin{enumerate}
        \item for all $1 \leq m \leq M$, $x_i^m \in B(0, R)$,
        \item for all $0 \leq t \leq M \Delta t$, $X_i(t) \in B(0, R)$.
    \end{enumerate} 
    Define $G := F_{RM}$, where $F_{RM}$ is defined in equation \eqref{eqn15}.
    Then, up to modifying $G$ outside of $B(0, R)$ in a way which changes neither discrete nor continuous time solutions, $G$ satisfies the boundedness and Lipschitz conditions in Theorem \ref{thm:error_analysis}.
\end{lemma}
    
\section{Numerical Approach for Random Blob Methods}

\label{sec:num_approach}

\subsection{Choice of Initialization and Parameters}

In all our simulations, we suppose that $\rho_0(x)$ has compact support in $\R^d$ or decays sufficiently quickly so that we commit arbitrarily small error by truncating its support.
We initialize particles $\{ x_1^0, \ldots, x_N^0 \}$ on a grid with side length $h$ that covers this support, and we define $m_i$ as the integral of $\rho_0(x)$ on the grid square centered at $x_i$.
We evolve particles according to the ODE system described by system \eqref{eqn:xdot}.
In all simulations, the mollifier $\varphi_{\varepsilon}$ is given by $\varphi_{\varepsilon}(x) = \frac{1}{\eps^d} \eta_1 \left ( \frac{x}{\eps}\right ) $,
where $\eta_1(x)$ is the standard Gaussian.
This choice of mollifier is justified in appendix \ref{sec:moll}, where we compare the performance of several different mollifiers.
Unless specified otherwise, in all simulations, we set $\eps = 4h^{0.99}$.
At times, we choose $f_{\eps} \neq f$, and we allow this to be a different choice of $\eps$ than above.

At time $t$, the particle solution is given by
\begin{equation}
    \label{eqn2}
    \rho^{\eps}(t, x) = \sum \limits_{i = 1}^N m_i \delta_{x_i(t)}(x) ,
\end{equation}
where $x_i(t)$ are the solutions of system \eqref{eqn:xdot}.
In order to visualize the particle solution and compare to classical solutions of the PDE, we plot its convolution with a Gaussian mollifier; that is, $\mu(t, x) = \varphi_{\eps} * \rho^{\eps}(t, x)$.
Again, we allow this to be a different choice of $\varepsilon$.





\subsection{Measuring Runtime and Error}
    \label{sec:error_and_runtime}

    To measure runtime, we use the time() method from the Python time library.
    We measure the time from the start of the first ODE time step to the completion of the final ODE time step.
    All simulations were completed on a 2025 Apple Macbook Air with an Apple M4 chip.

    For the majority of simulations, we measure error by approximating the 2-Wasser\-stein distance between the particle solution and exact continuum solution at some final time $T$.
    We discretize the exact solution at final time $T$, then compute the 2-Wasserstein distance between our particle solution and this discretized solution.
    This calculation is performed using the emd2() function from the POT library.
    See Appendix $\ref{sec:exact}$ for exact continuum solutions for several variants of equation $\eqref{eqn:PDE}$.

    The RB and RM methods are both stochastic methods; their performance varies for different choices of random seed.
    In many simulations throughout this work, we calculate the particle solution for five or ten distinct random seeds.
    This averaging is done on a case-by-case basis, and expressly noted in each simulation where it occurs.
    
    In the $d = 1$ setting, we consider only the average runtime and error, as in all observed cases, there was little meaningful variation in either runtime or error.
    In the $d = 2$ setting, variation in runtime is almost completely negligible, as the total runtime of simulations masks any variation due to background processes.
    On the other hand, variation in error is much more pronounced in many $d = 2$ cases.
    In all $d = 2$ simulations, we include vertical bars corresponding to standard deviation in errors across all seeds, while only considering the average runtime over all simulations.
   
    \subsection{Batch Size and Time Step Ratio}
    \label{sec:batch_size}

    In this section, we examine the effect of $p$, the batch size (RB) or number of particles with fine time step (RM).
    We also consider the time step ratio $k$ used in the RM method.

    For $f$ given by equation $\eqref{eqn1}$ with $m = 2$ and $v(x) = \frac{x}{3}$, consider the corresponding ODE system given by system \eqref{eqn:xdot}.
    We set $h = 0.01$ and choose the same uniformly spaced time steps for both methods; in this case
    \begin{equation}
        \label{eqn:t_steps}
        \Delta t \in \{ 0.005, 0.006, \ldots, 0.013, 0.014 \} .
    \end{equation}
    The time steps are chosen to highlight the regime of largest variation in the performance of the methods. 
    For each time step, we consider the runtime and error and plot the values as a point of the graph.
    In the case of the RB method, the FE method is the same as the ``one batch case," i.e. the case where all particles are included in the same batch.
    In the case of the RM method, the FE method is the same as the case that 100\% of particles move with fine time step.
    
    In the two leftmost panels of Figure \ref{fig8}, we vary the number of batches (for the RB method) and number of particles moving with fine time step (for the RM method).
    The RB method works as expected; as we choose more and more batches, we reduce the computation time, while increasing the error.
    On the other hand, in the case of the RM method, we observe better accuracy when only 50\% of particles move with fine time step, compared to the FE method.
    This outperformance of RM with FE is a common feature of our simulations of blob methods for diffusion.
    We believe that ``averaging" the fine and coarse step sizes removes problematic outliers in our simulations.

    In the rightmost panel, we implement the RM method, varying the time step ratio $k$.
    Notice that the case $k = 1$ coincides with the FE method.
    As the time step ratio increases, the computation time decreases, while the solutions become less accurate.
    As in all other experiments, we found that setting $k = 2$ led to optimal results in the case of the RM method.
    Based on these simulations, unless otherwise noted, in all of the following simulations we use the following parameters:
    \begin{enumerate}
        \item In the case of the RB method, we use two batches, i.e. $p = N /2$.
        \item In the case of the RM method, 50\% of the particles move according to the fine time step, and we set the time step ratio $k = 2$.
    \end{enumerate}

    \begin{figure}
        \centering
        \includegraphics[width=0.35\linewidth]{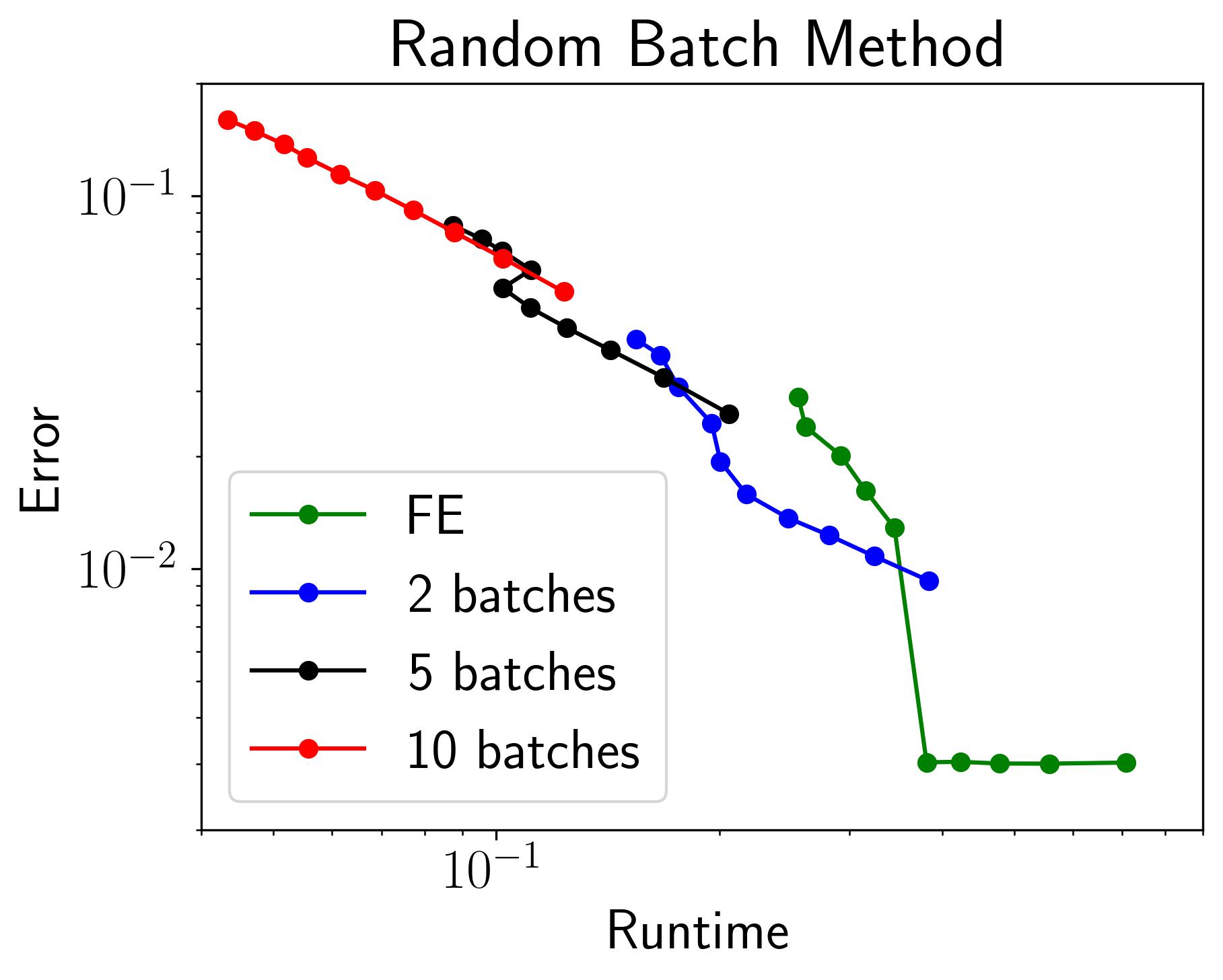}
        \includegraphics[width=0.3\linewidth]{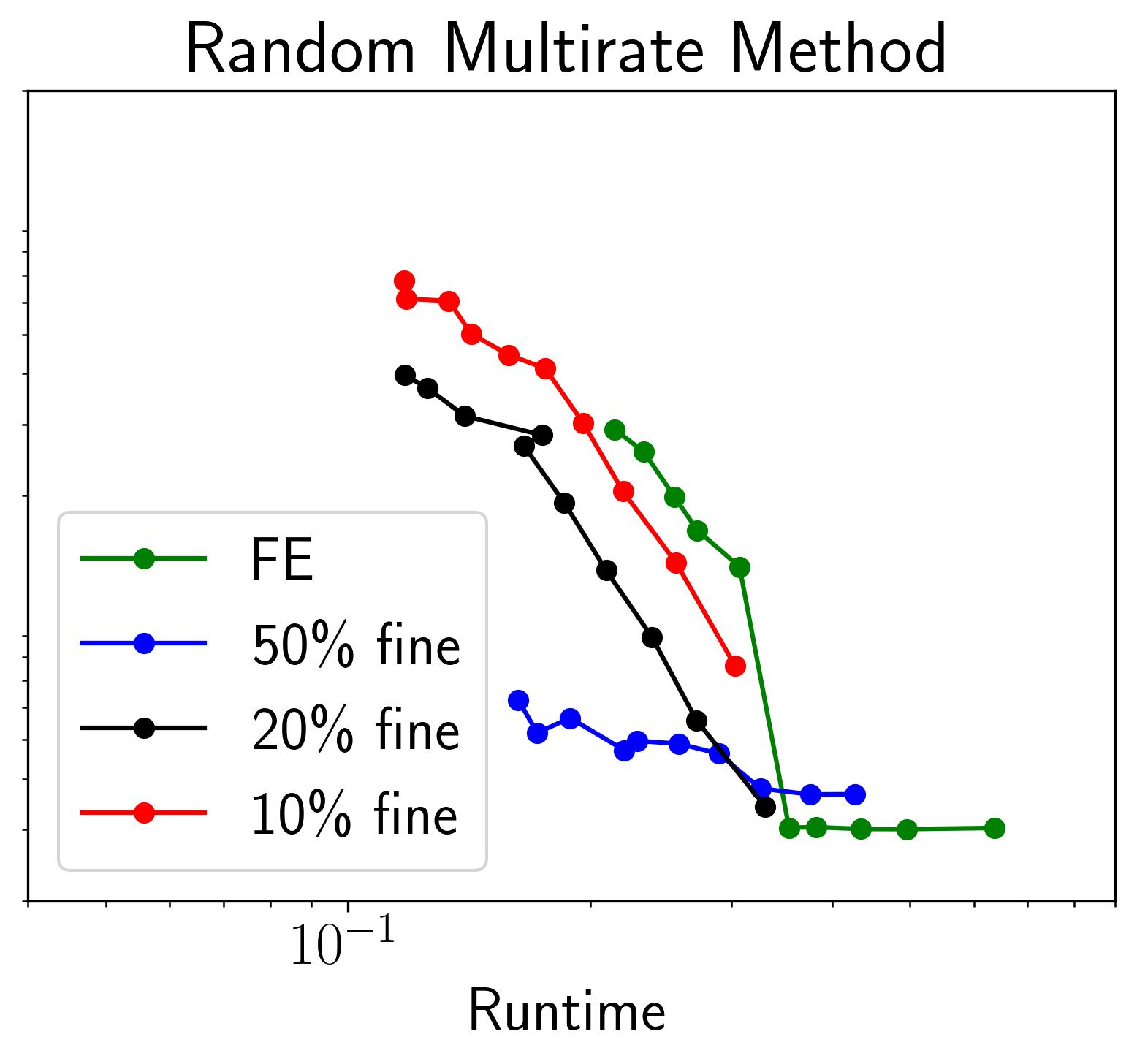}
        \includegraphics[width=0.3\linewidth]{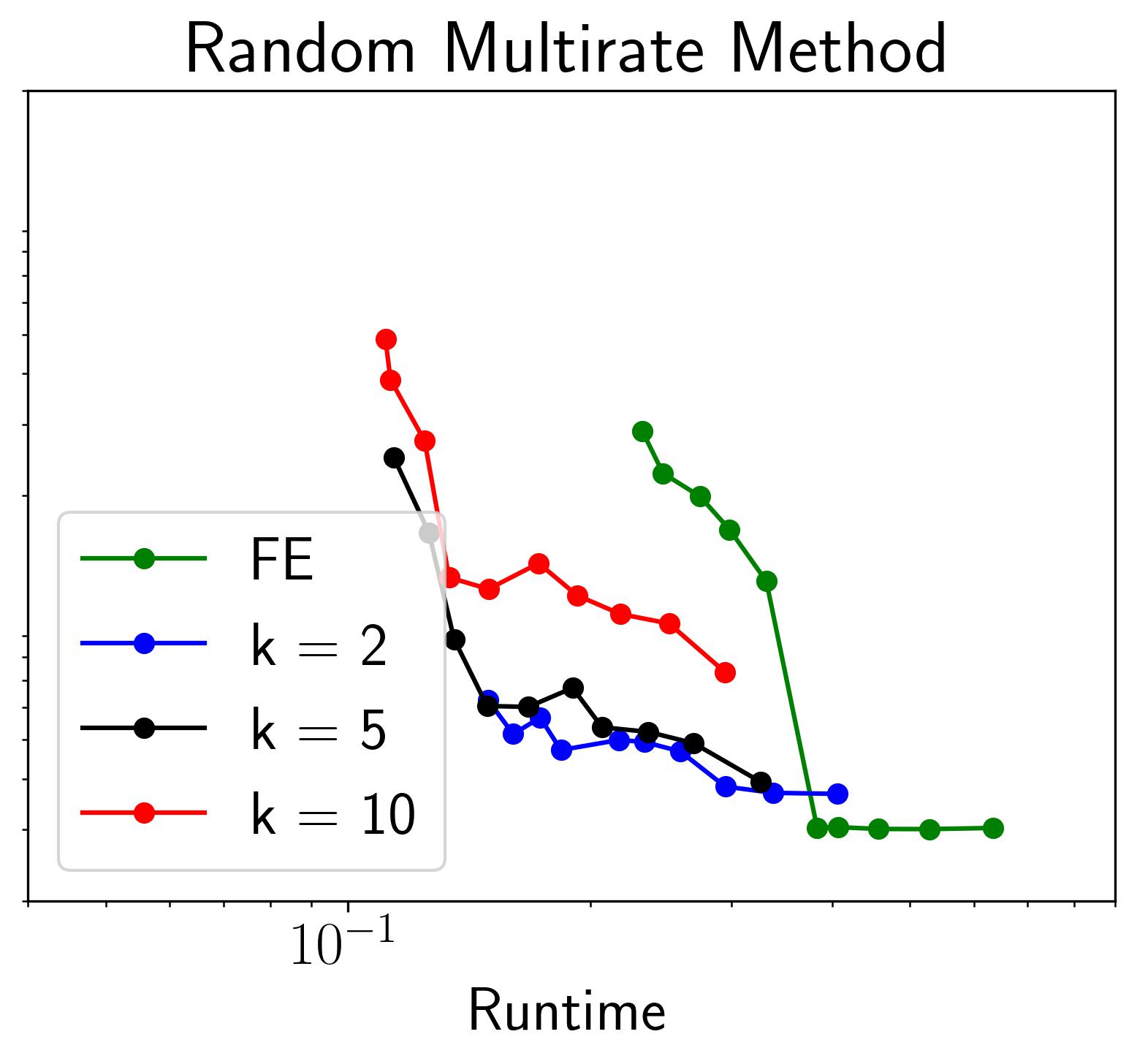}
        \caption{For the RB method, we consider the effect of different batch sizes. For the RM method, we consider both the effect of the proportion of particles moving with fine time step and the effect of time step ratio $k$.}
        \label{fig8}
    \end{figure}
     
\section{One-Dimensional Simulations}

\label{sec:1d}

\subsection{Interacting Particle Systems}

\label{sec:interacting}

\subsubsection{A Non-Diffusion Example}

    \label{sec:toy}
    While the RB and RM method both give a way to efficiently simulate interacting particle systems, the structure of the dynamics impacts the effectiveness of both methods. 
    In order to illustrate the limitations of the RM method, we begin by presenting an interacting particle system that does not arise from a blob method discretization of a diffusive PDE.
    We reproduce Figure 4 in \cite{Jin2019}, adding in data corresponding to the RM method.
    This figure corresponds to an interacting particle system with $N$ particles.
    The trajectory of the $i$th particle is given by the following expression:
    \begin{equation}
        \dot{x}_i(t) = \frac{1}{N - 1} \sum \limits_{j = 1}^N \frac{x_i(t) - x_j(t)}{1 + |x_i(t) - x_j(t)|^2} - x_i(t) .
        \label{eqn9}
    \end{equation}
    The initial conditions of the particles are samples taken from $\rho_0(x) = \frac{\sqrt{4 - x^2}}{2 \pi}$ using acceptance-rejection sampling.
    We calculate a reference solution using the FE method with very fine time step $2^{-15}$.
    Let $\{ \tilde{x}_1(T), \ldots, \tilde{x}_N(T) \}$ equal the reference solution at final time $T$.
    For time steps $2^{-6}, \ldots, 2^{-2}$, we solve the ODE system using the FE, RB, and RM methods.
    These differing solutions are represented by different dots in the figure.
    We assess the accuracy of our method using the error function
    \begin{equation}
        E_N(x_1(T), \ldots, x_N(T)) = \sqrt{\frac{1}{N} \sum_{i = 1}^N |\tilde{x}_i(T) - x_i(T)|^2} .
        \label{eqn10}
    \end{equation}
    We simulate the case $N = 10^2$ (in black), $N = 10^3$ (in blue), and $N = 10^4$ (in red).
    
    As expected, as the time step becomes more fine, the runtime increases, while the error decreases.
    We observe the same behavior for the FE and RB methods as in \cite{Jin2019}; the time savings of the RB method are most apparent in the case of more particles.
    In this example, for a fixed time step, the solution by the RM method behaves very similar to the solution by the FE method. 
    Consequently, we do not observe a large benefit of the RM method for the present ODE.
   
    \begin{figure}
        \centering
        
        \includegraphics[width = .4 \linewidth]{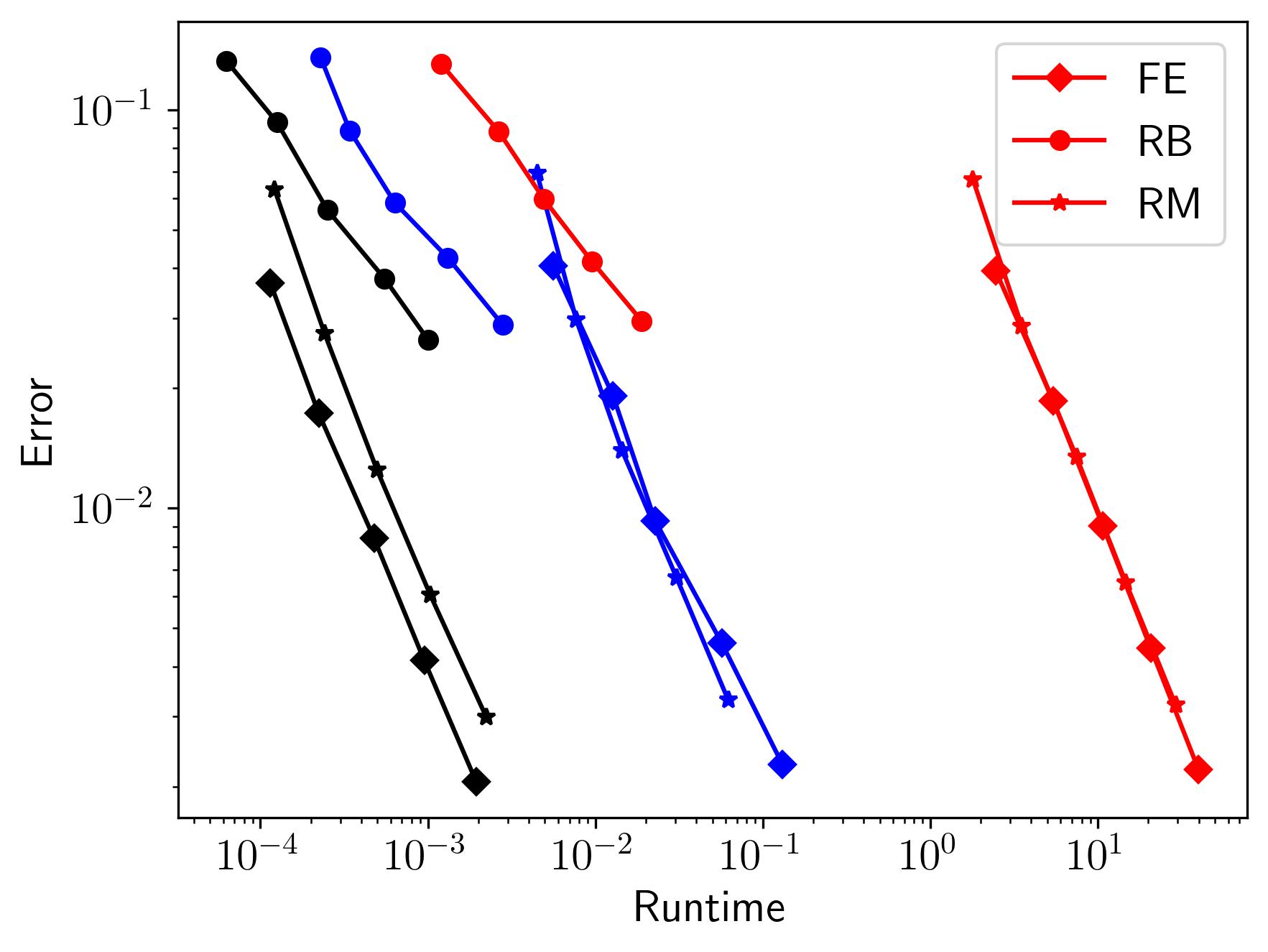} 

        \caption{A non-diffusion example demonstrating the strength of the RB method. We reproduce Figure 4 in \cite{Jin2019}, adding a third line corresponding to the RM method. We simulate the case $N = 10^2$ (in black), $N = 10^3$ (in blue), and $N = 10^4$ (in red).}
        \label{fig1}
    \end{figure}

\subsubsection{A Diffusion-Type Regime}
    \label{sec:diffusion_type}
    The behavior is markedly different in the diffusion regime, where we observe a large benefit of the RM method.
    We consider the system \eqref{eqn:xdot}, which arises from a blob method discretization of a diffusive PDE.
    More specifically, particles move according to system $\eqref{eqn:xdot}$, where $f$ is given by equation $\eqref{eqn1}$ with $m = 2$ and $v(x) = \frac{x}{3}$.
    As in Section \ref{sec:toy}, we fix $N$, and for various time steps, we calculate the particle solution up to time $T = 1$ using the FE, RM, and RB methods.
    For each of the three methods, we simulate each time step in expression \eqref{eqn:t_steps}, which are chosen to highlight the regime of largest variation in the performance of FE method. 
    We simulate the case $N = 521$ ($h = 0.01$, in black), $N = 1041$ ($h = 0.005$, in blue), and $N = 2083$ ($h = 0.0025$, in red).
    
    In Figure \ref{fig6}, we use two different approaches to assess the accuracy of our particle solutions.
    In the left panel, as in Section \ref{sec:toy}, we use equation \eqref{eqn10} to compare the particle solution to a reference solution calculated using fine time step $10^{-4}$.
    In the right panel, we approximate the 2-Wasserstein distance between the particle solution and the exact solution at time $T = 1$, as described in Section \ref{sec:error_and_runtime}.
  
    In both cases, we see that the error for the FE method is relatively constant once the time step is sufficiently fine. 
    Interestingly, in certain cases, for a fixed number of particles and time step, using the RM method is \emph{more accurate} than using the FE method.
    We attribute this to inherent ``averaging" of two different time steps in the RM regime, which may preserve stability.
    Throughout this work, we continue to observe this trend, especially in the slow diffusion regime.
    
    The RB method struggles in comparison to the reference particle solution, but performs favorably in comparison with the exact continuum solution.
    Similar to Figure \ref{fig1}, we see the benefits of the RB method most clearly in the case of more particles.

    \begin{figure}
        \centering
        \includegraphics[width=0.4\linewidth]{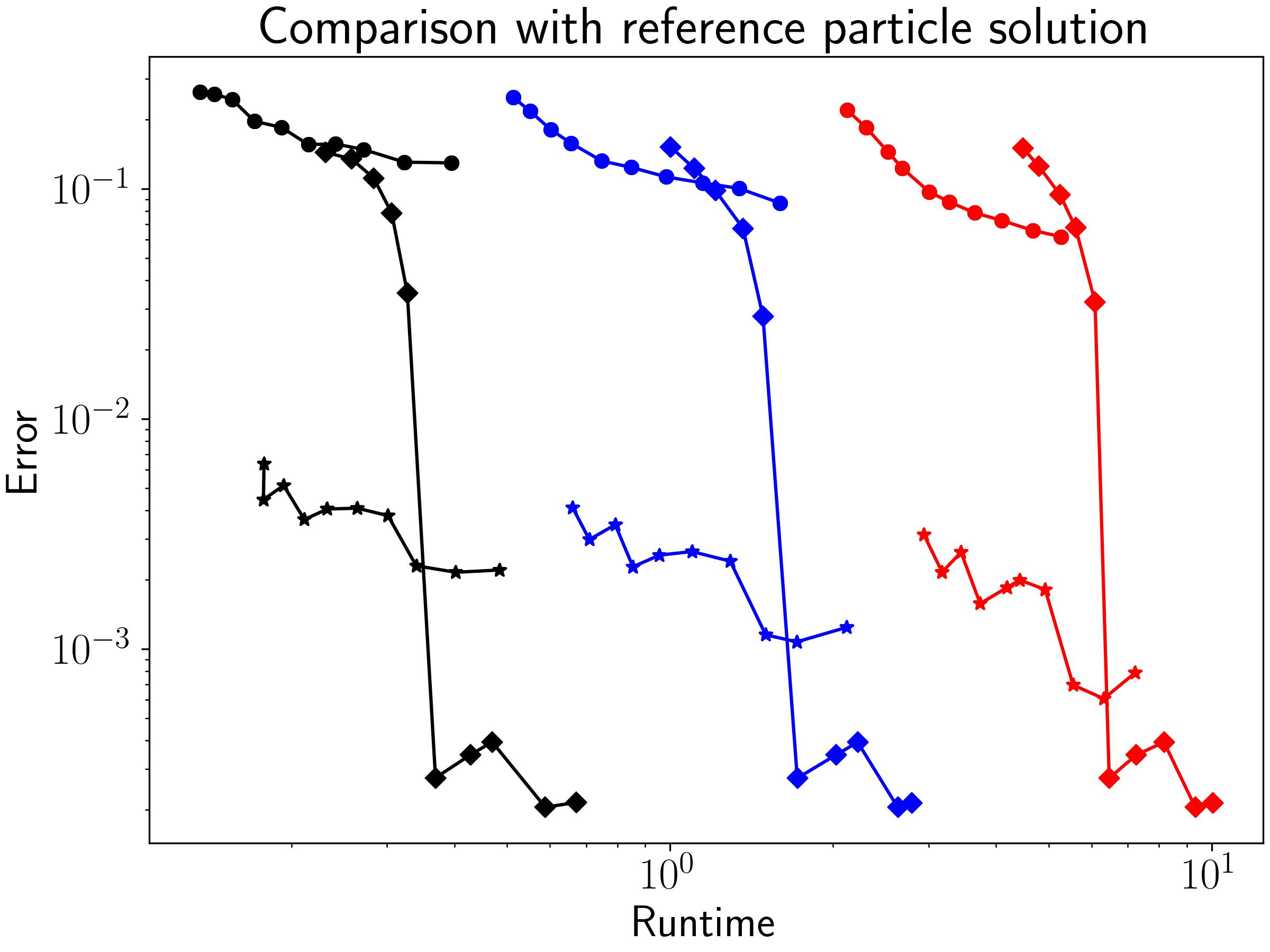}
        \includegraphics[width=0.475\linewidth]{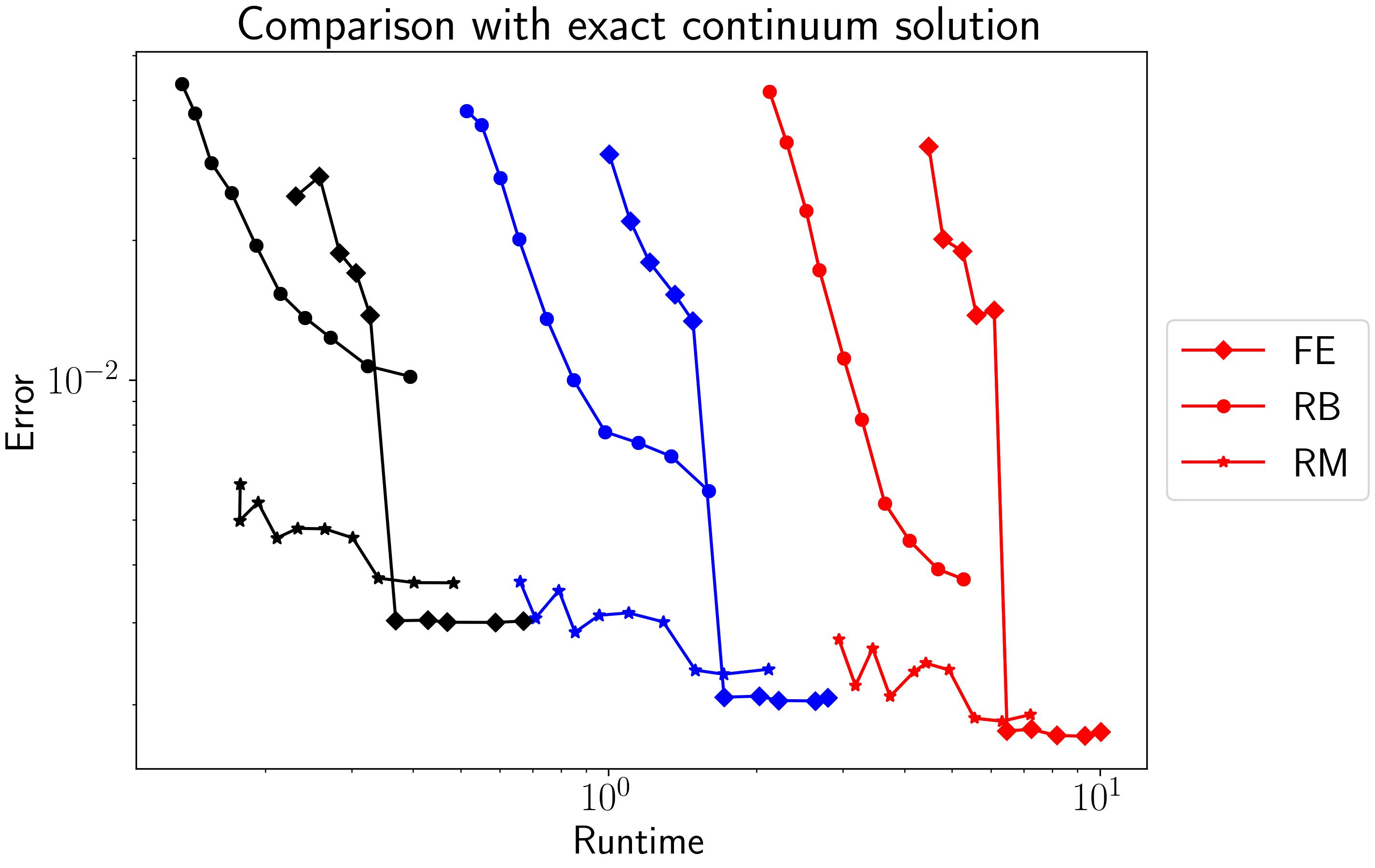}
        \caption{A diffusion-type regime demonstrating the strength of the RM method. We simulate the $d = 1$ nonlinear diffusion equation with $m = 2$ and external velocity field $v(x) = \frac{x}{3}$. The left panel measures error with respect to a reference particle solution, while the right panel measures error with respect to the exact continuum solution. We simulate the case $N = 521$ (in black), $N = 1041$ (in blue), and $N = 2083$ (in red).}
        \label{fig6}
    \end{figure}

    \subsection{Diffusion Equations: Dynamics}
    \label{sec:diffusion_exact}

    We now consider the RB and RM methods through the lens of diffusion equations, noting how the type of diffusion and external velocity field affects these methods.
    In particular, we consider system \eqref{eqn:xdot}, where $f$ is given by equation \eqref{eqn1}; we vary $m$ and the external velocity field $v$.
    In the case $v = 0$, we choose $h = 0.005$ and final time $T = 0.1$; in the case $v = \nabla V$ for quadratic external potential $V$, we choose $h = 0.01$ and final time $T = 1$.
    In Figure \ref{fig12}, for each of these choices of $V$, we plot the exact continuum solution and the particle solution using the FE method in the case $m = 2$.
    Visually, we observe that the particle solution accurately models the exact continuum solution.

   \begin{figure}
       \centering
       \includegraphics[width=0.4\linewidth]{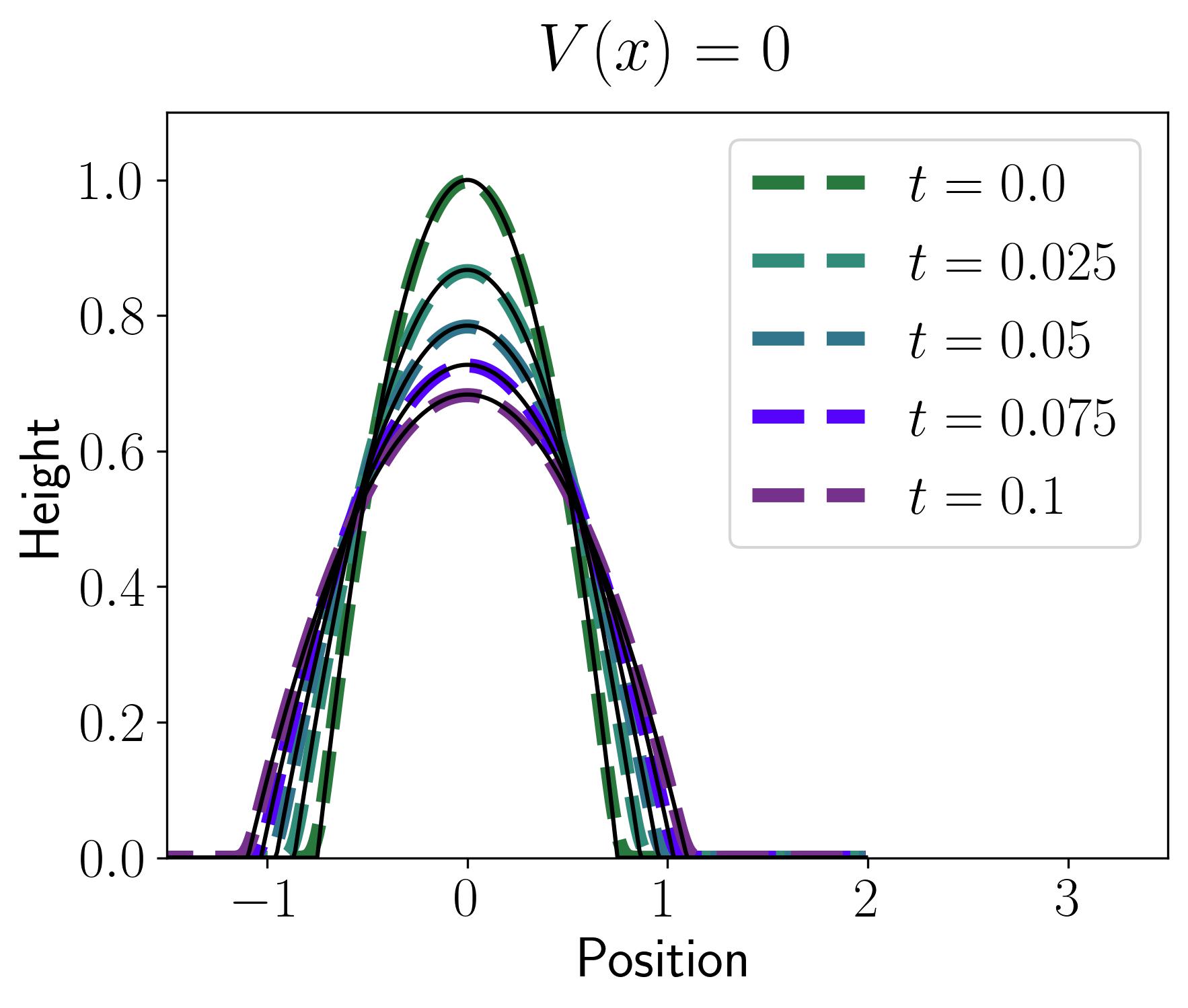}
       \includegraphics[width=0.3825\linewidth]{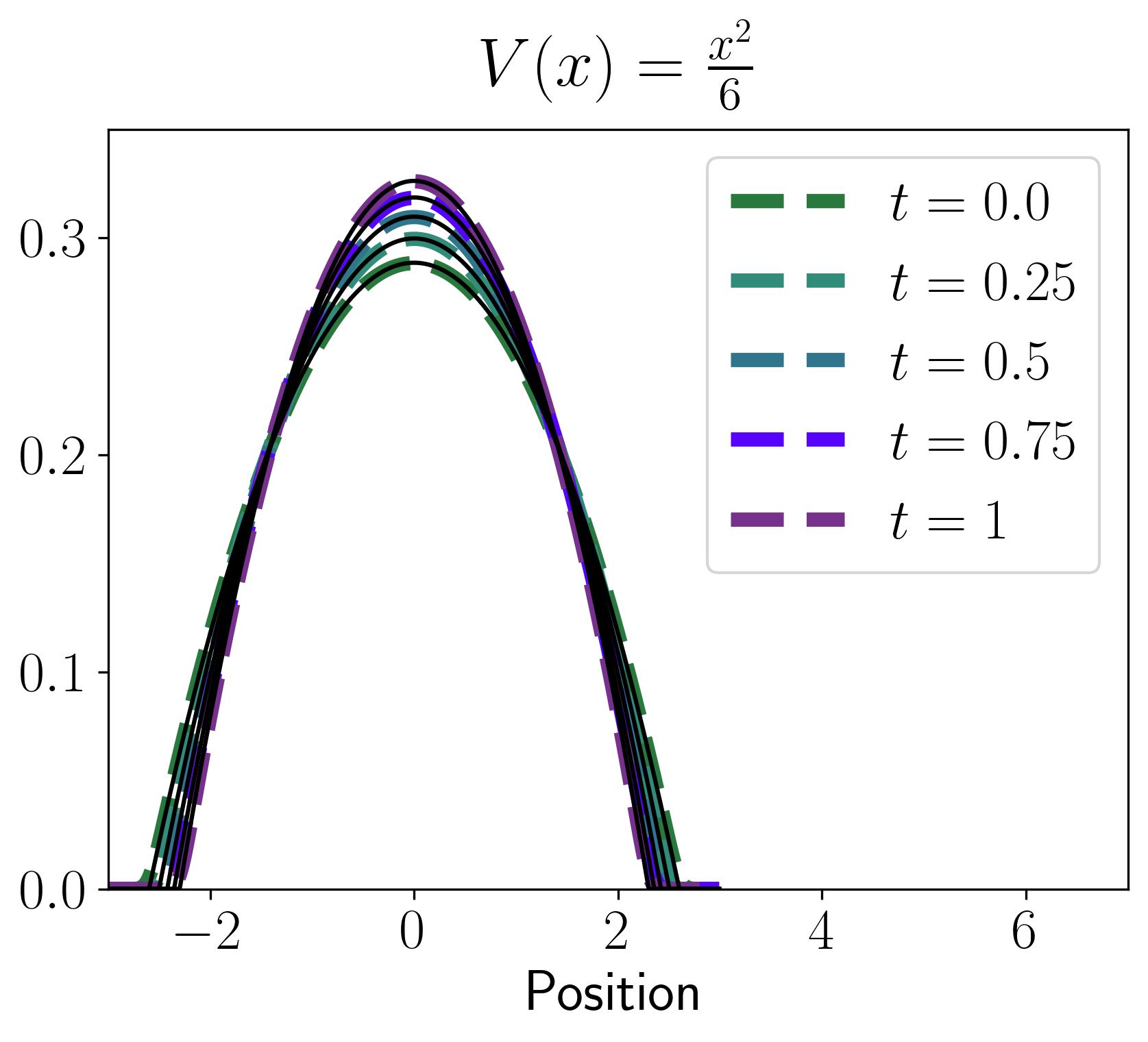}
       \caption{Particle and exact continuum solutions to the $d = 1$ nonlinear diffusion equation with $m = 2$ and two different external potentials $V(x)$. Solid black lines correspond to the exact continuum solution, while dashed lines correspond to the particle solution.}
       \label{fig12}
   \end{figure}
    
    Figure \ref{fig2} demonstrates the effect of time step on the accuracy of the particle solution for the FE, RB, and RM methods.
    For different choices of time step, indicated by different dots in the figure, we run the particle method up to some final time $T$, and approximate the 2-Wasserstein distance between the particle solution and the exact continuum solution at time $T$. 
    In the case $V(x) = 0$, we take $h = 0.005$.
    In order to achieve comparable runtimes, in the case of a quadratic external potential, we take $h = 0.01$ for $m < 5$, $h = 0.005$ for $m = 5$.  

    For $m \geq 1$, we consider the average error and runtime over ten different random seeds to ensure our results are robust against random variations.
    On the other hand, in the case $m = \frac{3}{4}$, we only consider one random seed due to the large computational time required for each individual simulation. 
    
    Notice that as $m$ increases, the FE method exhibits more pronounced ``drops:" once the time step is sufficiently fine, the error of the FE method decreases sharply. 
    On the other hand, the RB method does not exhibit drops.
    Instead, as $m$ increases, the RB method becomes more sensitive to the choice of time step.
    For $m < 1$, the error is almost constant, while for $m$ large, the error decreases as the step size becomes finer.
    Finally, the RM method also displays more gradual decreases in error.
    Its relative accuracy compared to the two other methods is most apparent for large $m$.

    \begin{figure}[htbp]
        
        \raggedright \boxed{V(x) = 0,\; T = 0.1} \\
        \centering
        \includegraphics[width=.29\linewidth]{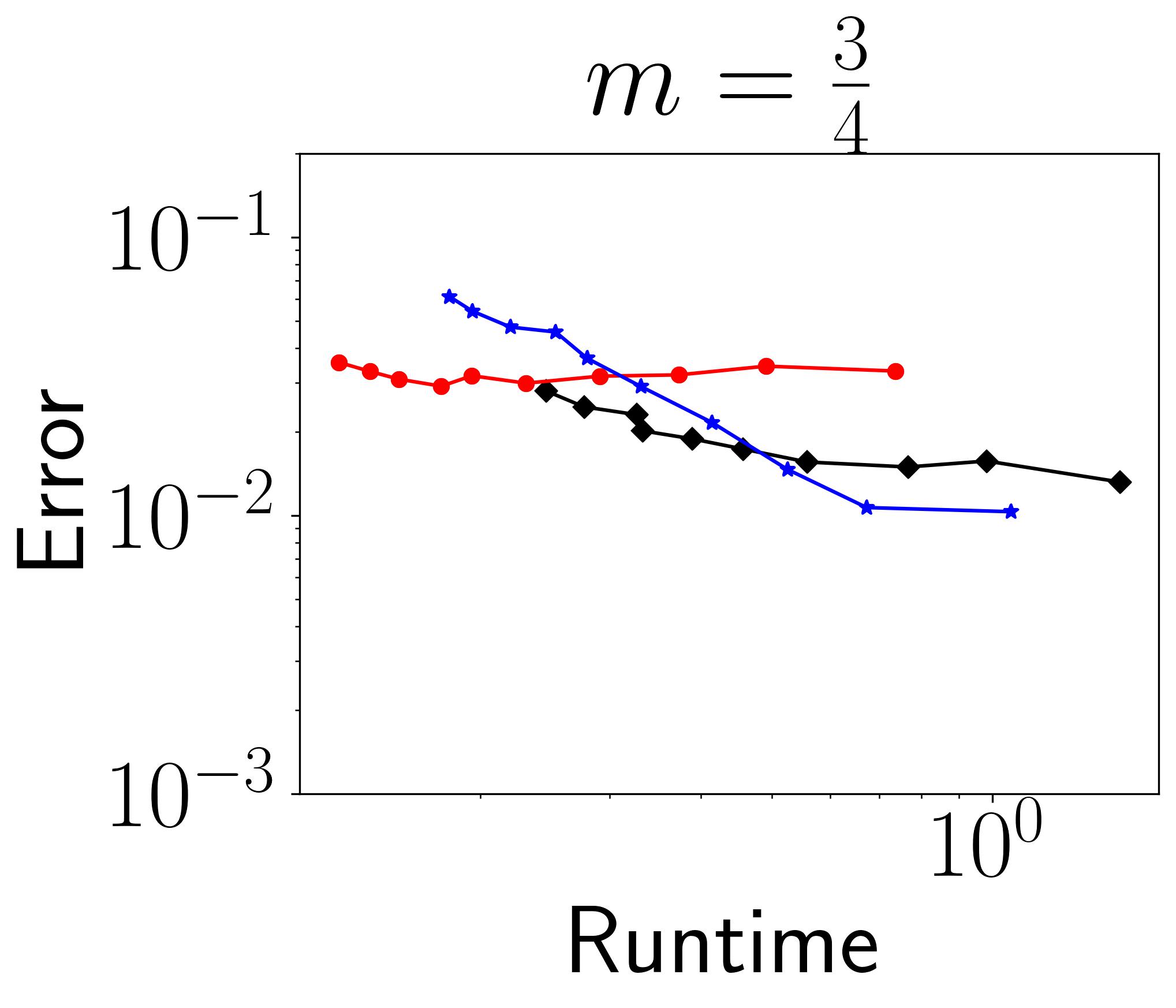}
        \includegraphics[width=.23\linewidth]{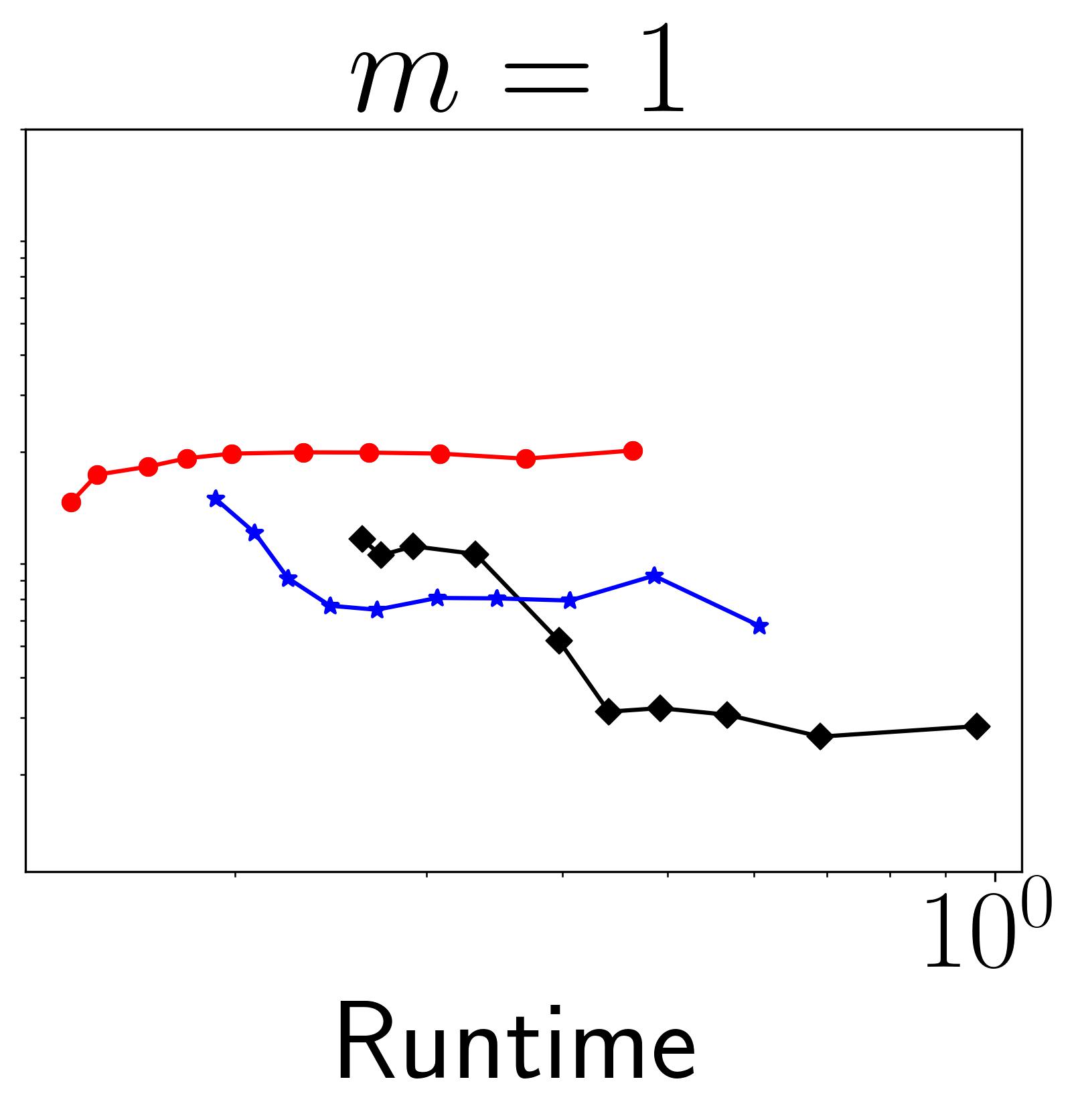}
        \includegraphics[width=.22\linewidth]{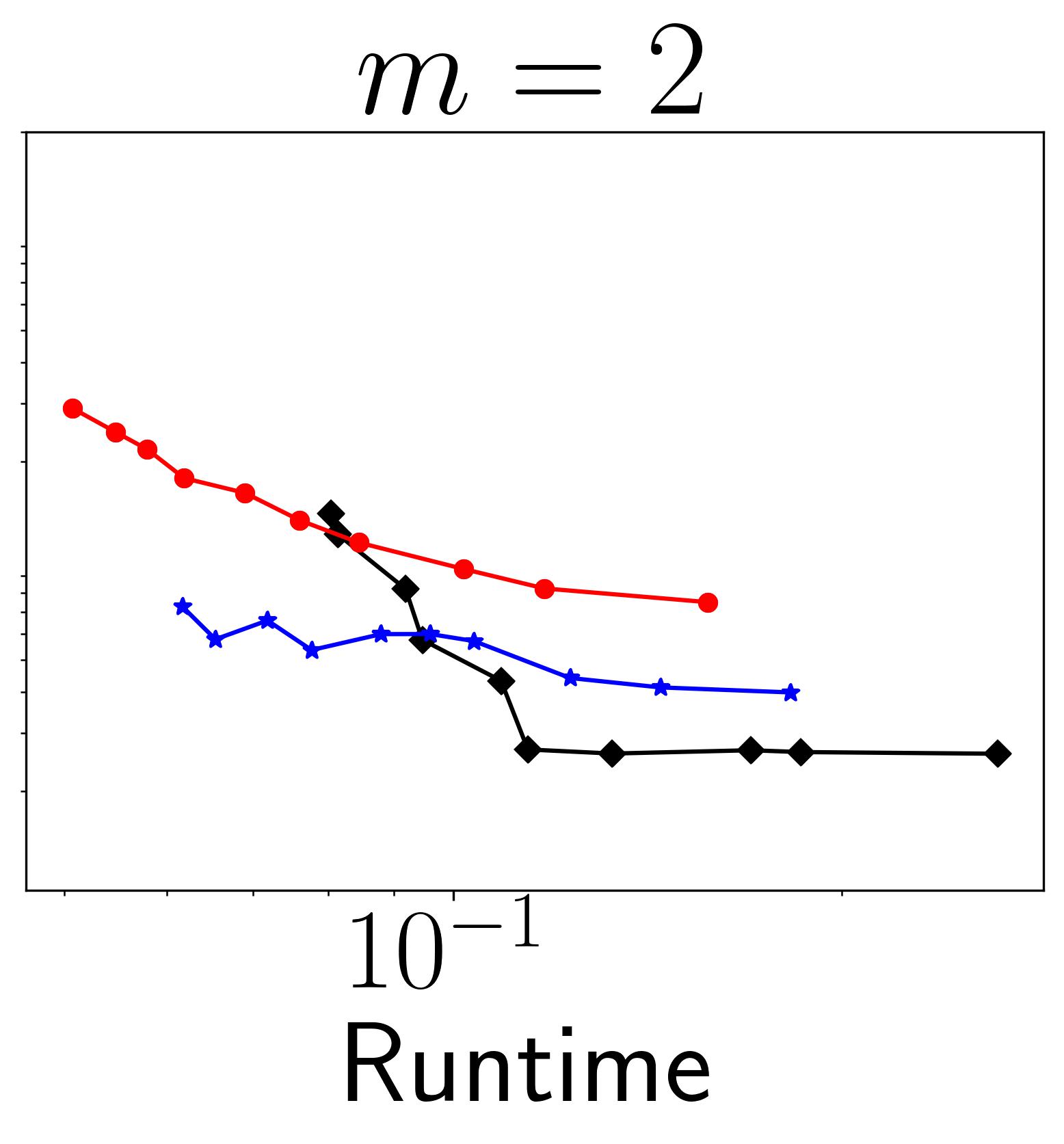}
        \includegraphics[width=.22\linewidth]{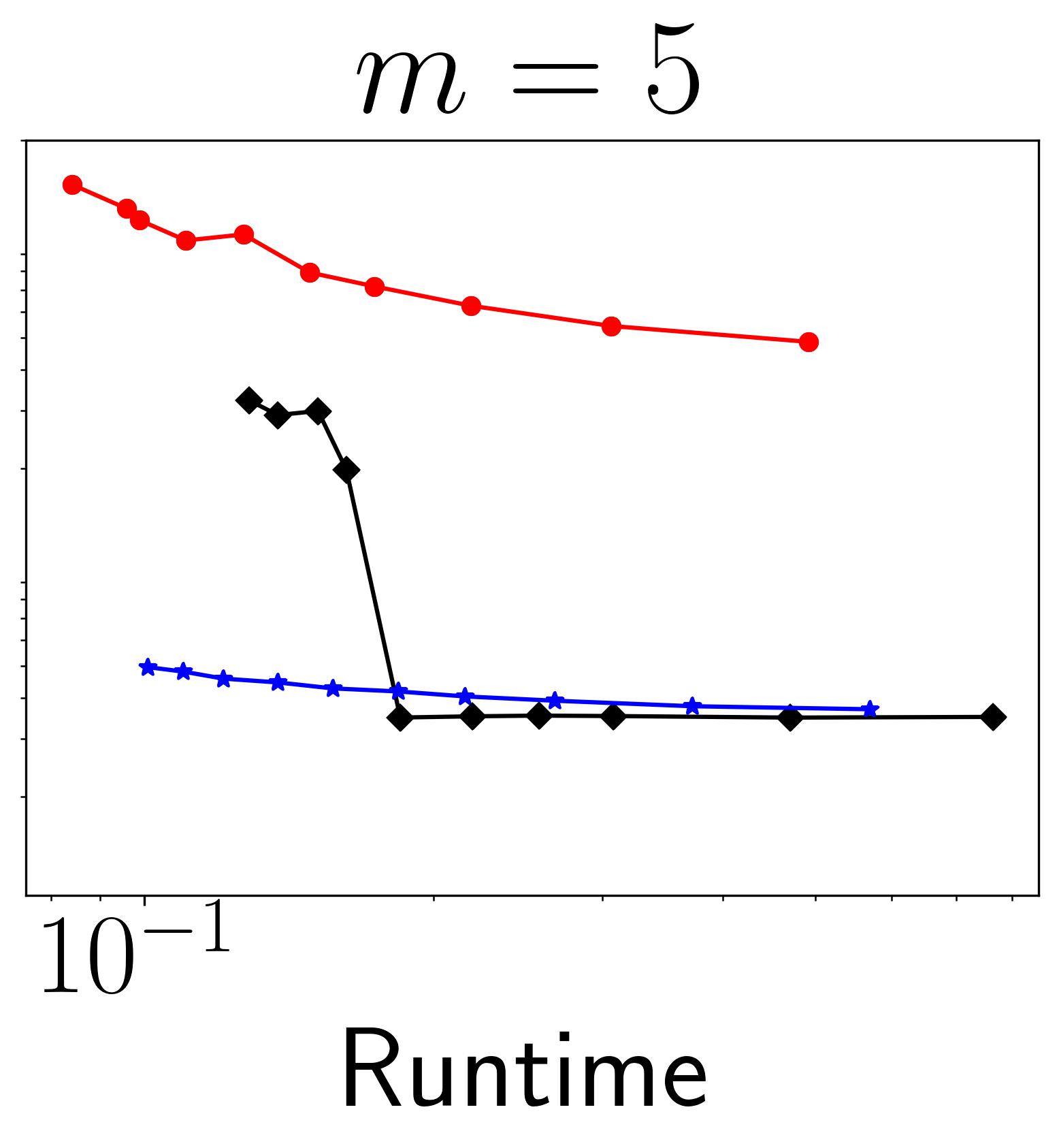}

        \raggedright \boxed{V(x) = \frac{x^2}{2(m+1)},\; T = 1} \\
        \centering
        \includegraphics[width=.29\linewidth]{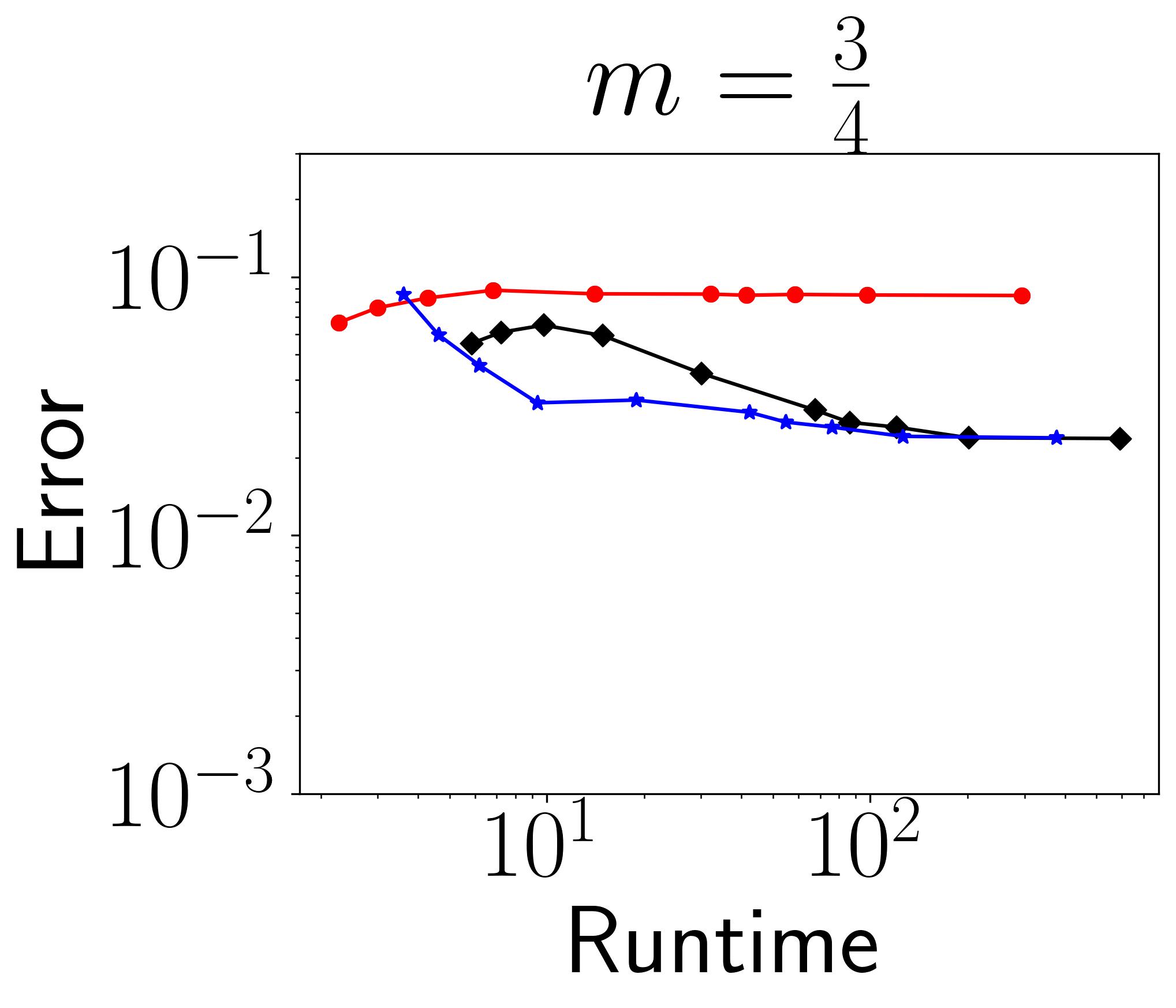}
        \includegraphics[width=.225\linewidth]{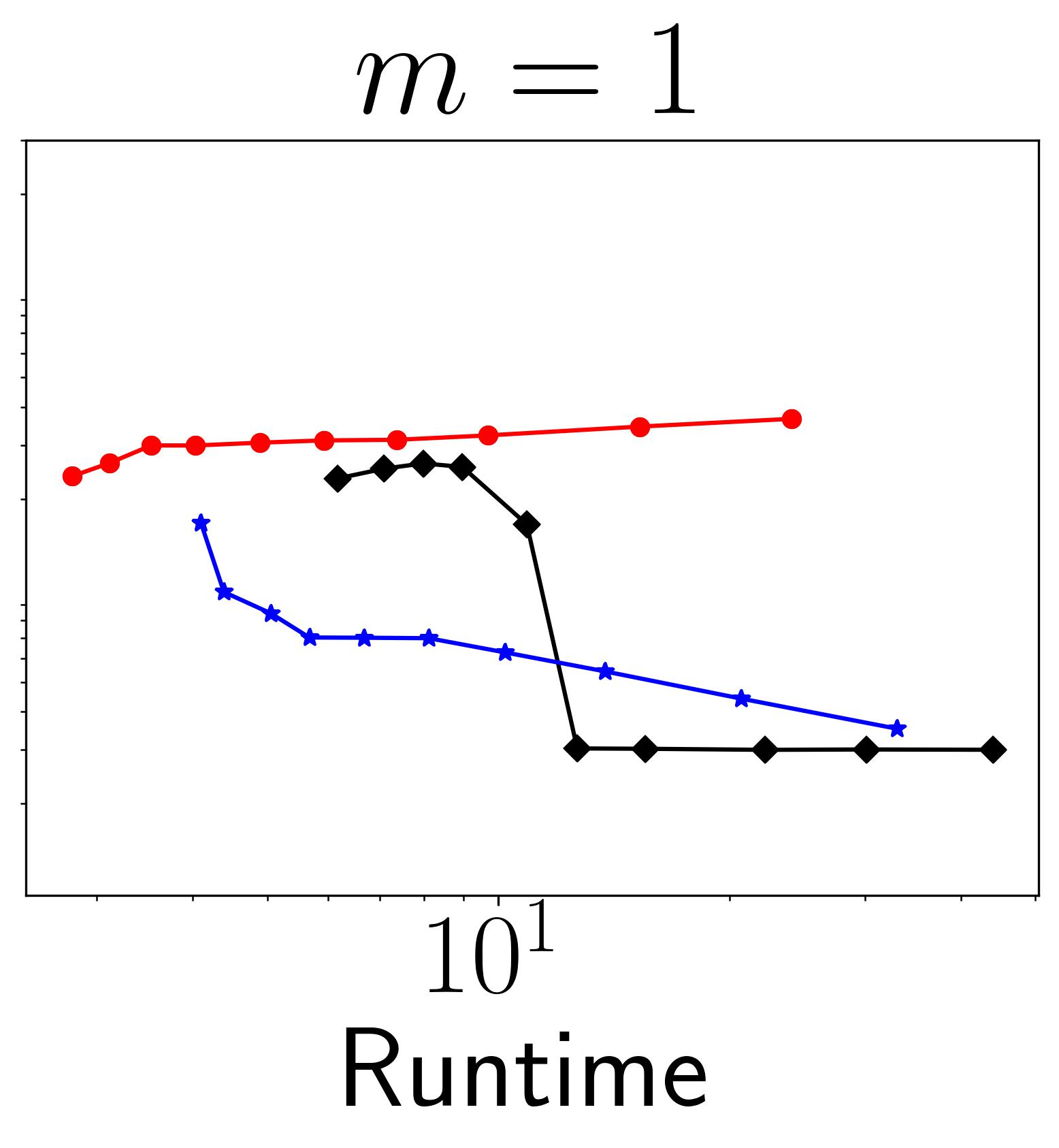}
        \includegraphics[width=.225\linewidth]{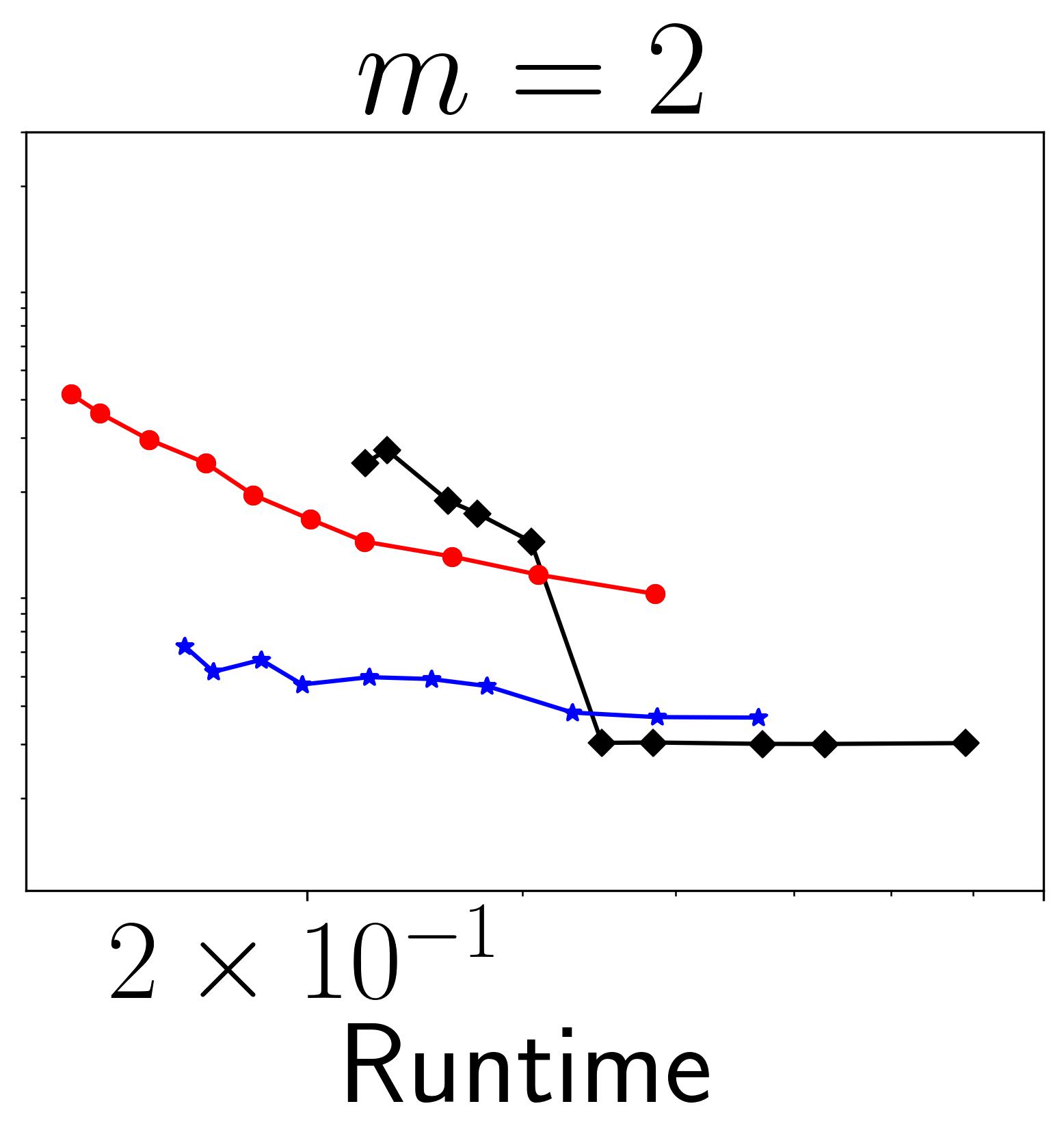}
        \includegraphics[width=.225\linewidth]{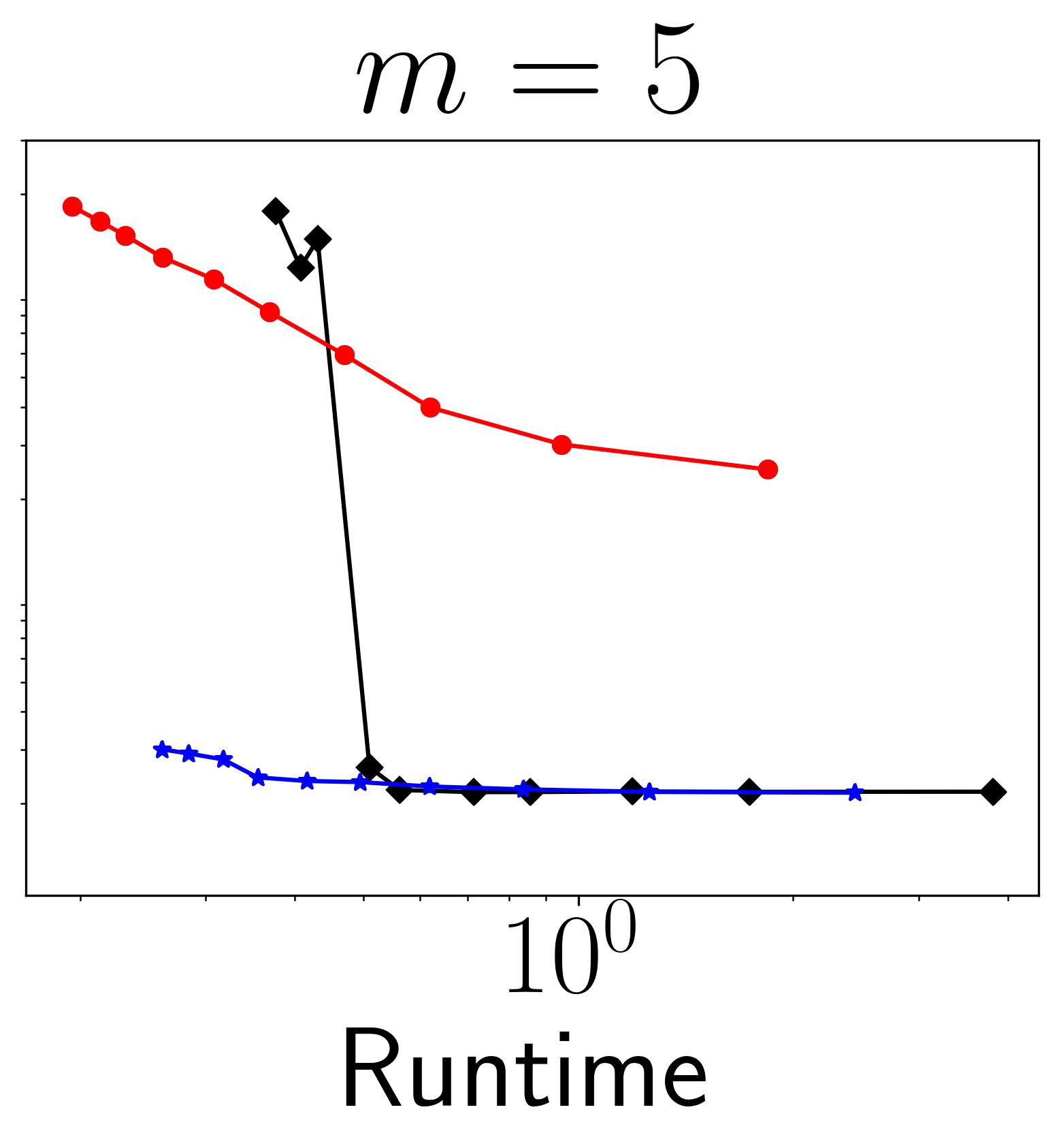}

        \raggedright \includegraphics[width = .4 \linewidth]{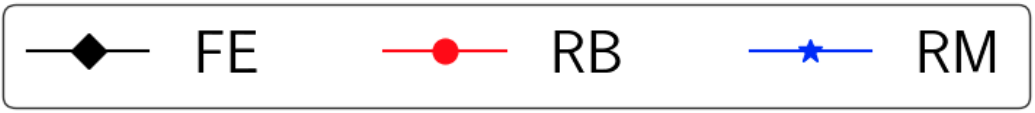}

        \caption{Error between exact continuum solution and particle solution. We simulate the nonlinear diffusion equation for different choices of $m$ and external potential $V(x)$.}
        \label{fig2}
    \end{figure}
      
  \subsection{Diffusion Equations: Longtime Behavior}
    \label{sec:diffusion_longtime}
    In this section, we investigate the longtime behavior of the particle system solved using the FE, RB, and RM methods.
    As mentioned in the introduction, for certain choices of $f$, $v$, and $\rho_0$, the solution of equation \eqref{eqn:PDE} converges exponentially quickly to some longtime steady state.
    In this section, we determine which of these three methods preserve this longtime behavior.
    
    In Figure \ref{fig3}, we repeat the simulations in Figure \ref{fig2}, except we run the simulation to a larger final time, and consider the 2-Wasserstein distance between the particle solution at final time $T$ and the exact continuum longtime behavior.
    In the first row, we consider a quadratic external potential, and in the second row, we consider a ``double well" external potential.  
    In all simulations, we take $h = 0.01$. 
    For $m \geq 1$, we consider the average error and runtime over ten different random seeds. 
    For $m < 1$, we only consider one random seed.
    

    We see the greatest contrast between figures \ref{fig2} and \ref{fig3} in the case of the RB method.
    In Figure \ref{fig2}, when comparing the particle solution and exact continuum solution at some relatively small final time $T$, the FE method outperforms the RB method.
    In Figure \ref{fig3}, when comparing the particle solution and theoretic longtime solution at some relatively large final time $T$, the RB method outperforms the FE method in certain cases.
    This phenomenon is most apparent in the case $m = 2$ and $m = 5$.
    Despite these differences, we see that the RM method again exhibits an excellent tradeoff between accuracy and efficiency.

    \begin{figure}[htbp]
        \centering
        
        \raggedright \boxed{V(x) = \frac{x^2}{2(m + 1)}} \\
        \centering
        \includegraphics[width=.29\linewidth]{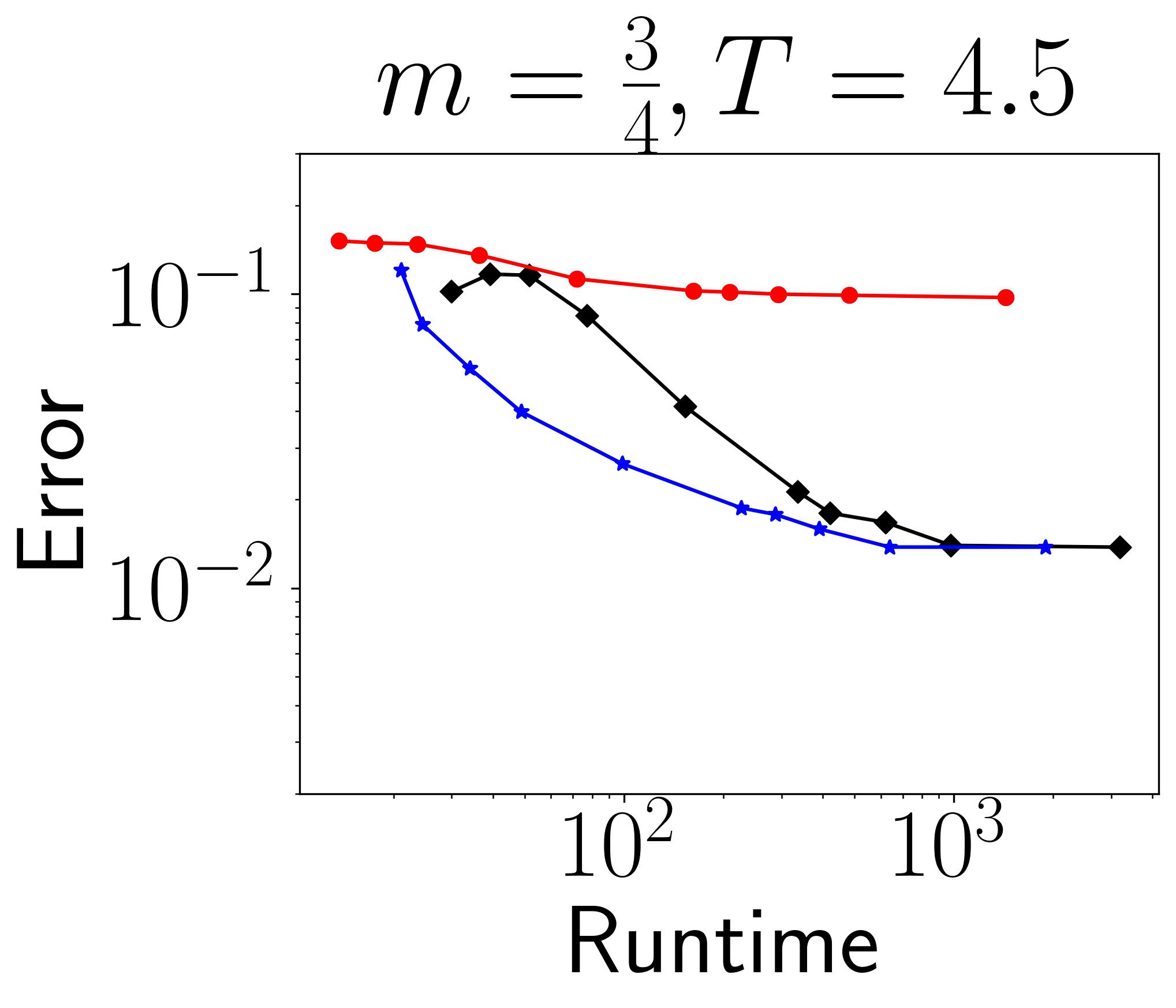}
        \includegraphics[width=.22\linewidth]{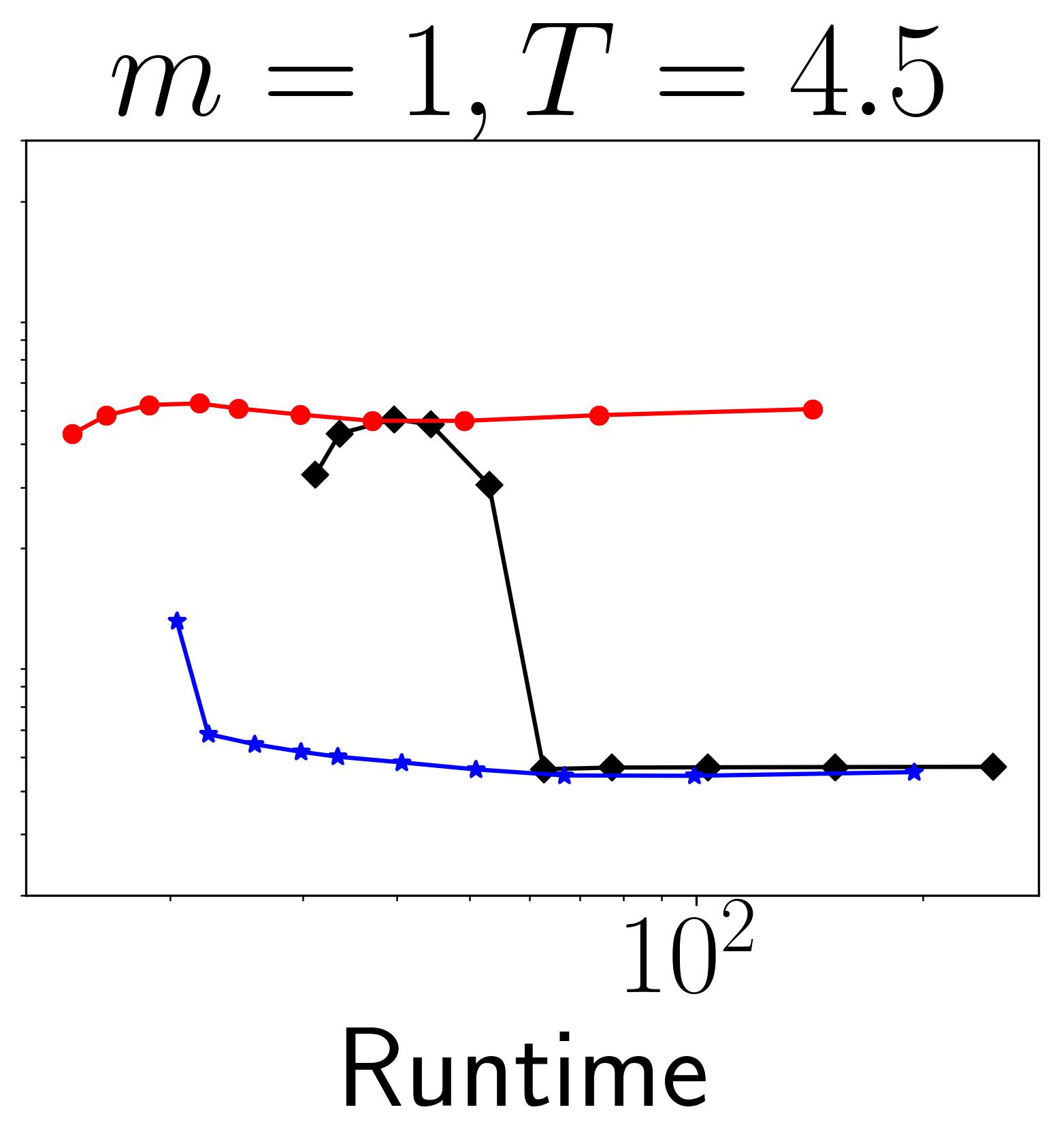}
        \includegraphics[width=.23\linewidth]{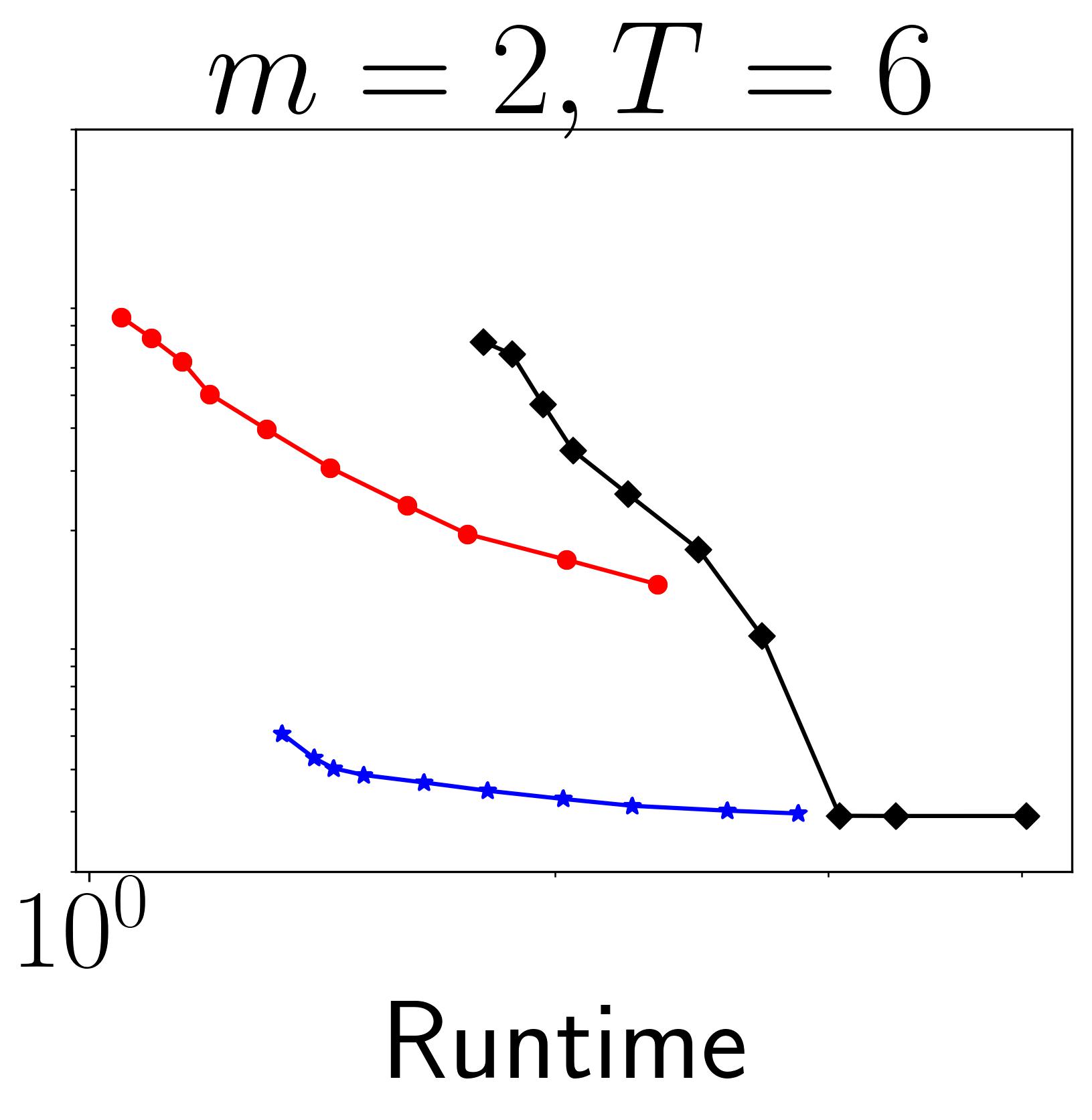}
        \includegraphics[width=.22\linewidth]{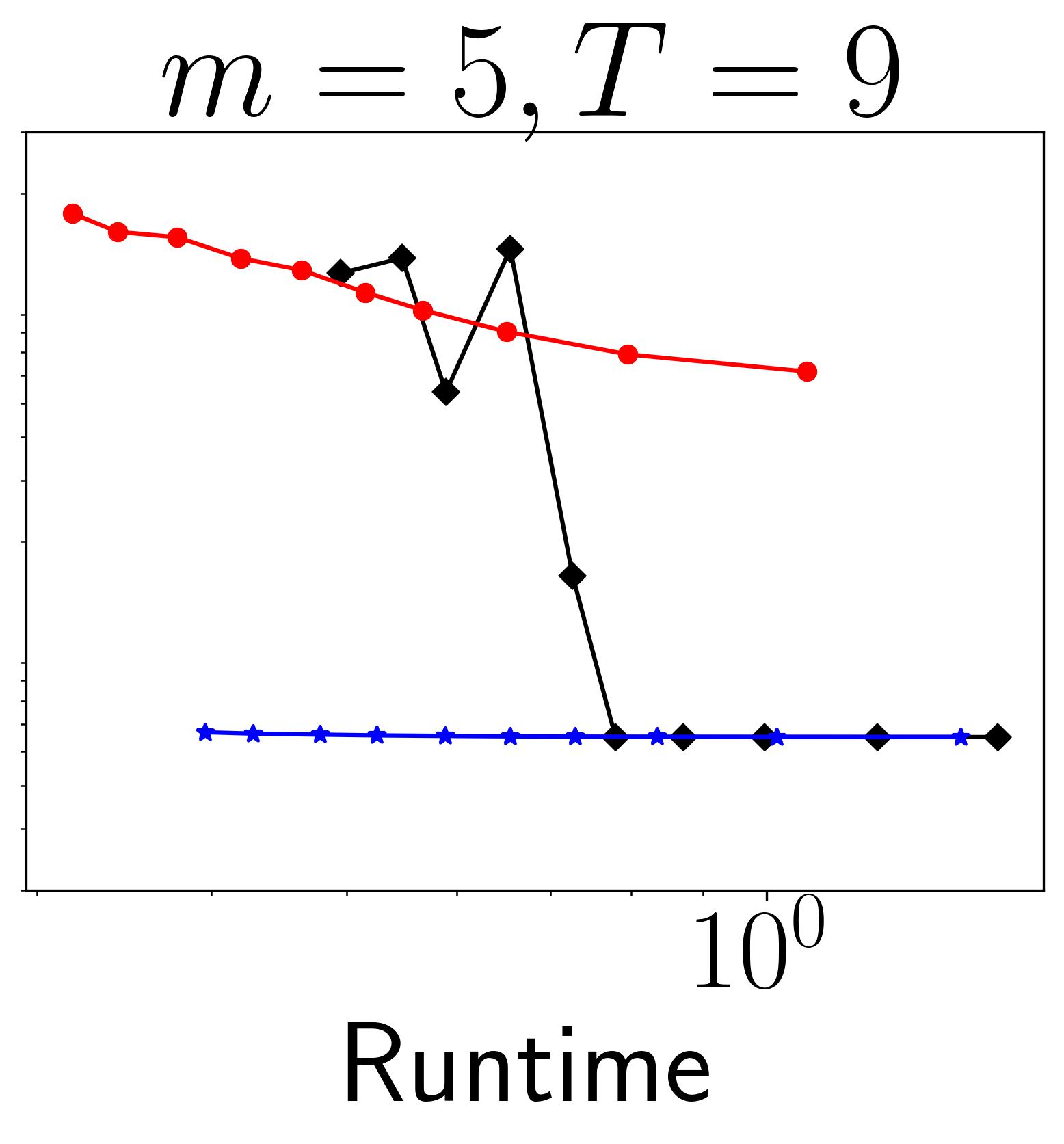}

        \raggedright \boxed{V(x) = (1 - x)^2(1 + x)^2} \\
        \centering
        \includegraphics[width=.29\linewidth]{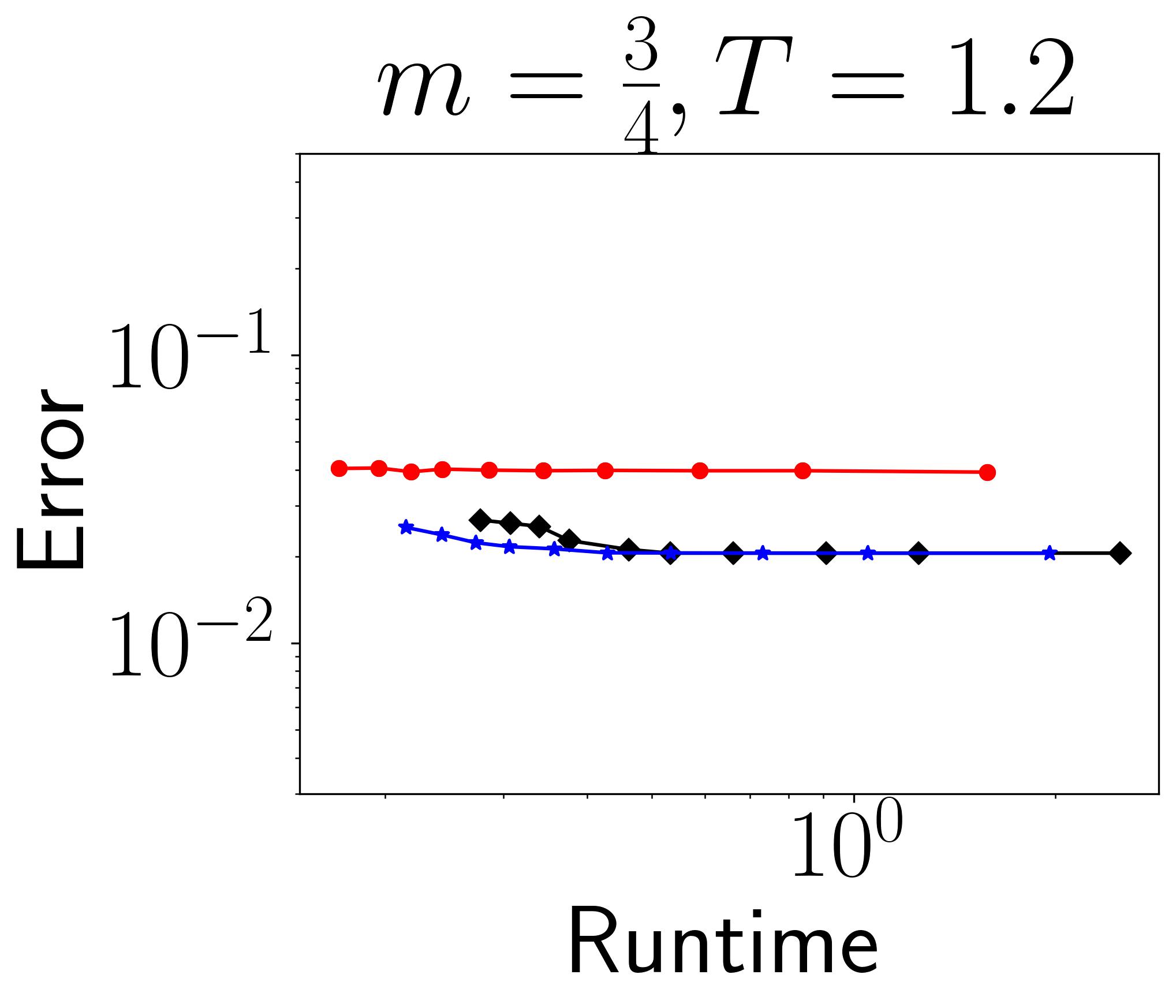}
        \includegraphics[width=.22\linewidth]{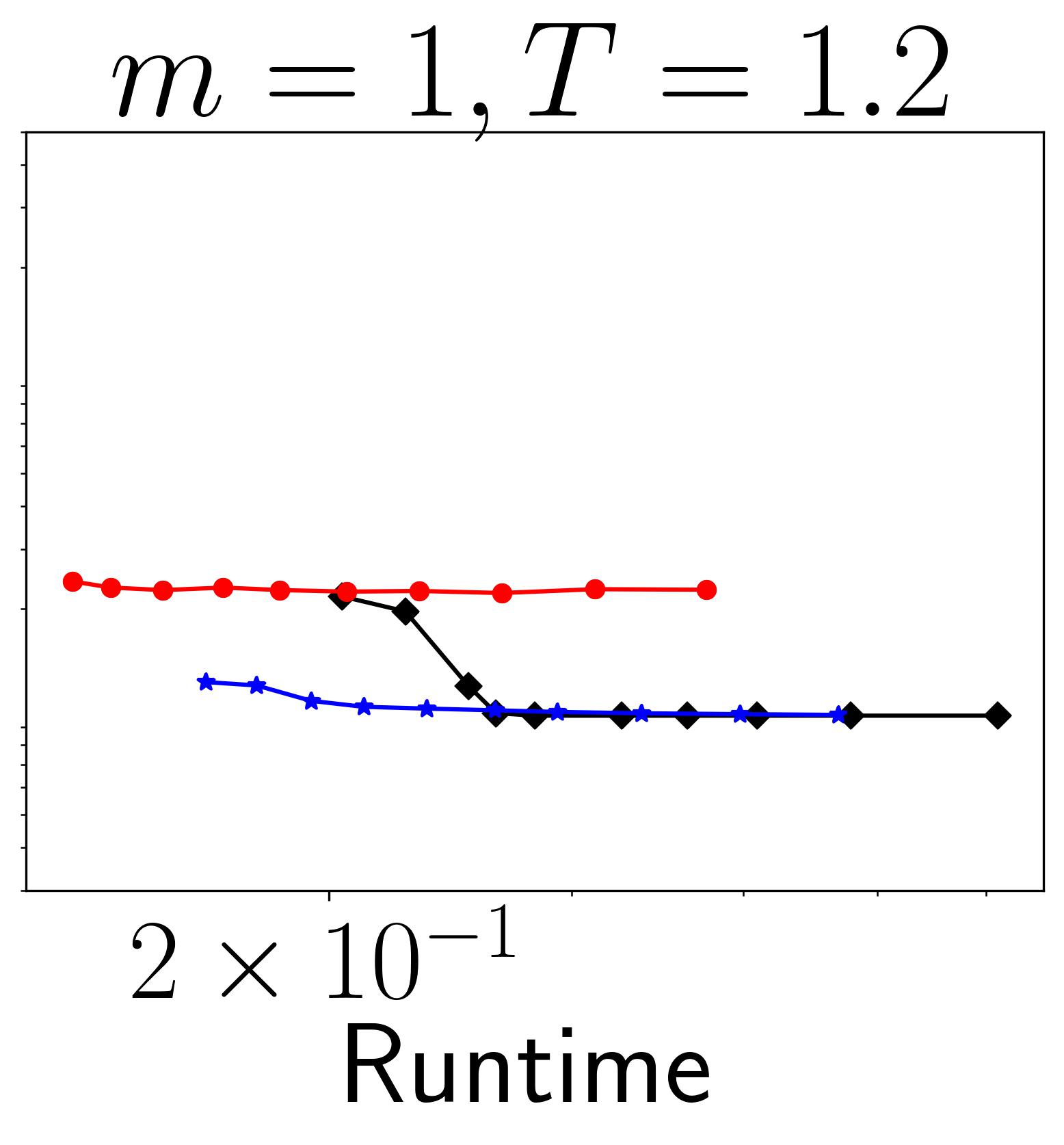}
        \includegraphics[width=.22\linewidth]{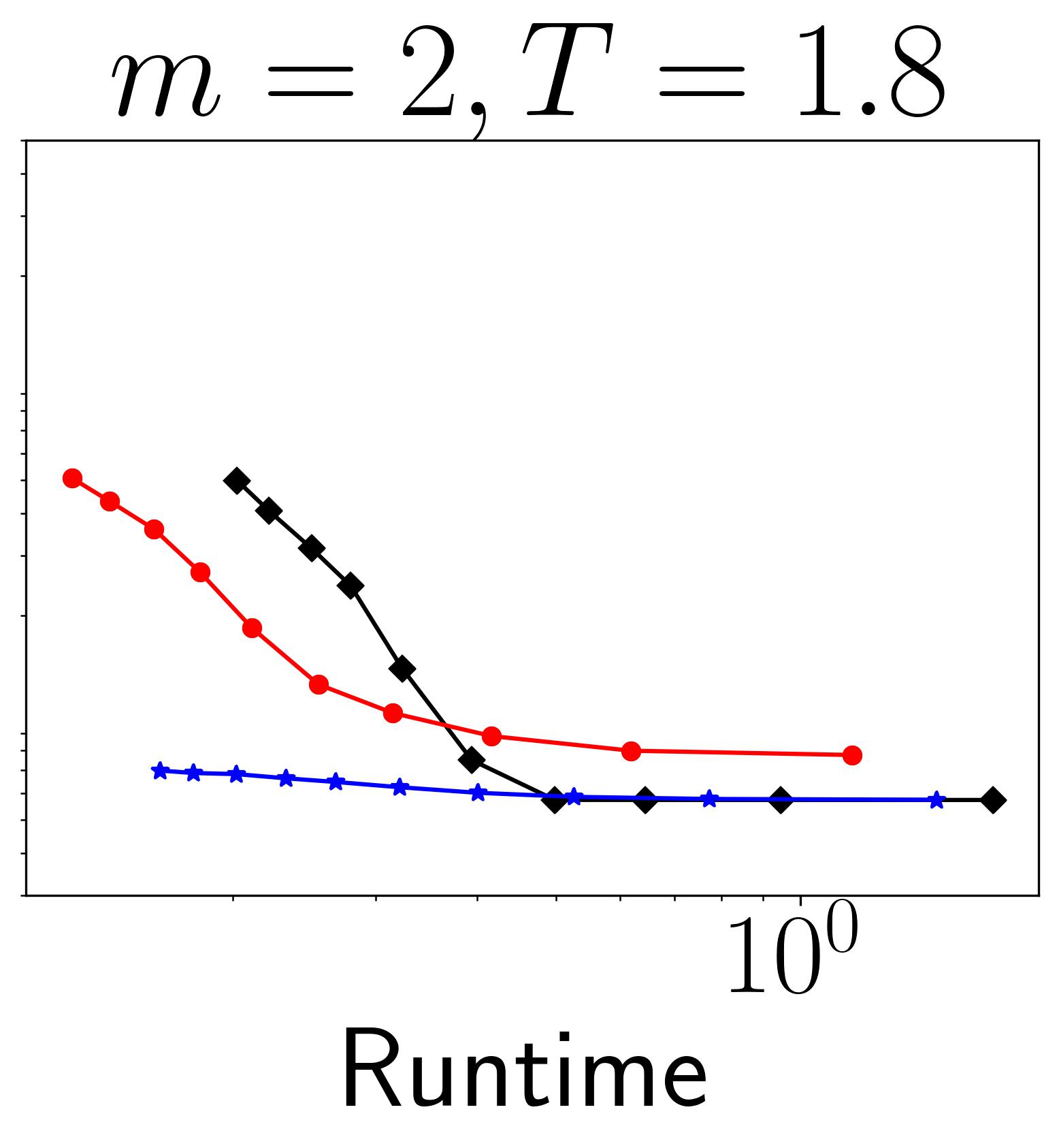}
        \includegraphics[width=.23\linewidth]{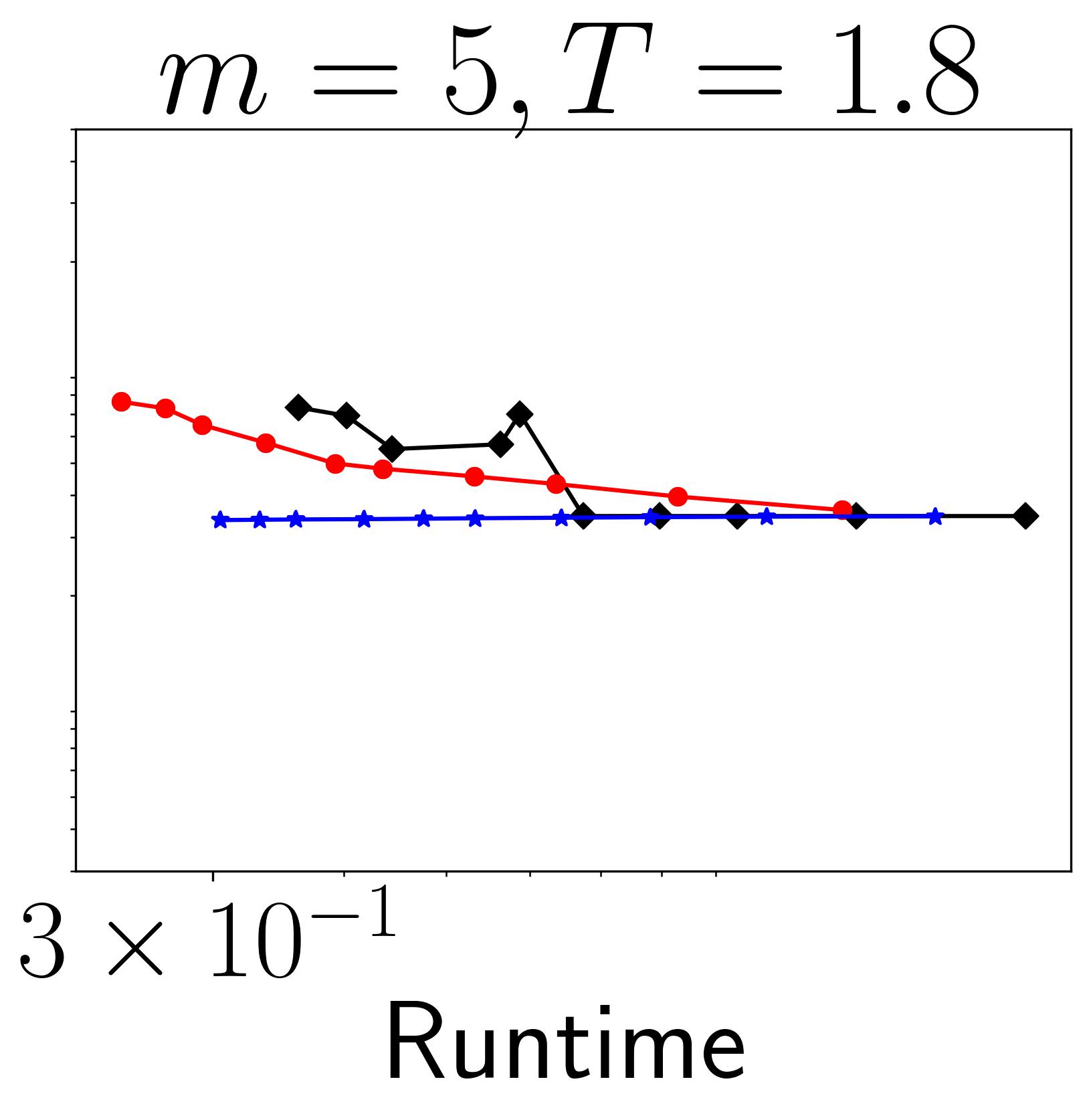}

        \raggedright \includegraphics[width = .4 \linewidth]{legend.png}
        
        \caption{Error between theoretic longtime behavior and particle solution at some sufficiently large final time $T$. We simulate the nonlinear diffusion equation for different choices of $m$ and external potential $V(x)$.}
        \label{fig3}
    \end{figure}
    
 \subsection{Rate of Convergence to Longtime Behavior}
    \label{sec:convergence_rate}
    In the previous section, we verify that the particle solution solved using the RM method converges to the theoretic steady state.
    Now, we demonstrate that the RM method preserves the exponential rate of convergence under the 2-Wasserstein metric to the steady state.
    
    In Figure \ref{fig7}, we simulate the porous medium equation with $m = 2$ and external velocity field $v(x) = \frac{x}{3}$ and compare the particle solution at various times with the theoretic longtime behavior.
    We set $h = 0.01$.
    For two time steps, each corresponding to a different panel in the figure, we calculate the solutions up to time $T = 9$.

    We note that the RM method achieves better accuracy by an order of magnitude.
    If we use the FE method, we need to choose a much finer time step to achieve this accuracy.
    Even in the case of this finer time step, the RB method does not achieve this level of accuracy.


    \begin{figure}
        \centering
        \includegraphics[width=0.4\linewidth]{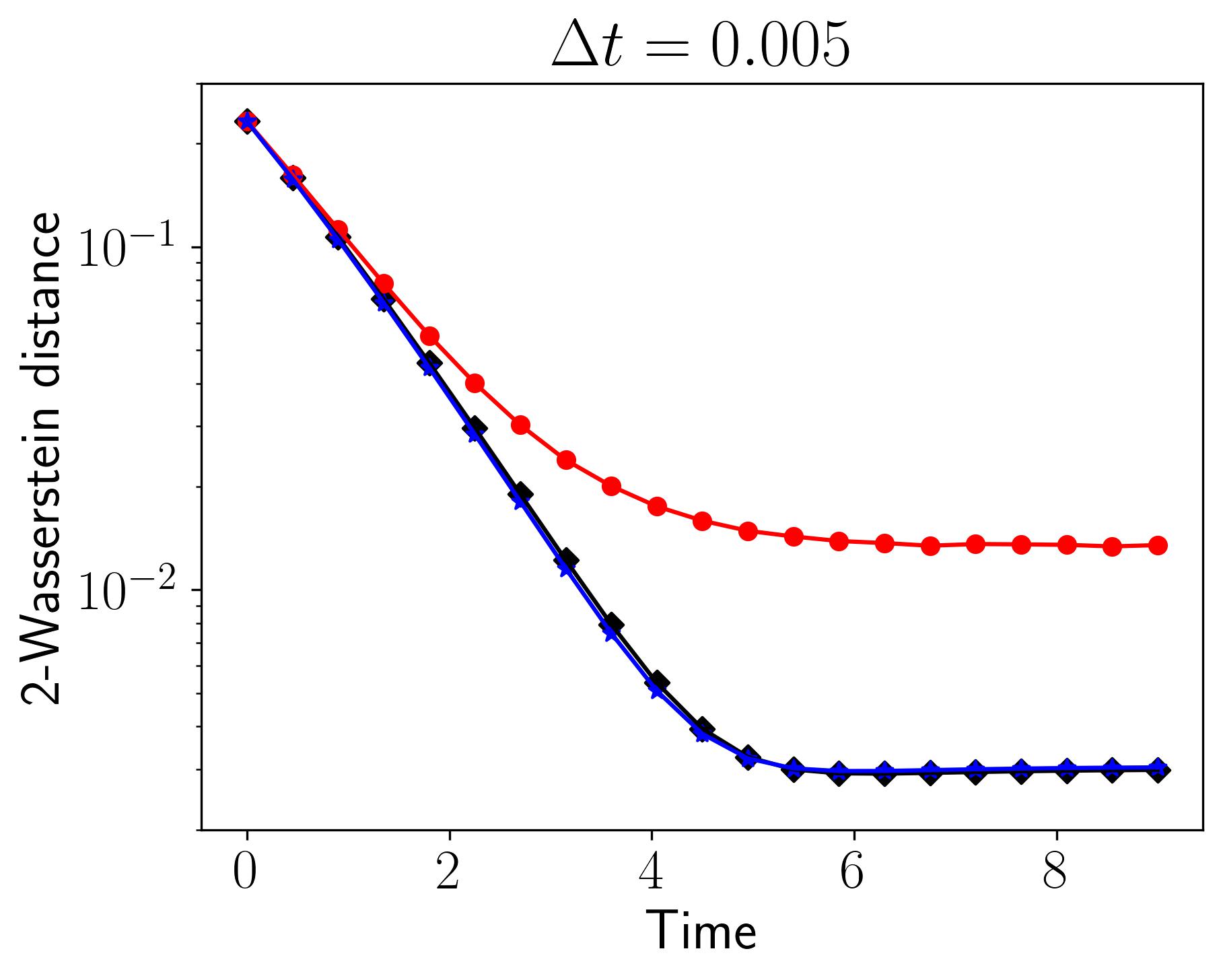}
        \includegraphics[width=0.3425\linewidth]{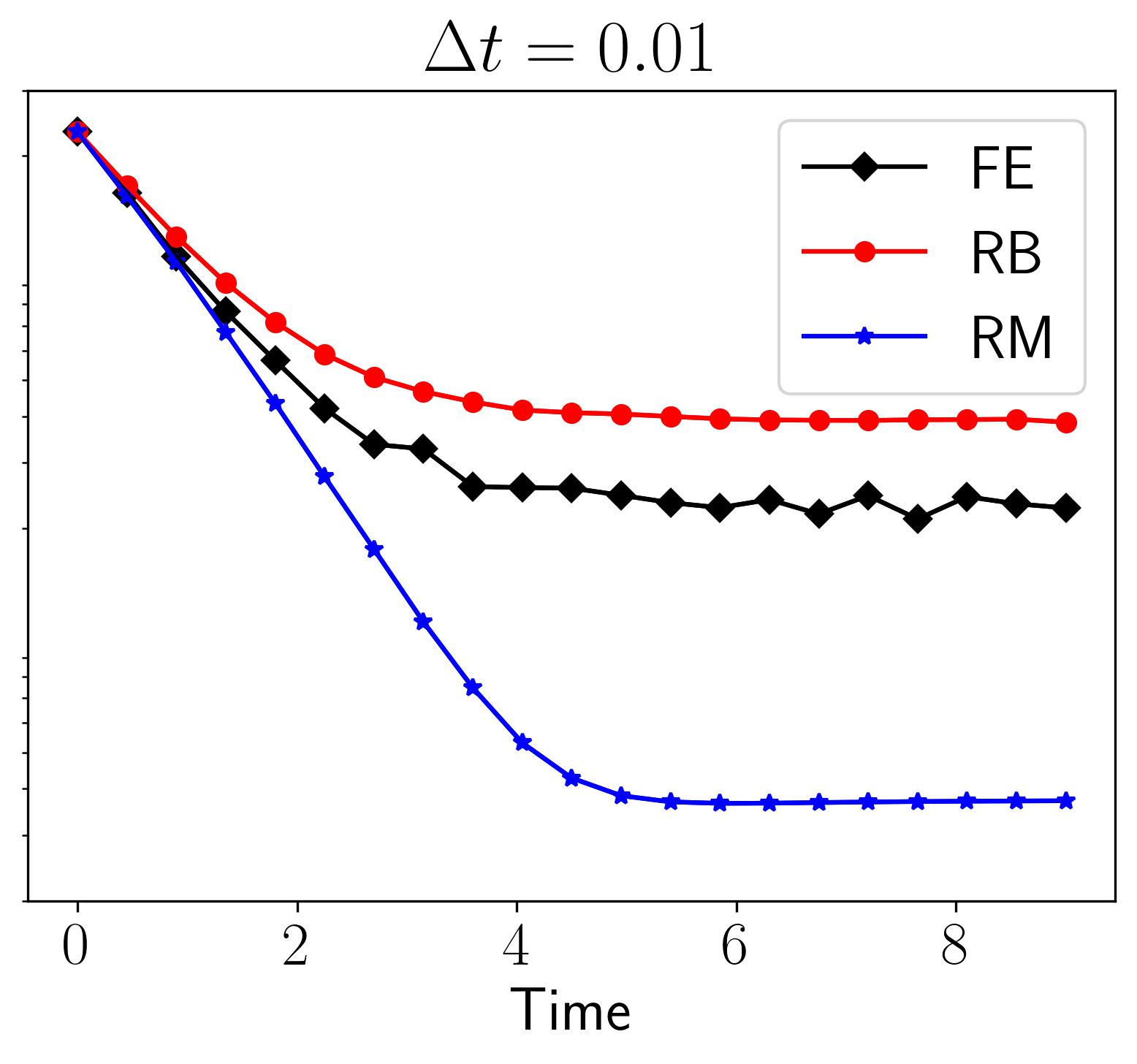}
        \caption{Exponential convergence to the theoretic longtime behavior, in the case of the $d = 1$ porous medium equation with $m = 2$ and external velocity field $v(x) = \frac{x}{3}$. At certain times $t$, we approximate the 2-Wasserstein distance between the particle solution at time $t$ and the theoretic longtime behavior.}
        \label{fig7}
    \end{figure}

 \subsection{The Height-Constraint Case}
    \label{sec:height_constraint}
    We now go beyond the variants of the porous medium / fast diffusion equation considered in the previous sections to demonstrate the benefits of the RM method when the internal energy density $f$ is not smooth.
    We begin by considering the case of height-constrained transport, as in equation $\eqref{eqn12}$.
    Formally, the solution to the PDE moves according to the transport equation, up to the restriction that the density of the particle solution can never exceed height one.
    To the best of the authors' knowledge, the only alternative method proposed for numerically simulating height-constrained dynamics is the back and forth method for Wasserstein gradient flows \cite{jacobs2020backandforthmethodwassersteingradient}.
    Our method provides a meshfree alternative that likewise succeeds in capturing key features of solutions, including the fact that solutions can still move on regions with zero velocity when ``pushed" by other components of the solution.
    
    In order to simulate the height-constraint case, we approximate $f$ by a suitable $f_{\eps}$, as discussed in Section \ref{sec:motivation}.
    The internal energy density $f$ for the height-constraint case is the $m \rightarrow \infty$ limit of the internal energy density given by equation \eqref{eqn1}, as in previous work by Craig and Topaloglu \cite{CRAIG2020239}.
    Thus, in this case, $f_{\eps}$ is given by equation $\eqref{eqn1}$ with $m = \frac{1}{\eps}$; in Figure \ref{fig19}, $\eps = \frac{1}{400}$ and in Figure \ref{fig4}, $\eps = \frac{1}{100}$.
    
    In Figure \ref{fig19}, the external velocity field is 
    \begin{equation}
        v(x) = \begin{cases}
            -1 & x \leq 0 \\
            0 & \text{otherwise.}
        \end{cases}
        \label{eqn17}
    \end{equation}
    We set $h = 0.025$ and solve the system using the RM method with $\Delta t = 10^{-6}$ up to final time $T = 1.5$.
    Our initial condition is two connected components.
    Up to time $t_{\text{crit}} := 0.5$, the left component moves to the right with speed one.
    After time $t_{\text{crit}}$, we can think of this left component ``pushing" the right component farther to the right. 
    The fact that the density can continue to evolve on regions where the velocity is zero is a key feature of the height-constraint equation \eqref{eqn12}.
    This was a motivation of recent work by Masson et al. \cite{masson2026stifflimitnonhomogeneousconservation}, which proposes an alternative model of height-constrained transport where the density does not move on regions where the velocity is zero and can instead block the motion of density on regions with nonzero velocity.

    \begin{figure}
        \centering
        \includegraphics[width = .25 \linewidth]{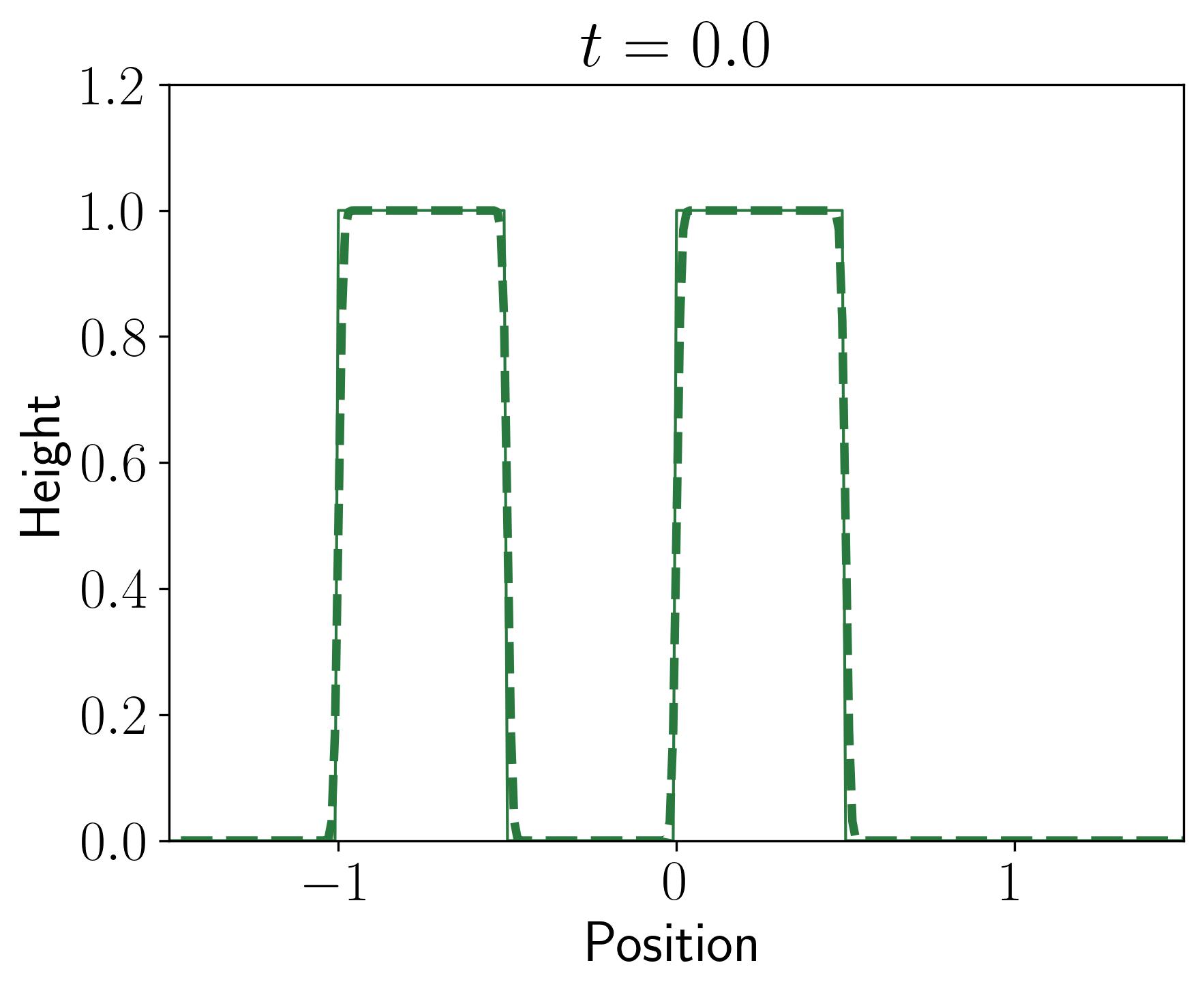}
        \includegraphics[width = .22 \linewidth]{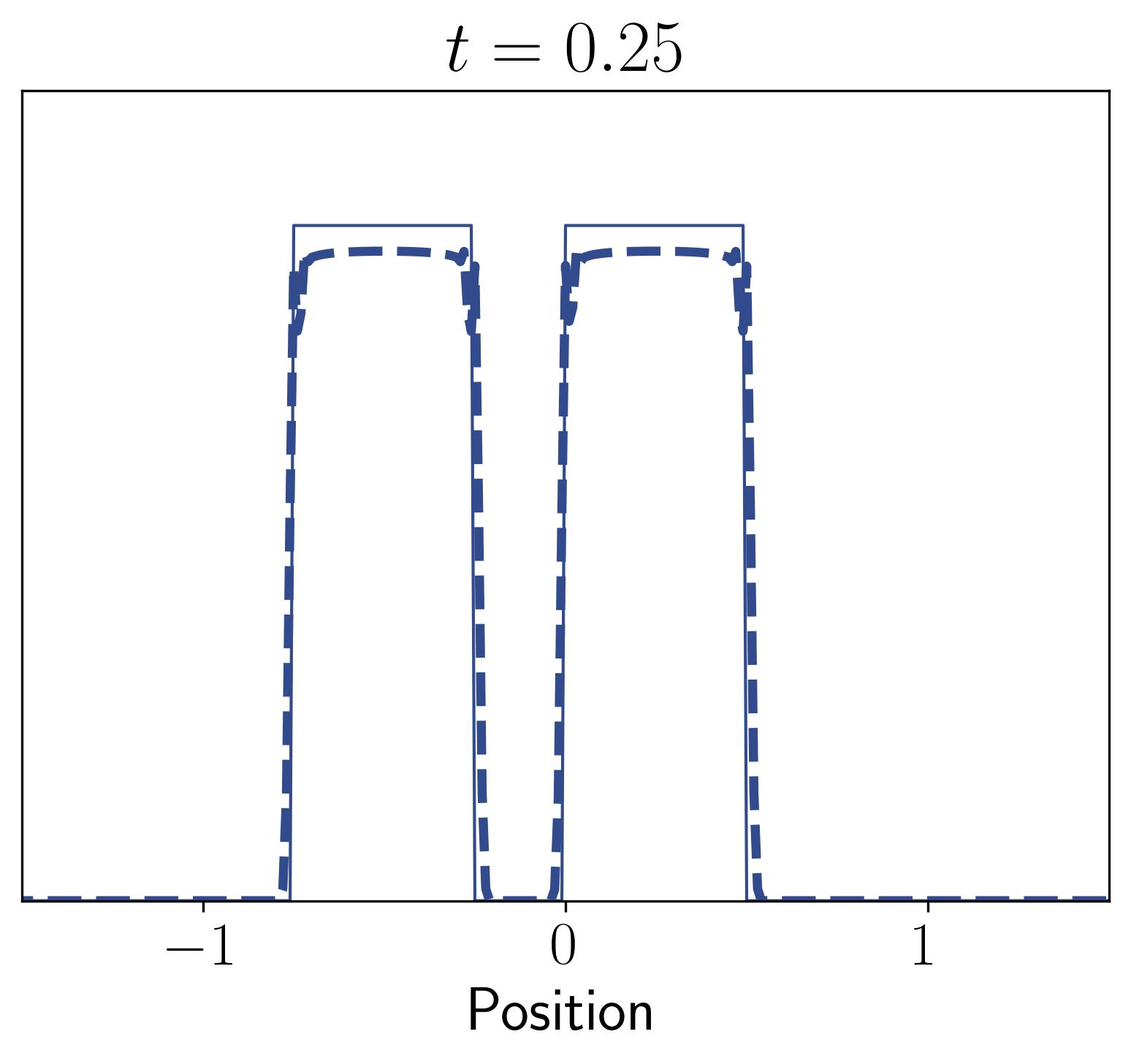}
        \includegraphics[width = .22 \linewidth]{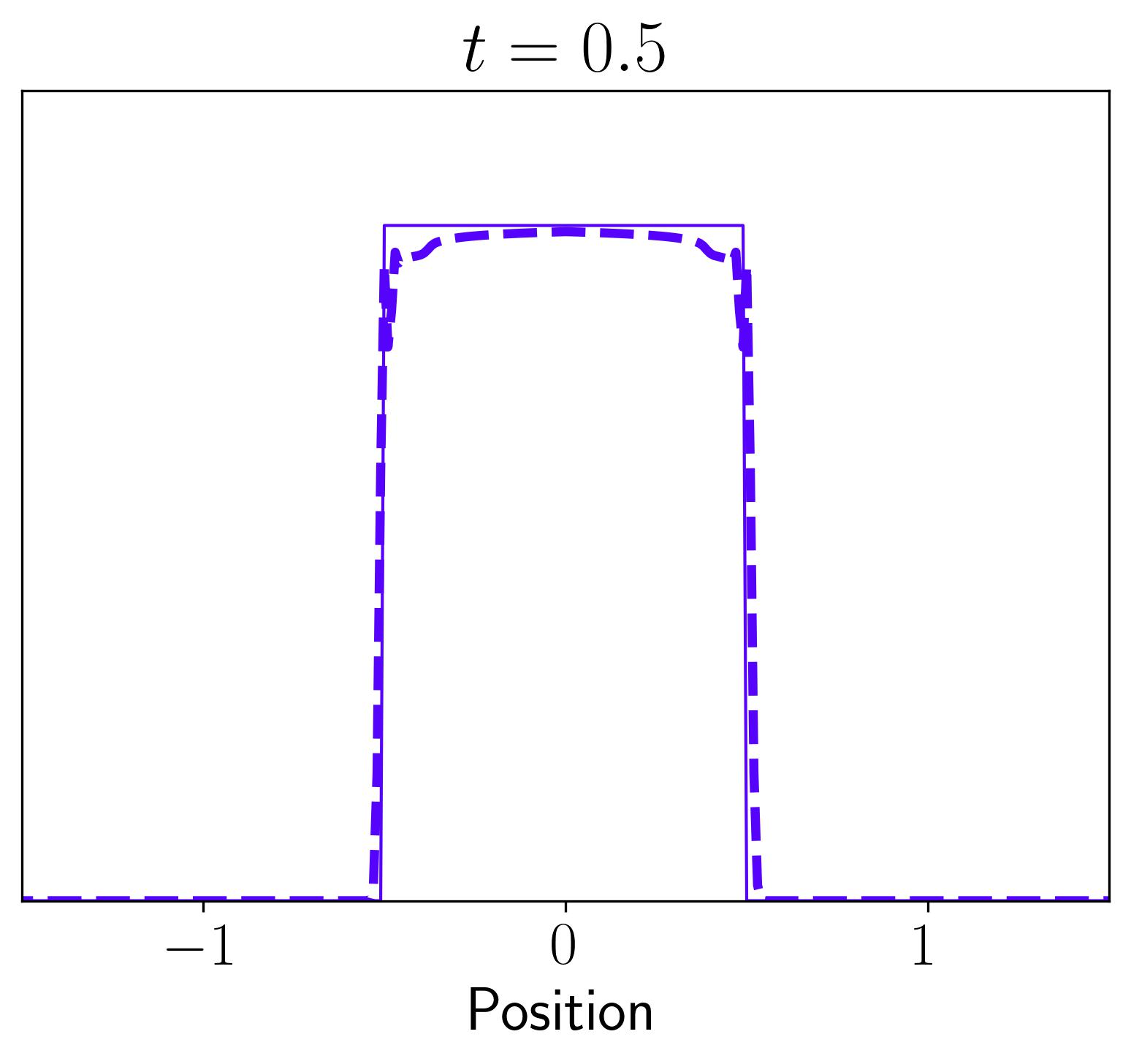}
        \includegraphics[width = .22 \linewidth]{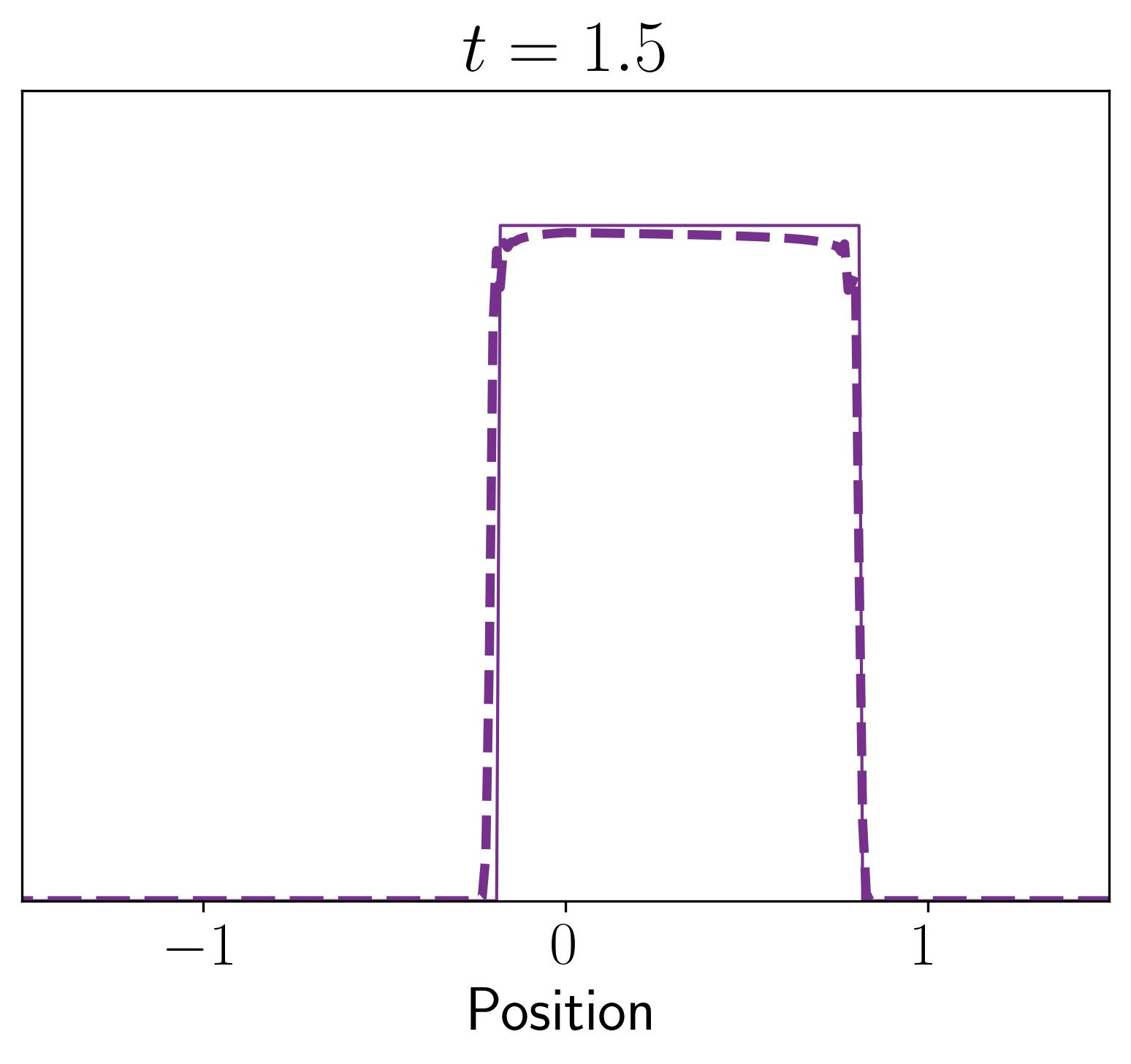}
        \caption{Particle solution to the $d = 1$ height-constraint case with external velocity field given by equation \eqref{eqn17}, solved using the RM method.
        At various times $t$, we display the particle solution (dashed line) and exact continuum solution (solid line).}
        \label{fig19}
    \end{figure}

    In Figure $\ref{fig4}$, the external velocity field is $v(x) = x$.
    We set $h = 0.005$. 
    Using the FE, RB, and RM methods, we solve the ODE system up to final time $T = 1.5$. 
    Notice that the RB method is not a viable numerical method for this type of regime; at a certain point, particles ``shoot out" toward infinity, leading to an extremely inaccurate solution.
    In fact, these accuracies persist even though we project particles that leave the interval $[-3, 3]$ back into the domain.
    On the other hand, while the RM method requires greater computation time than the RB method, it preserves the accuracy of the system.
    The particle solution found using the RM method is virtually identical to the particle solution found using the FE method.
    
    \begin{figure}
        \centering

        \includegraphics[width = .3 \linewidth]{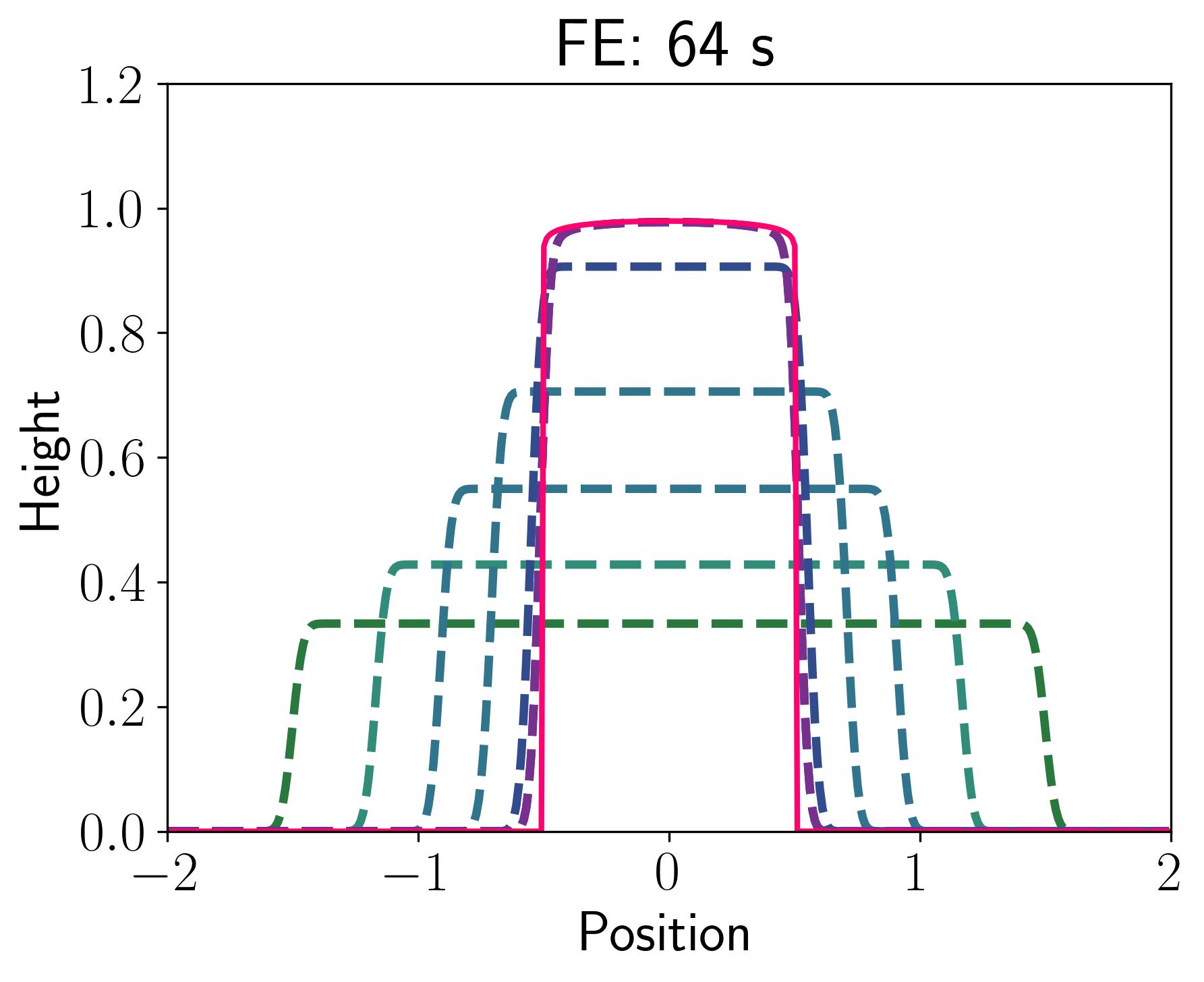}
        \includegraphics[width = .272 \linewidth]{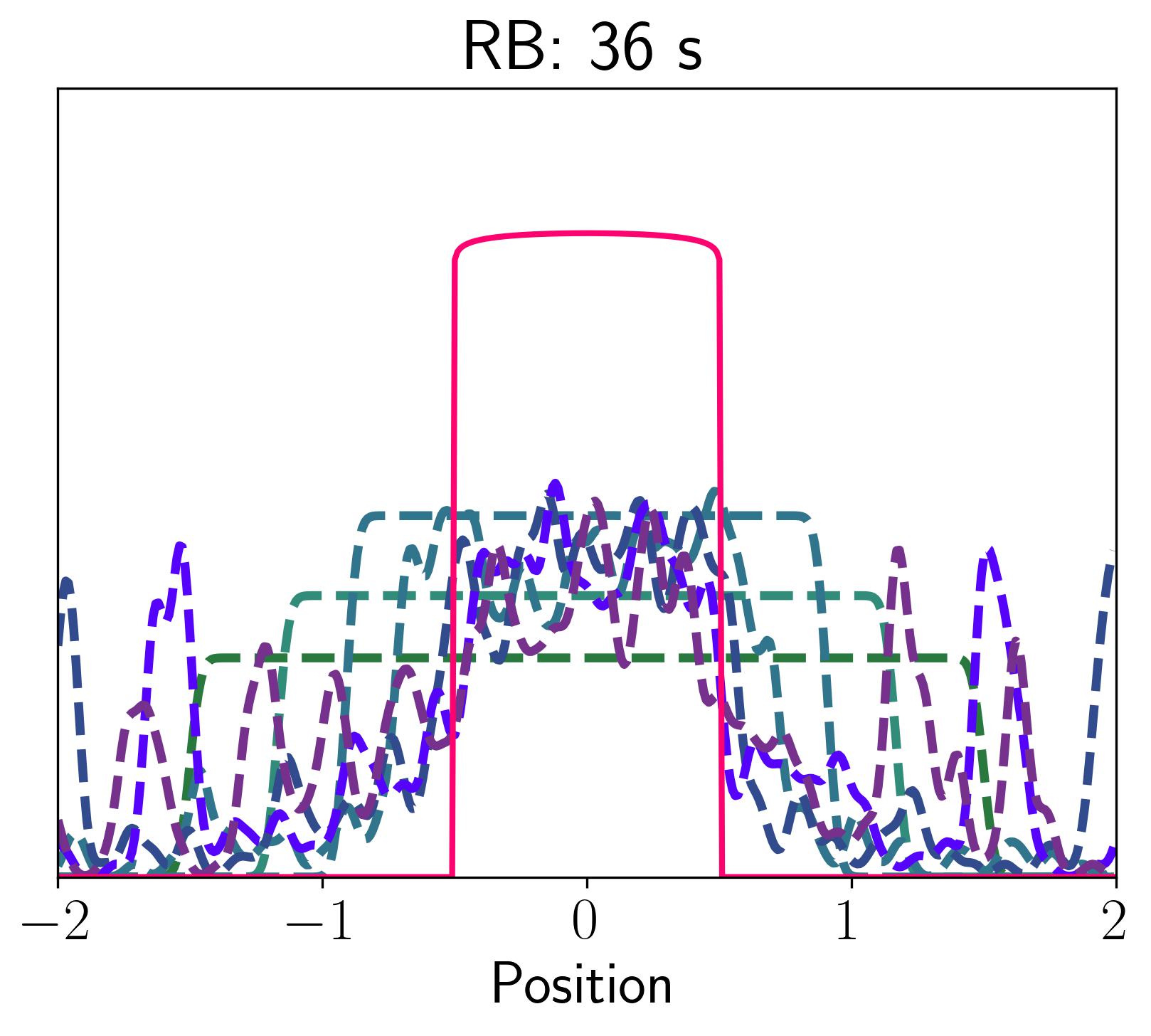}
        \includegraphics[width = .3725 \linewidth]{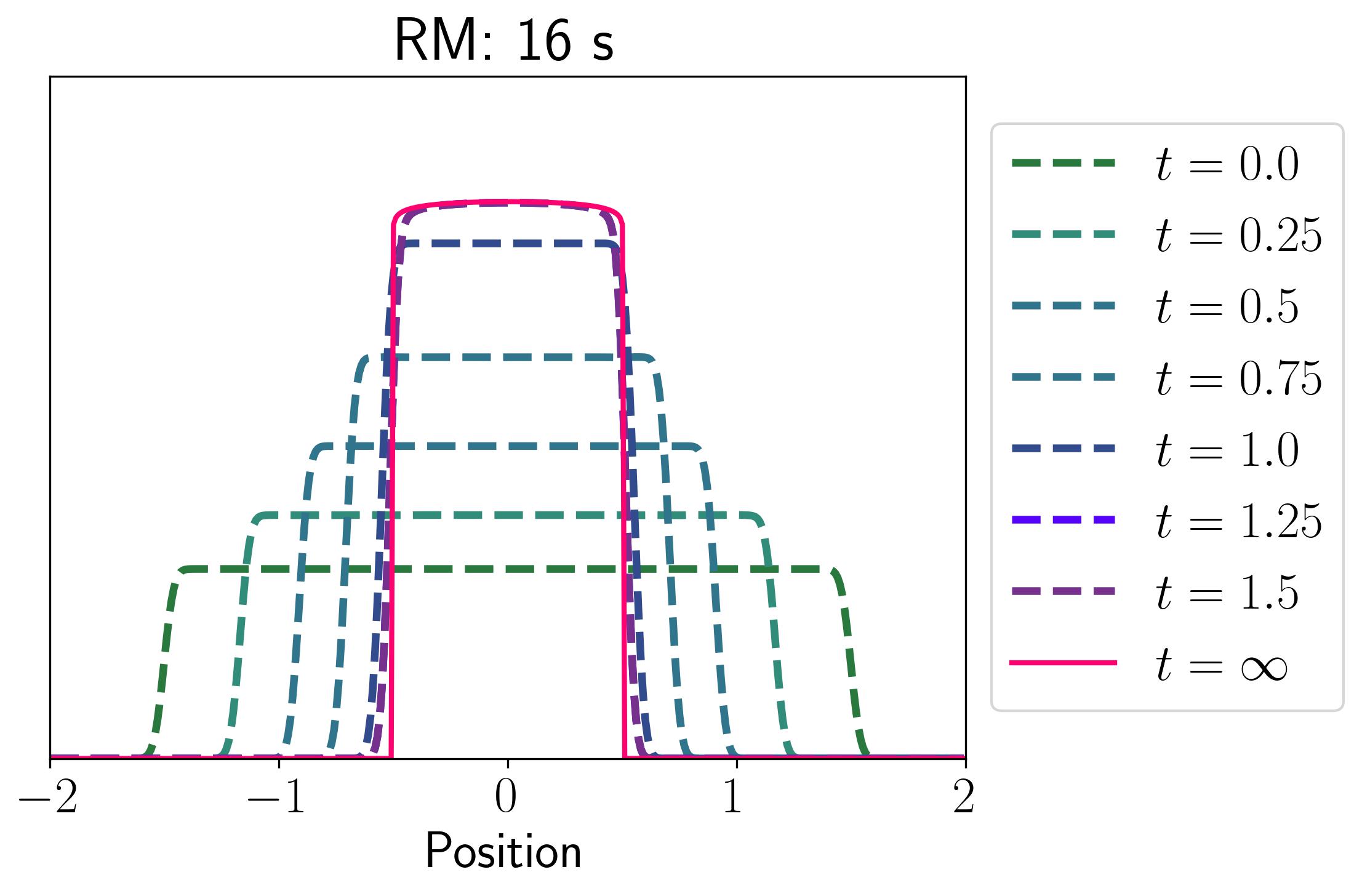}

        \caption{Particle solutions to the $d = 1$ height-constraint case with external velocity field $v(x) = x$, approximated via slow diffusion with $m = 100$.
        We solve the system via the FE, RB, and RM methods up to final time $T = 1.5$. 
        We display the runtime of each method above each plot.
        The dashed lines correspond to the particle solution at different times $t$, while the solid pink line corresponds to the theoretic longtime behavior.}
        \label{fig4}
    \end{figure}

    In Figure \ref{fig5}, we repeat the numerical setup as in Figure \ref{fig4}, varying the time step sizes:
    \[ \Delta t \in \left \{ 10^{-4} \left ( \frac{n}{2} \right ): 2 \leq n \leq 11 ,  n \in \Z \right \} \] 
    We take the average error and runtime over ten different random seeds.
    We see that the FE method displays a significant drop in error once the time step becomes suitably fine.
    The RM method also demonstrates this drop, but at a coarser time step.
    This phenomenon is consistent with the results of Sections \ref{sec:diffusion_exact} and \ref{sec:diffusion_longtime}, where we observed better and better behavior of the RM method in the slow diffusion limit.
    The height-constraint case ($m \rightarrow \infty$) is the extreme case of this regime.
    For all time steps simulated, the RB method is relatively inaccurate.
 
    \begin{figure}
        \centering
        \includegraphics[width=0.5\linewidth]{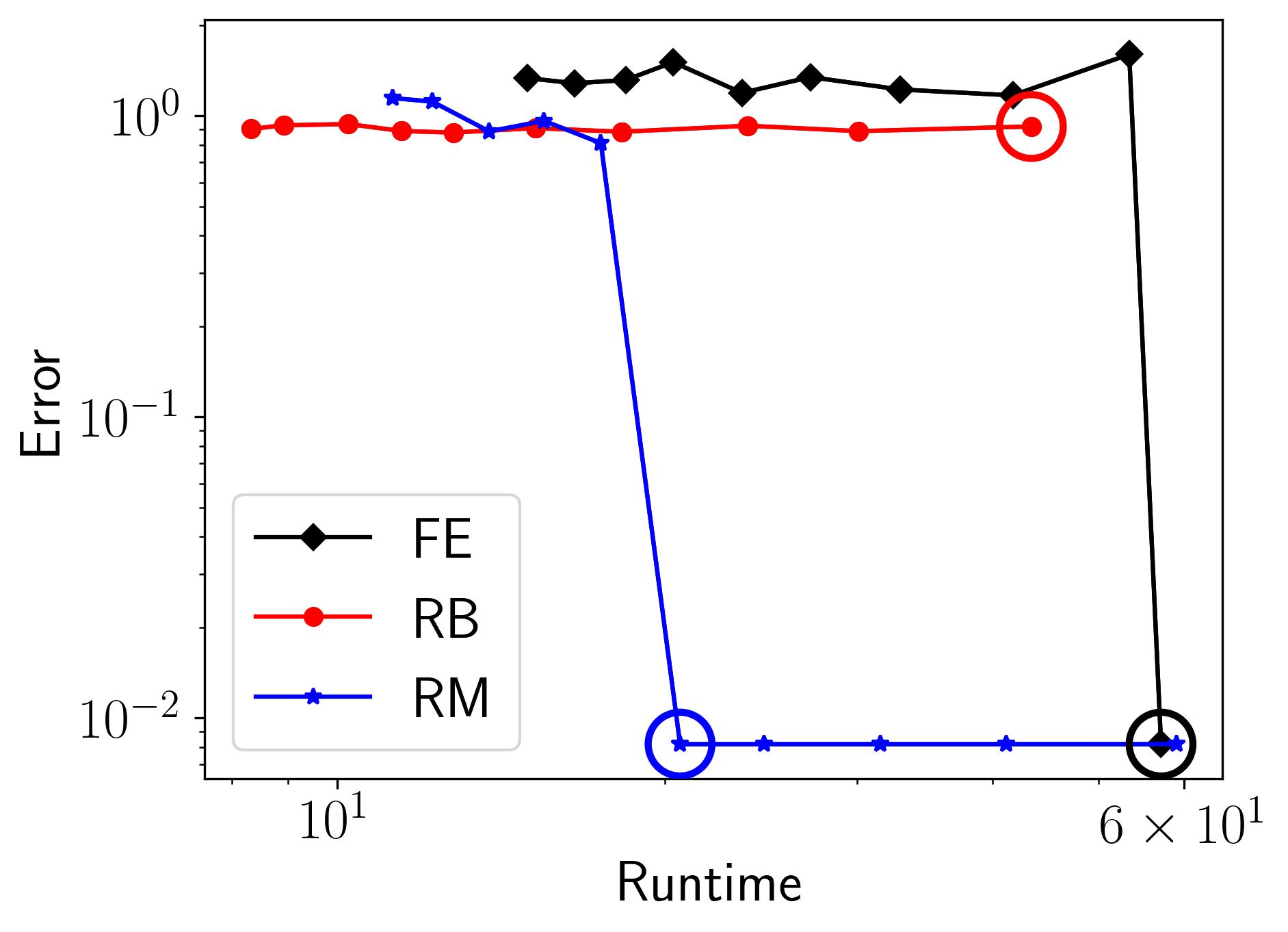}
        \caption{The $d = 1$ height-constraint case with external velocity field $v(x) = x$. For the FE, RB, and RM methods, and for various time steps, we approximate the 2-Wasserstein distance between the particle solution at $T = 1.5$ and the theoretic longtime behavior.
        We circle the simulations corresponding to the particle solutions in Figure \ref{fig4}.}
        \label{fig5}
    \end{figure}
  
 \subsection{The Sandpile Model}
    \label{sec:sandpile}
    As in the section above, we simulate another variant of equation \eqref{eqn:PDE} with nonsmooth internal energy density $f$ to demonstrate the wide array of diffusion equations that the RM method can effectively handle.
    We now consider the sandpile dynamics PDE; see equation \eqref{eqn18}.
    Heuristically, solutions of this PDE behave as follows: if the density exceeds the critical height threshold $r_c > 0$ on a given region, it diffuses on that region, but there is no diffusion when the density is beneath the critical threshold.
    To the best of the authors' knowledge, the only existing method proposed for the sandpile dynamics is a finite difference method, which uses a smoothed step function to capture the transition between the diffusive and non-diffusive regimes \cite{Bantay1992}. 
    Our method provides a meshfree alternative that likewise succeeds in capturing key features of solutions, including the growth of the critical region.
    We regularize $f$ via $f_{\eps}$ with first derivative given by
    \begin{equation*}
        f'_{\eps}(s)
        = \begin{cases}
            0 & s \leq r_c - \eps \\
            S \left ( \frac{s - (r_c - \eps)}{2 \eps} \right ) \left [ 1 + \ln \left ( \frac{s}{r_c}\right ) \right ] & r_c - \eps < s < r_c + \eps \\
            1 + \ln \left ( \frac{s}{r_c} \right ) & s \geq r_c + \eps , 
        \end{cases}
    \end{equation*}
    where $S(t) = 6t^5 - 15t^4 + 10t^3$.
        
    In the top row of Figure \ref{fig18}, we plot the exact continuum solution and the particle solution solved using the FE method up to final time $T = 0.05$. 
    On the left, we set $r_c = 0.25$; in order to obtain an accurate solution, we set $h = 0.001$ and use the FE method with time step $\Delta t = 10^{-6}$.
    On the right, we consider the less computationally intensive case that $r_c = 0.1$; we set $h = 0.005$ and use the FE method with time step $\Delta t = 10^{-4}$.

    In the bottom row of Figure \ref{fig18}, we compute the particle solution, using the FE, RB, and RM methods.
    In the top row, we see that the blob method discretization successfully captures the two regimes of behavior in the sandpile dynamics: the region on which diffusion is occurring and the region on which no diffusion is occurring. 
    Particularly when $h$ is small, this discretization is able to capture the growth of the critical region at the boundary where diffusion occurs.
    In the bottom row, we observe that the RM method achieves the best balance between accuracy and efficiency, though the improvement over the FE and RB methods is not as large as we saw in the slow diffusion setting.

    \begin{figure}
        \centering
        \includegraphics[width=0.4\linewidth]{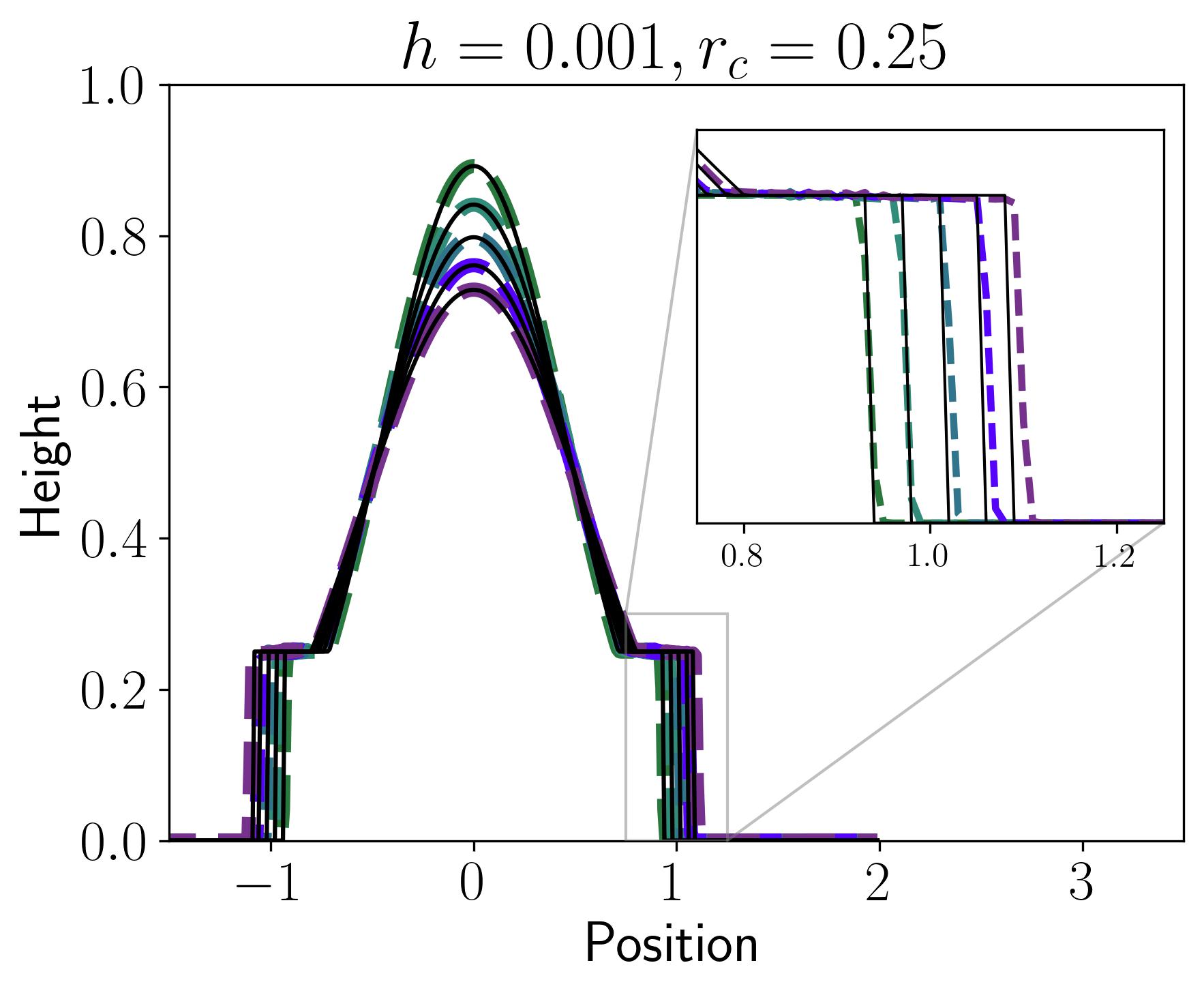}
        \includegraphics[width=0.507\linewidth]{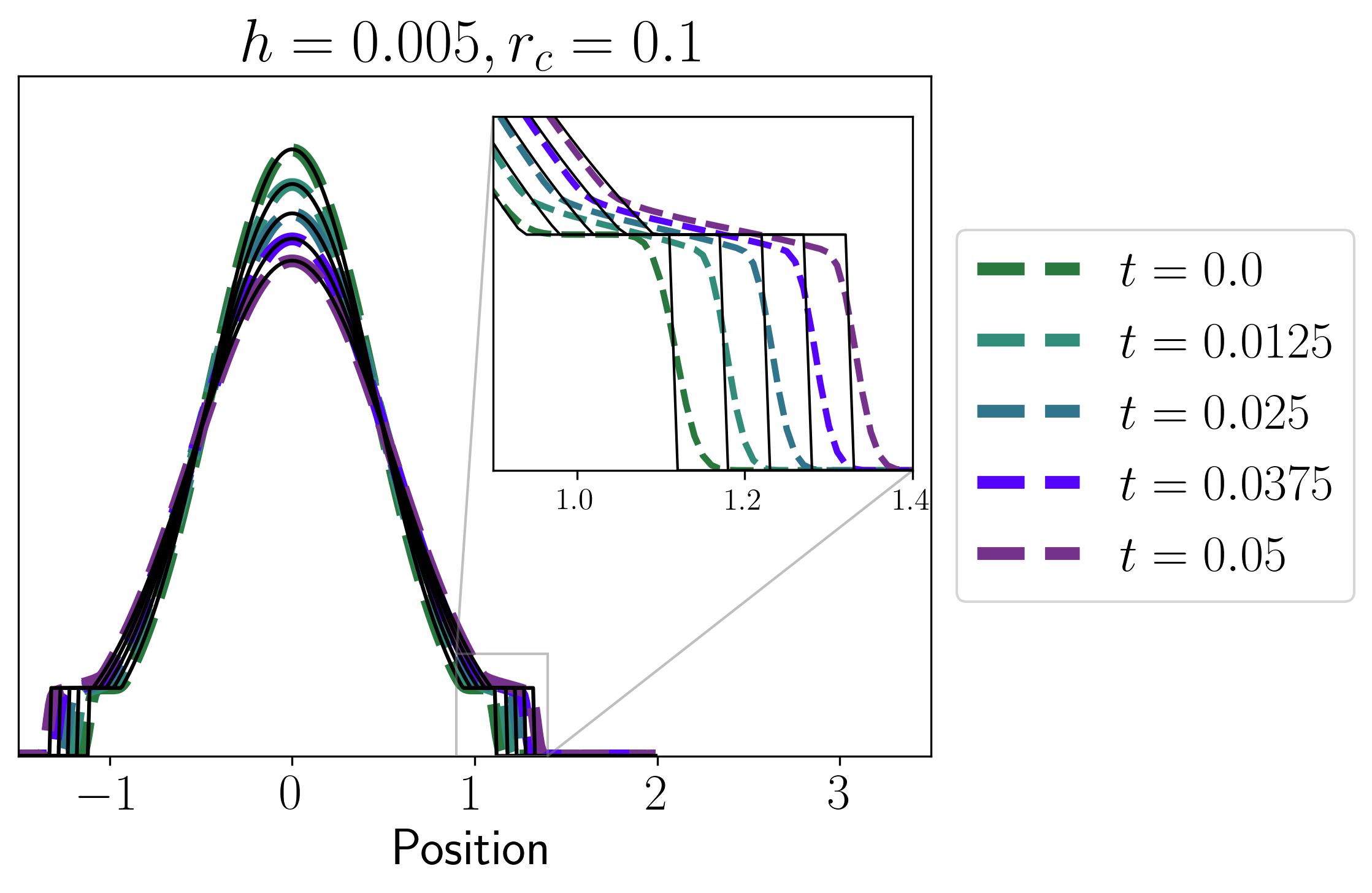} \\
        
        \includegraphics[width=0.5\linewidth]{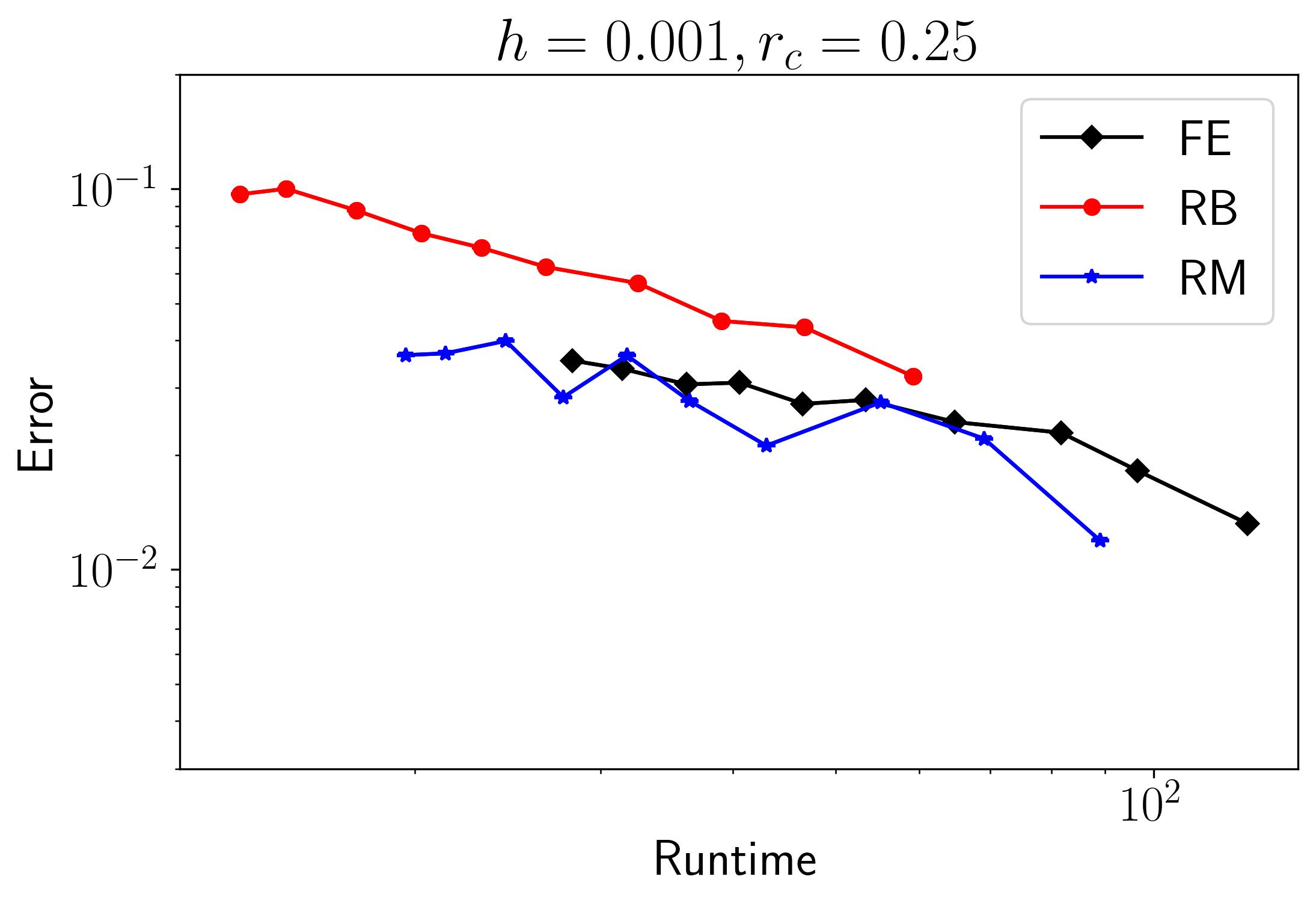}
        \includegraphics[width=0.438\linewidth]{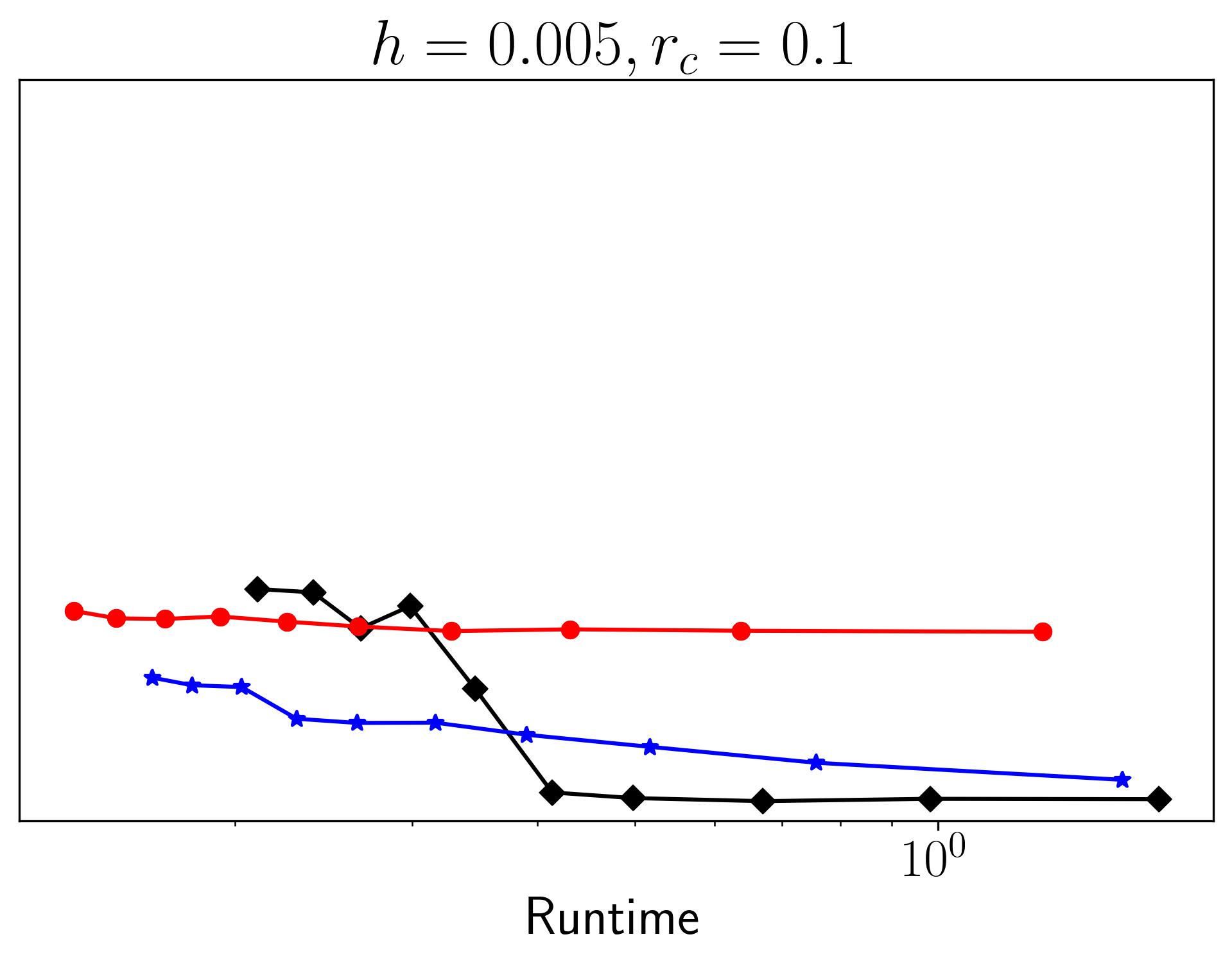}
        \caption{In the top row, we calculate the particle solution to the sandpile regime for different choices of critical threshold $r_c$.
        In the bottom row, for the FE, RB, and RM methods, and for various time steps, we approximate the 2-Wasserstein distance between the particle solution at time $T = 0.05$ and the exact continuum solution. }
        \label{fig18}
    \end{figure}

\section{Two-Dimensional Simulations}

\label{sec:2d}

\subsection{2D Diffusion}
    \label{sec:2d_diffusion}
    The trends we observed in one dimension largely persist in two dimensions: the RM method exhibits an excellent tradeoff between and accuracy and efficiency, particularly in the slow diffusion regime.
    To begin our two-dimensional simulations, we consider the slow diffusion regime with $m = 5$ and no external velocity.
    On the left-hand side of Figure \ref{fig10}, we repeat the experimental setup in Section \ref{sec:diffusion_exact}, adapted to the $d = 2$ setting.
    We set $T = 0.1$, $h = 0.01$.
    We include vertical bars showing the standard deviation in error over five random seeds.
    We do not include variation in runtime, as the runtime of each simulation is relatively large, and variations in runtime due to background processes (e.g. mouse movement, scheduled tasks, etc.) are negligible compared to the total runtime.

    We see that the RM method outperforms both the FE and RB methods; for any choice of time step size, the RM method achieves greater accuracy than either of these other methods.
    While in the $d = 1$ version of this plot (Figure \ref{fig12}), the RB method is outperformed by both the FE and RM methods in one dimension, in the $d = 2$ case, the RB method outperforms the FE method.

    \begin{figure}
        \centering
        \includegraphics[width=.45\linewidth]{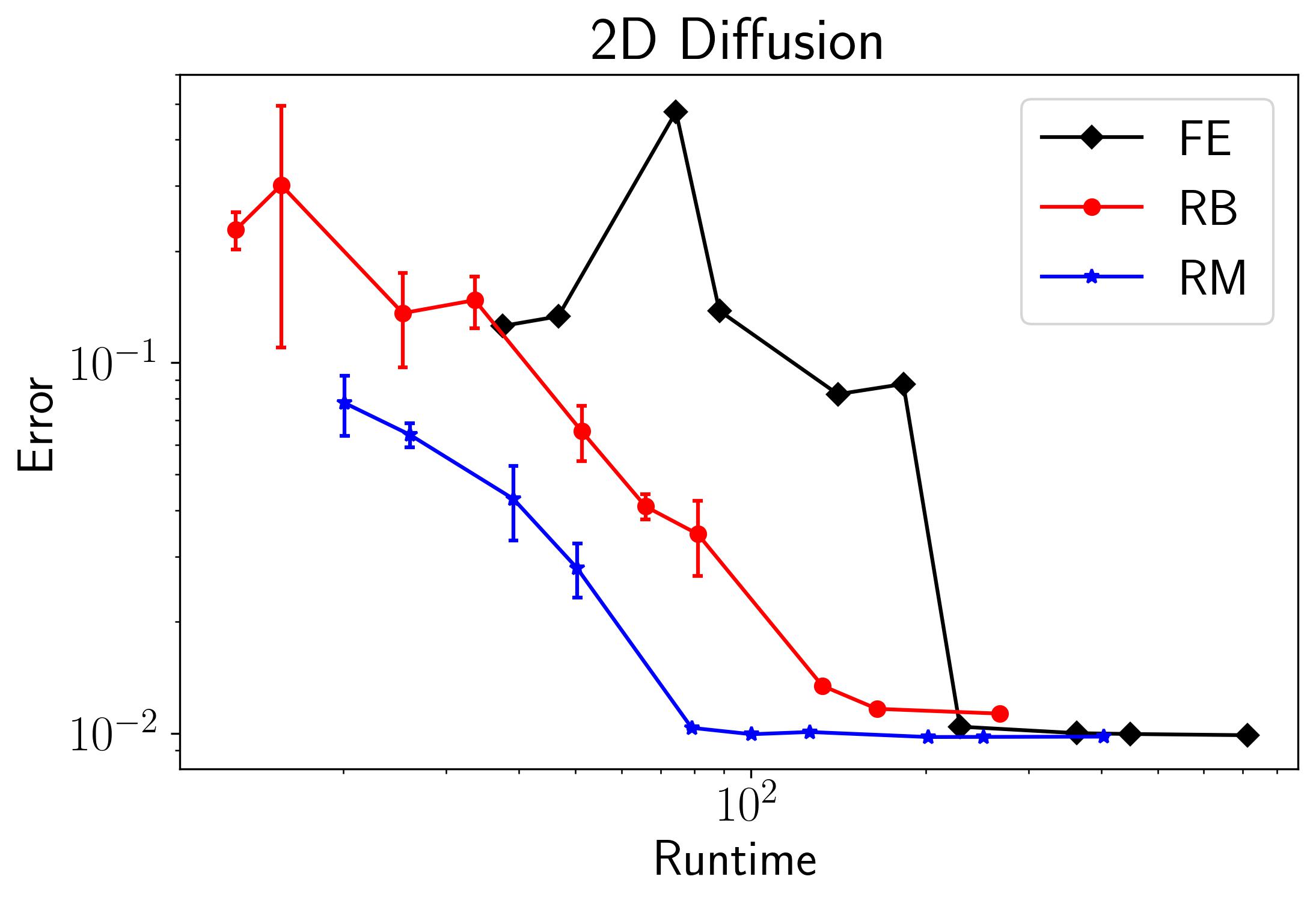}
        \includegraphics[width = .45 \linewidth]{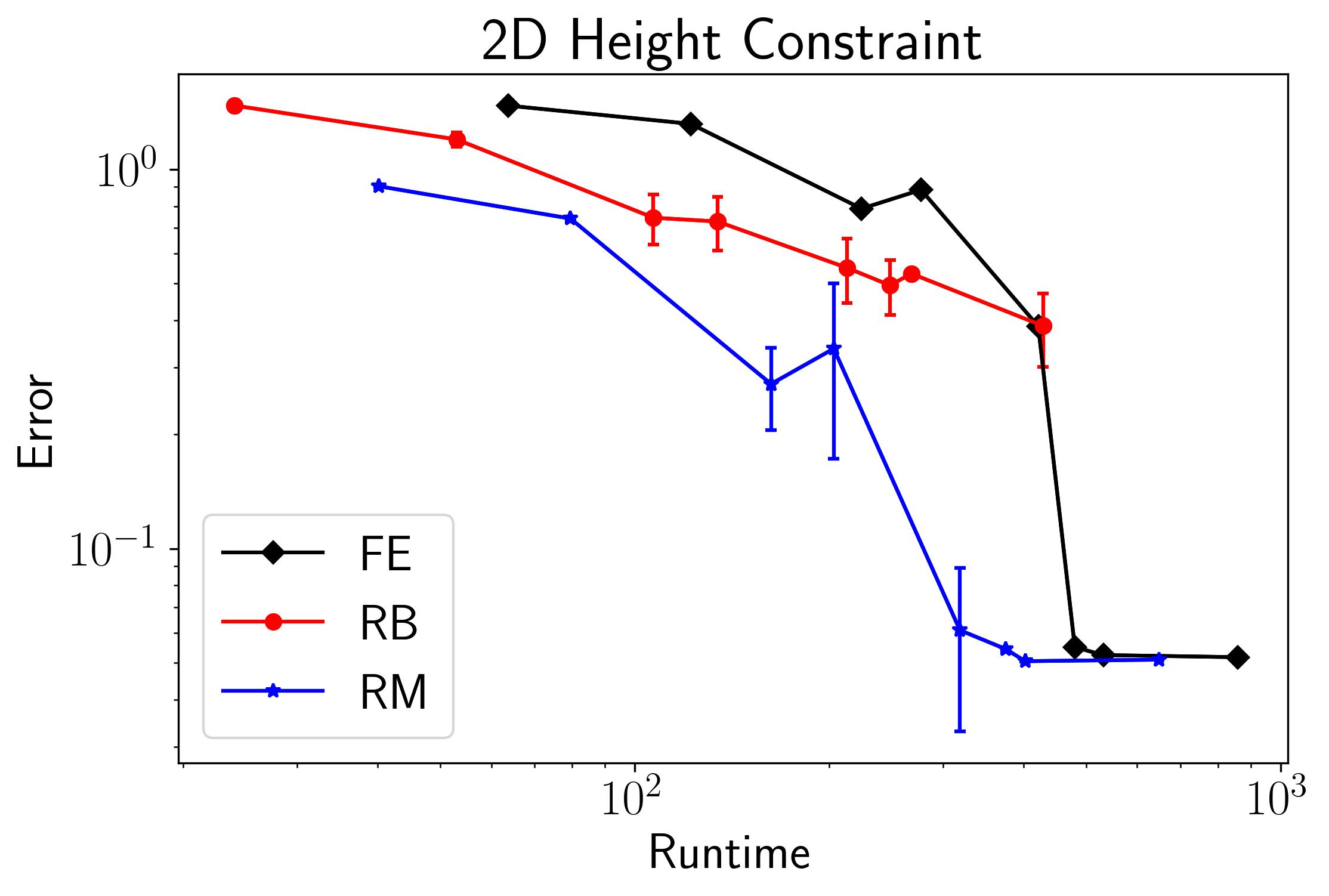}
        
        \caption{Two examples of 2D nonlinear diffusion equations. 
        On the left is the 2D nonlinear diffusion equation with $m = 5$ and external velocity field $v(x) = 0$, simulated up to final time $T = 0.1$. 
        On the right is the 2D height-constraint case with external velocity field $v(x, y) = (x, y)$, simulated up to final time $T = 1.1$.}
        \label{fig10}
    \end{figure}

\subsection{2D Height-Constraint}
    \label{sec:2d_height_constraint}

    We now consider the $d = 2$ analog of the height-constraint dynamics first explored in Section \ref{sec:height_constraint}.
    As in Section \ref{sec:height_constraint}, we approximate the internal energy density $f$ given by equation \eqref{eqn12} via $f_{\eps}$ as defined by equation \eqref{eqn1} with $m = \frac{1}{\eps} = 100$.
    The external velocity field is $v(x, y) = (x, y)$.
    We choose initial distribution
    \[ \rho_0(x, y) = \begin{cases}
        \frac{1}{e^2} & |x| < \frac{e}{2} \text{ and } |y| < \frac{e}{2} \\
        0 & \text{otherwise.}
    \end{cases}\]
    With this initial distribution, we hit the ``ceiling" one at time $T_{ceil} = 1$.

    On the right-hand side of Figure \ref{fig10}, we use the FE, RB, and RM methods for various time step sizes to calculate the particle solution up to final time $T = 1.1$.
    We set $h = 0.025$.
    As in the $d = 1$ case, particles are constrained to remain in the square $[-3, 3] \times [-3, 3]$.
    As in Section \ref{sec:2d_diffusion}, we include bars showing the standard deviation in error over five random seeds.

    Again, we compare the performance of our methods in one dimension (Figure \ref{fig5}) and two dimensions (right-hand side of Figure \ref{fig10}).
    In one dimension, the RM method preserves accuracy for coarser time steps than the FE method.
    In two dimensions, the RM method effectively ``shifts" the runtime-error line of the FE method leftward.
    In other words, for almost every time step, the FE and RM methods are equally accurate, although the RM method is faster.
    In one dimension, the RB method is inaccurate for all random seeds chosen, while in two dimensions, the RB method performs better in terms of the accuracy / efficiency tradeoff, with large variation in error its main drawback.

\subsection{The 2D Navier-Stokes Equation}

    We conclude by considering the behavior of the RB and RM methods when used to simulate the viscosity formulation of the 2D Navier-Stokes equation, as in equation \eqref{eqn16}.
    Note that the ODE system in system \eqref{eqn:xdot} has become slightly more complex, as the final term $v(x_j(t))$ implicitly depends on the locations of all particles.
    In fact, this is a key reason blob methods for diffusion are an interesting choice for this equation: since computing the evolution of the particles according to the velocity $v$ already entails solving an interacting particle system, choosing a numerical approximation of the diffusion in terms of an interacting particle system does not increase the computational complexity in the way an approximation by, e.g., Brownian motion would.
    
    In Figure \ref{fig14}, we compute the particle solutions to this equation with viscosity coefficient $\nu = 0.1$, using the FE, RB, and RM methods, for ten evenly spaced time steps, and consider the average runtime and error over five random seeds.
    In all cases, we set $h = 0.025$ and $T = 0.5$.
    We use Nordmark's eighth-order cutoff function with $\delta = 45h^{0.99}$ to regularize the Biot-Savart kernel \cite{Nordmark1995}. 
    Given that the initial solution is a Gaussian, the exact continuum solution is a Gaussian evolving according to the heat equation with diffusion coefficient $\nu$ \cite{Bertozzi2002}.
    While the RM method continues to achieve a good balance between accuracy and efficiency, its behavior is very similar to that of the FE method in this case.

    \begin{figure}
        \centering
        \includegraphics[width=0.6\linewidth]{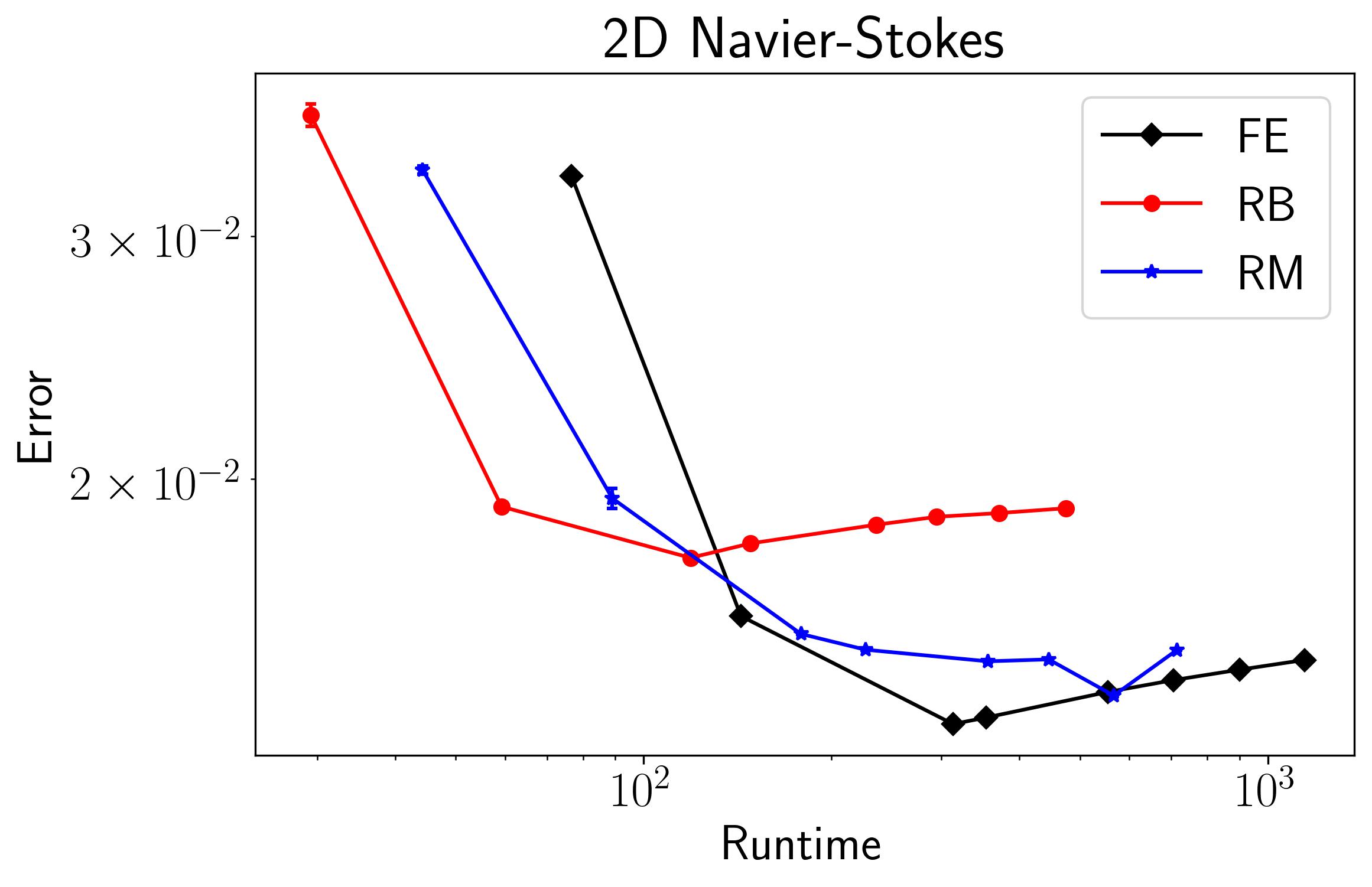}
        \caption{The 2D Navier-Stokes equation.
        For the FE, RB, and RM methods, and for various time steps, we approximate the 2-Wasserstein distance between the particle solution at time $T = 0.5$ and the exact continuum solution.}
        \label{fig14}
    \end{figure}

\subsection{The Double Bump}
\label{doublebump}

    To illustrate the strength of the RM method, we conclude this work with an application, demonstrating that the RM approach can successfully capture key qualitative features of solutions to the Navier-Stokes equation. 
    We compare the results of our simulations to the classical work by Nordmark \cite{Nordmark1995}, which utilizes regridding to simulate the Navier-Stokes equation in the case $\nu = 0.0005$ with a ``double bump" initial condition given by 
    \[ \rho_0(x, y) = C(\max(0, 1 - x^2) \max(0, 1 - y^2))^7 . \]
    Here, $C$ is a constant chosen so that the function integrates to 1.
    
    In Figure \ref{fig13}, we reproduce this simulation (without regridding) using the RM method with time step $0.05$ up to $T = 11.3$. 
    We consider the case of $4688$ particles, which are initially distributed in concentric circles.
    Instead of using Hald's infinite-order cutoff function, as in \cite{Nordmark1995}, we use Nordmark's eighth order cutoff function with $\delta = 45h^{0.99}$.
    This choice decreased the runtime of simulations by around 75\%, while maintaining a high amount of accuracy.

    \begin{figure}
        \centering

         \includegraphics[width=0.225\linewidth]{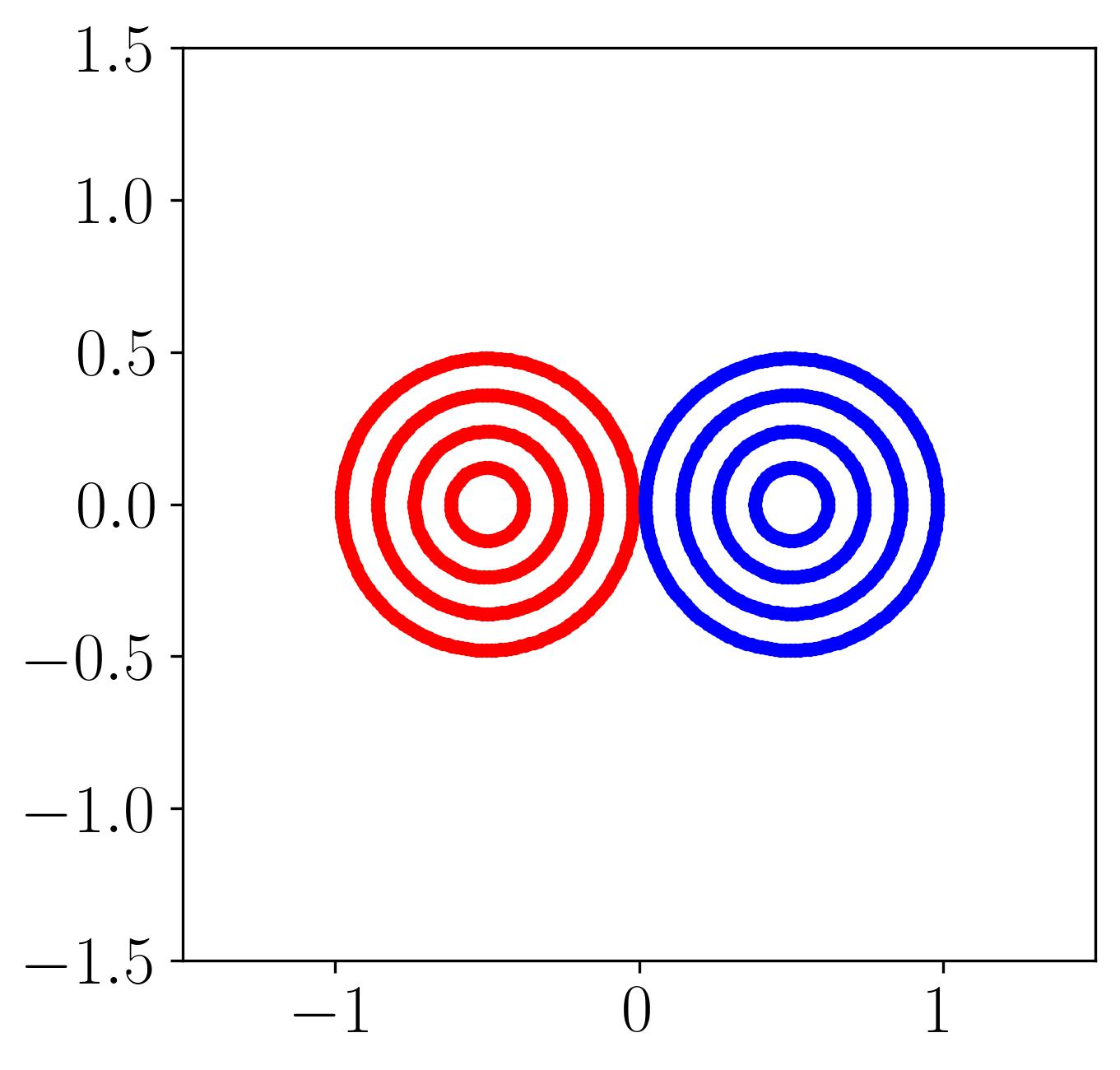}   
        \includegraphics[width=0.2225\linewidth]{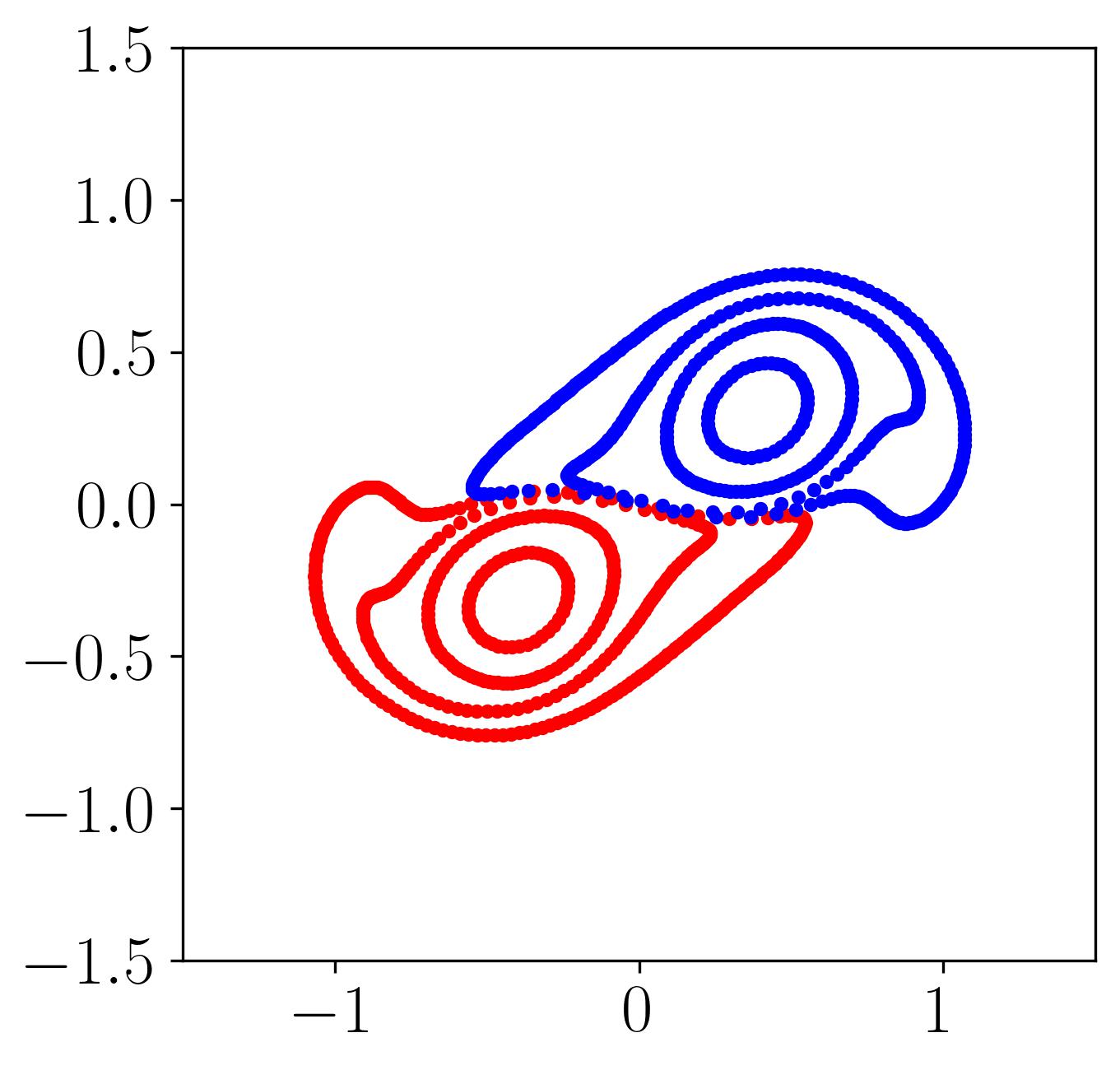} 
        \includegraphics[width=0.225\linewidth]{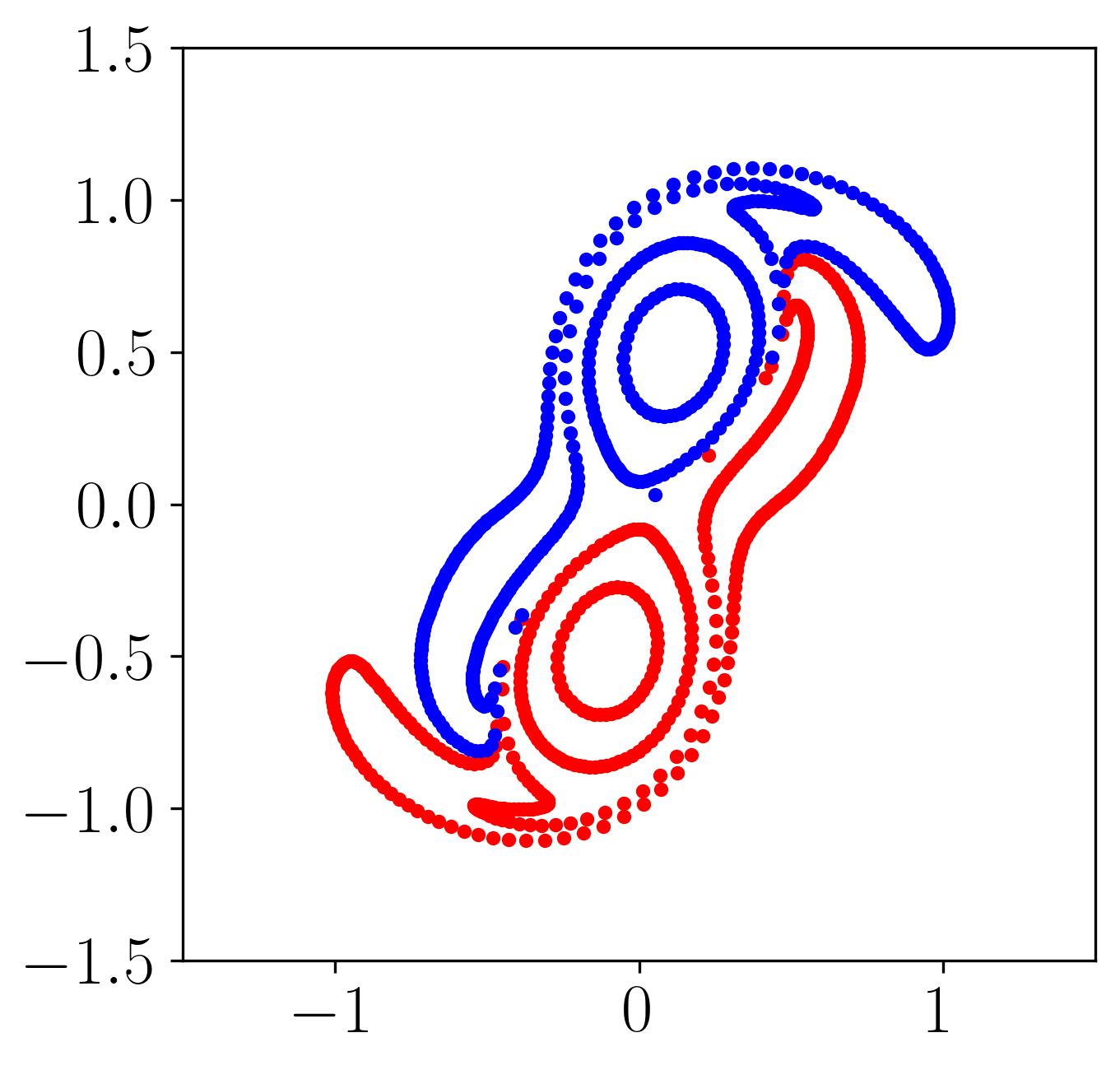}
        \includegraphics[width=0.225\linewidth]{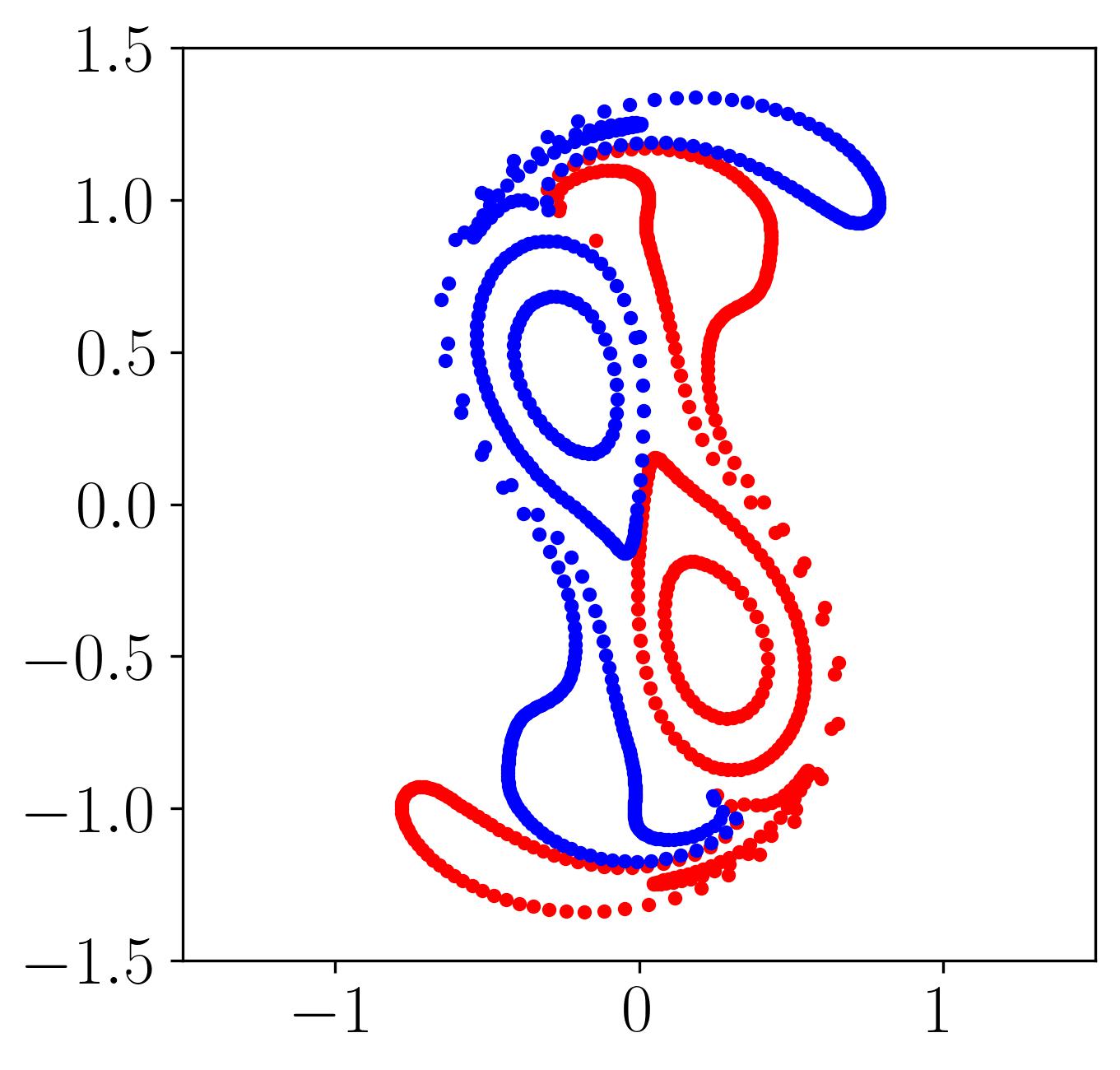} \\
        
        \caption{Reproducing a solution to the Navier-Stokes equation with a ``double bump" initial distribution. We present our results using the RM method.}
        \label{fig13}
    \end{figure}

\section{Conclusion}

A wide range of PDEs can be expressed as variants of the nonlinear Fokker-Planck equation.
While blob methods have been successfully used to develop a particle method for this PDE \cite{Burger2023, carrillo2019blobmethoddiffusion, CarrilloCraigYao2019, Craig2023,  Carrillo2024, Craig2024}, existing implementations of this approach have suffered from high computational complexity.
In this work, we compare the random batch (RB) method with a novel stochastic time discretization, called the random multirate (RM) method; both methods circumvent this computational bottleneck.
We prove that the RM method converges to the underlying ODE system at a rate of $O(k \Delta t)$, and prove by example that this rate is sharp.
We then simulate a wide range of diffusive PDEs in one and two dimensions to demonstrate the computational benefits of this method and to contrast the RM and RB methods.
In practice, we find that for ODEs arising from blob methods for diffusion equations, the RM method is even more accurate than theoretically guaranteed, outperforming the RB and forward Euler method, especially in the slow diffusion regime.

 There are several directions for future work. For example, blob methods are a natural choice for aggregation diffusion equations, due to the fact that the aggregation term already requires computation of an interacting particle system, so using a blob method to simulate the diffusion does not increase the computational complexity. In this vein, blob methods have already been implemented to simulate food seeking behavior of slime mold \cite{gottlich2020food}.
 We are optimistic that the RM method could improve the computational efficiency of this approach, while preserving accuracy.

A second direction for future work is application of this approach to develop an efficient particle method for mean-field variational inference, where 2-Wasserstein gradient flows are used to approximate a high-dimensional target measure by a product of finitely supported lower-dimensional measures \cite{jiang2025algorithmsmeanfieldvariationalinference}. When the loss function  is  a classical statistical divergence, such as the KL divergence or $\chi^2$ divergence, this boils down to solving linear or nonlinear diffusion equations for each lower-dimensional component---a context in which the RM method has already shown promising results.

\appendix

\section{Error Bounds for the RM Method}
\label{sec:error_append}
The proof of Theorem \ref{thm:error_analysis} follows the classical consistency-and-stability strategy, relying on two lemmas.
Lemma \ref{lte} establishes consistency; it bounds the error of the RM method over $k$ time steps when the initial solution is the exact ODE solution. 
Lemma \ref{propBound} establishes stability; it controls how an initial discrepancy between two RM trajectories is amplified over $k$ time steps. 
Lemma \ref{discrete_gronwall}, a discrete Gr\"onwall-type inequality, accumulates these estimates into a global bound.

First, we recall a classical discrete Gr\"onwall-type inequality.

\begin{lemma}[{c.f. Proposition 3.2 in \cite{Emmrich1999}}]
    \label{discrete_gronwall}
    Fix $k \in \N$.
    Suppose that for all $\ell \in \N$, $0 \leq \ell \leq k - 1$,
    $\varphi^{\ell + 1} \leq p \varphi^{\ell} + q $,
    where $p, q \geq 0$, $p \neq 1$.
    Then, 
    \[ \varphi^k \leq p^k \varphi^0 + q \left ( \frac{p^k - 1}{p - 1} \right ) . \]
\end{lemma}

We now use the above inequality to prove the consistency and stability lemmas.
Throughout the proofs, for a given $I_{\text{fine}}$ and $I_{\text{coarse}}$, we define 
\[ \tilde{x}^m_i := L(i, m, I_{\text{fine}}, I_{\text{coarse}}) , \]
where $L$ is as defined in equation \eqref{eqn19}.
\begin{lemma}[Consistency]
    \label{lte}
     Under the assumptions of Theorem \ref{thm:error_analysis},
    \[ \max_{1 \leq i \leq N}||X_i(k \Delta t) - x_i^{k}|| \leq L_0 L_1 e^{2L_1 k \Delta t} \left ( L_1(k \Delta t)^3 + ( k \Delta t)^2 \right ) . \]
\end{lemma}

\begin{proof}[Proof of Lemma \ref{lte}]
    For ease of notation, for $0 \leq \ell \leq k$, set $t_{\ell} = \ell \Delta t$ and define 
    \begin{equation*}
        E^{{\ell}}_j = ||X_j(t_{{\ell}}) - x_j^{{\ell}}||, \; \; \; 
        E^{{\ell}}_{\text{coarse}} = \max_{j \in \Icoarse} E_j^{{\ell}}, \; \; \; 
        E^{{\ell}}_{\text{fine}} = \max_{j \in \Ifine} E_j^{{\ell}} . 
    \end{equation*}
    By error estimates for the forward Euler method, $E^{k}_{\text{coarse}} \leq L_0 L_1 (k \Delta t)^2$.
    For $i \in \Ifine$, we seek a bound for $E^{{\ell} + 1}_{\text{fine}}$ in terms of $E^{{\ell}}_{\text{fine}}$, which allows us to invoke Lemma \ref{discrete_gronwall} to obtain the desired result.
    Note that, for $0 \leq \ell \leq k - 1$,
    \begin{align*}
        E_i^{{\ell} + 1}
        & = ||X_i(t_{{\ell} + 1}) - x_i^{{\ell} + 1}|| \\
        & \leq || X_i(t_{{\ell}}) + \Delta t G(X_i(t_{{\ell}}), \bm{X}(t_{\ell})) -  (x_i^{\ell} + \Delta tG(x_i^{\ell}, \bm{\tilde{x}}^{\ell})) || + L_0 L_1 (\Delta t)^2 \\
        & \leq E_i^{\ell} + \Delta t ||G(X_i(t_{\ell}), \bm{X}(t_{\ell})) - G(x_i^{\ell}, \bm{\tilde{x}}^{\ell}))|| + L_0 L_1 (\Delta t)^2 \\
        & \leq E_i^{\ell} + \Delta t L_1\left ( E_i^{\ell} + \max_{1 \leq j \leq N} ||X_{j}(t_{\ell}) - \tilde{x}_j^{\ell}|| \right ) + L_0 L_1 ( \Delta t)^2 . 
    \end{align*}
    The second line follows by error estimates of the FE method, the third by the triangle inequality, and the fourth by the Lipschitz assumption on the function $G$.
    Let us consider the maximum in the final line of the above system.
    On the one hand, if $j \in \Ifine$, $\tilde{x}_j^{\ell} = x^{\ell}_{j}$, thus $||X_{j}(t_{\ell}) - \tilde{x}_j^{\ell}|| = E^{\ell}_{j} \leq E^{\ell}_{\text{fine}}$.
    On the other hand, if $j \in \Icoarse$, $||X_{j}(t_{\ell}) - \tilde{x}_j^{\ell}|| \leq L_0 L_1 ((\ell + 1) \Delta t)^2 \leq L_0 L_1 (k \Delta t)^2$, by error estimates for the FE method and convexity. 
    In sum, for $0 \leq \ell \leq k - 1$, we have the bound
    \begin{align*}
        E^{{\ell} + 1}_{\text{fine}}
        & \leq E^{\ell}_{\text{fine}} + \Delta t L_1\left ( 2E^{\ell}_{\text{fine}} + L_0 L_1 (k \Delta t)^2 \right ) + L_0 L_1 ( \Delta t)^2 \\
        & \leq (1 + 2 \Delta tL_1)E^{\ell}_{\text{fine}} + L_0 L_1 \left ( L_1 k^2 (\Delta t)^3 + k ( \Delta t)^2 \right ) .
    \end{align*}
    By Lemma \ref{discrete_gronwall},
    $ E^{k}_{\text{fine}} 
        \leq L_0 L_1 \left ( L_1 k^2 (\Delta t)^3 + k ( \Delta t)^2 \right ) \left ( \frac{(1 + 2L_1 \Delta t)^k - 1} {2 \Delta t L_1} \right ) . 
    $
    Notice that 
    \begin{align*}
       \frac{(1 + 2L_1 \Delta t)^k - 1} {2 \Delta t L_1}
       \leq \frac{e^{2 L_1 k \Delta t} - 1}{2 \Delta t L_1} 
       \leq k e^{2 L_1 k \Delta t}. 
    \end{align*}
    In sum, 
    $ E^{k}_{\text{fine}} \leq L_0 L_1 e^{2L_1 k \Delta t} \left ( L_1(k \Delta t)^3 + ( k \Delta t)^2 \right ) . $
    Considering the bounds for $E^{k}_{\text{fine}}$ and $E^{k}_{\text{coarse}}$ gives the result.
\end{proof}

\begin{lemma}[Stability]
    \label{propBound}
   Under the assumptions of Theorem \ref{thm:error_analysis} and   for $0 \leq {\ell} \leq k$, let $\bm{y}^{\ell}, \bm{x}^{\ell}$ be   solutions to the RM method with $F_{RM} = G$ after ${\ell}$ time steps with initial values $\bm{y}^0, \bm{x}^0 \in \R^{dN}$,   where we select the same $\Ifine$ and $\Icoarse$ for both.
    Then,
    \begin{align*}
        \max_{1 \leq i \leq N}||y_i^{k} - x_i^{k}|| \leq e^{4L_1 k \Delta t} \max_{1 \leq i \leq N}||y_i^{0} - x_i^{0}|| .
    \end{align*}
\end{lemma}

\begin{proof}[Proof of Lemma \ref{propBound}]
    Similar to the proof of Lemma \ref{lte}, define 
    \begin{equation*}
        H^{{\ell}}_i = ||y_i^{\ell} - x_i^{{\ell}}||,  \; \;  
        H^{{\ell}}_{\text{coarse}} = \max_{i \in \Icoarse} H_i^{{\ell}}, \; \;   H^{{\ell}}_{\text{fine}} = \max_{i \in \Ifine} H_i^{{\ell}}  .
    \end{equation*}
    We also define $H^\ell =  \max \left ( H^{{\ell}}_{\text{coarse}} ,H^{{\ell}}_{\text{fine}}  \right)$.
    First, consider the case $i \in \Icoarse$.
    Note that 
    \begin{align*}
        \label{ineq1}
        H_i^{k}
        & = ||y_i^0 + k\Delta t G(y_i^0, \bm{y}^0) - (x_i^0 + k\Delta tG(x_i^0, \bm{x}^0))||   \\
        & \leq ||y_i^0 - x_i^0|| + k \Delta t ||G(y_i^0, \bm{y}^0) - G(x_i^0, \bm{x}^0)|| \\
        & \leq H_i^0 +L_1 k \Delta t H^0  .
    \end{align*}
    Taking the maximum over all $i \in \Icoarse$, we obtain
    \begin{equation}
        \label{ineq2}
        H^k_{\text{coarse}}  \leq (1 + L_1k \Delta t) H^0 .
    \end{equation}
    
    Now, we consider the case $i \in \Ifine$.
  For $0 \leq \ell \leq k - 1$,
    \begin{align*}
        H_i^{{\ell} + 1}
        & = ||y_i^{\ell} + \Delta t G(y_i^{\ell}, \bm{\tilde{y}}^{\ell}) - \left ( x_i^{\ell} + \Delta t G(x_i^{\ell}, \bm{\tilde{x}}^{\ell}) \right ) || \\
        & \leq H_{i}^{\ell} + \Delta t ||G(y_i^{\ell}, \bm{\tilde{y}}^{\ell}) - G(x_i^{\ell}, \bm{\tilde{x}}^{\ell})|| \\
        & \leq H_i^{\ell} + L_1\Delta t \left (H_{i}^{\ell} +  \max_{1 \leq j \leq N} ||\tilde{y}_{j}^{\ell} - \tilde{x}_{j}^{\ell} || \right ) . 
    \end{align*}
    Let us consider the maximum in the final line of the above system.
    On the one hand, if $j \in \Ifine$, $||\tilde{y}_{j}^{\ell} - \tilde{x}_{j}^{\ell} || = ||y_{j}^{\ell} - x_{j}^{\ell} || \leq H^{\ell}_{\text{fine}}$.
    On the other hand, if $j \in \Icoarse$, 
    \begin{align*}
        ||\tilde{y}_{j}^{\ell} - \tilde{x}_{j}^{\ell} ||
        & \leq \left ( 1 - \frac{\ell + 1}{k} \right ) ||y_{j}^0 - x_{j}^0 || + \frac{\ell + 1}{k} ||y_{j}^{k} - x_{j}^{k} || \\
        & \leq \left ( 1 - \frac{\ell + 1}{k} \right ) H^0 + \frac{\ell + 1}{k} (1 + L_1 k \Delta t) H^0  \\
        & \leq (1 + L_1 k \Delta t) H^0 . 
    \end{align*}
    In sum, we obtain
    \begin{align*}
        H^{{\ell} + 1}_{\text{fine}}
        & \leq H^{{\ell}}_{\text{fine}} + L_1 \Delta t \left ( 2 H^{\ell}_{\text{fine}} + (1 + L_1 k \Delta t) H^0 \right ) \\
        & \leq (1 + 2L_1 \Delta t) H^{\ell}_{\text{fine}} + (1 + L_1 k \Delta t) L_1 \Delta tH^0 .
    \end{align*}
    Define $\alpha = (1 + 2 L_1 \Delta t)$, $\beta = (1 + L_1 k \Delta t) $.
    By Lemma \ref{discrete_gronwall}, we obtain
    \begin{align*}
        H^k_{\text{fine}} 
        & \leq \alpha^k H^0_{\text{fine}}  + \beta L_1 \Delta tH^0 \left ( \frac{\alpha^k - 1}{2L_1 \Delta t} \right ) \leq H^0 \left ( \alpha^k + \beta (\alpha^k - 1) \right )  . 
    \end{align*}
    Comparing this inequality with inequality \eqref{ineq2}, we observe that the bound for $H^{k}_{\text{fine}}$ is greater than the bound for  $H^{k}_{\text{coarse}}$ by Bernoulli's inequality.
    In sum, we have
    \begin{equation}
        \label{eqn14}
        \max_{1 \leq i \leq N}||y_i^{k} - x_i^{k}|| \leq \left ( \alpha^{k} + \beta(\alpha^k - 1) \right ) \max_{1 \leq i \leq N}||y_i^{0} - x_i^{0}|| .
    \end{equation}
  Likewise, we have $\alpha^k  \leq e^{2L_1 k \Delta t}$ and $\beta \leq e^{2L_1 k \Delta t}$.
    As a result,
    \begin{align*}
        \alpha^{k} + \beta(\alpha^k - 1)
        \leq e^{2L_1 k \Delta t} + e^{2L_1 k \Delta t} (e^{2L_1 k \Delta t} - 1)
        = e^{4L_1 k \Delta t}. 
    \end{align*}
    Along with inequality \eqref{eqn14}, this gives the result.    
\end{proof}

We also include the proof of Lemma \ref{lm:relax_assumptions}, which allows us to relax the boundedness and Lipschitz conditions on the function $G$.

\begin{proof}[Proof of Lemma \ref{lm:relax_assumptions}]
    First we will show $G$ is locally Lipschitz.
    Since sums and products of locally Lipschitz functions are locally Lipschitz, we only need to argue that $f'' \circ h$ is locally Lipschitz, where $h: \R^{d} \times \R^{Nd} \rightarrow (0, \infty)$ is given by 
    $ h(x, \textbf{u}) = \sum_{j = 1}^N m_j \varphi_{\eps}(x - u_j) . $ 
    Fix $(x, \mathbf{u}) \in \R^d \times \R^{Nd}$ and let $U \subseteq (0, \infty)$ be an open neighborhood of $h(x, \mathbf{u}) \in (0, \infty)$ on which $f''$ is Lipschitz.
    Since $h$ is continuous and locally Lipschitz, we can find an open neighborhood $V$ of $(x, \mathbf{u})$ such that $h(V) \subseteq U$ and $h$ is Lipschitz on $V$.
    Note that $f'' \circ h$ is Lipschitz on $V$.
    Thus $f'' \circ h$ is locally Lipschitz.
    
    Let $\chi: \R^d \rightarrow [0, 1]$ be a smooth function with $\chi = 1$ on $B(0, R)$, $\chi = 0$ outside $B(0, R + 1)$.
    Define 
    \[ \tilde{G}(x, \textbf{u}) = \chi(x) \prod_{i = 1}^N \chi(u_i) G(x, \mathbf{u}) . \]
    By construction, $\tilde{G}$ is Lipschitz, since it is locally Lipschitz with compact support.
    Furthermore $\tilde{G}$ is bounded, since it is continuous with compact support.
    Note that by construction, we may replace $G$ with $\tilde{G}$ throughout our arguments, with no change to the conclusion.
\end{proof}

\section{Notes on Exact Continuum Behavior of Equation \eqref{eqn:PDE}}
    \label{sec:exact}

    In this section, we provide known exact solutions for equation $\eqref{eqn:PDE}$ for certain choices of $f$ and $v$.
    First, consider fast and slow diffusion equations, where $f$ is given by equation $\eqref{eqn1}$ and $v = 0$.
    We study this equation in the top row of Figure \ref{fig2} in Section \ref{sec:diffusion_exact} and in Figure \ref{fig10} in Section \ref{sec:2d_diffusion}.
    Consider the following Gaussian ($m = 1$) or Barenblatt ($m \neq 1$) function:
    \begin{equation}
        \psi(t, x) = 
        \begin{cases}
            \frac{1}{(4 \pi t)^{\frac{d}{2}}} e^{- |x|^2 / 4 t} & m = 1 \\
            t^{- d\beta} (K - \kappa t^{- 2 \beta} |x|^2)_+^{1 / (m - 1)} & m \neq 1. 
        \end{cases} 
        \label{eqn:psi}
    \end{equation} 
    Here, 
    $ \beta = \frac{1}{2 + d(m - 1)} \quad \text{ and } \quad \kappa = \frac{\beta}{2} \left ( \frac{m - 1}{m}\right ) $. 
    $K$ is chosen so that for all times $t$, $\int \psi(t, x) dx = 1$.
    If $\rho_0(x) = \psi(\tau, x)$ for some $\tau > 0$, the exact solution to equation \eqref{eqn:PDE} is
    $\rho(t, x) = \psi(t + \tau, x) $ . 
    In all simulations using this choice of $f$ and $v$, we fix $\tau$ so that $\rho_0(0) = 1$.

    Now consider fast and slow diffusion equations with a quadratic external potential; specifically, suppose $f$ is given by equation \eqref{eqn1} and $v = \nabla V$ for $V(x) = \frac{\beta |x|^2}{2}$.
    We consider this choice of $f$ and $v$ in Sections \ref{sec:batch_size}, \ref{sec:diffusion_type}, \ref{sec:diffusion_exact}, \ref{sec:diffusion_longtime}, and \ref{sec:convergence_rate}.
    Consider the following function:
    \begin{equation}
        \theta(t, x) = e^{ t d \beta} \psi(e^{t} + \tau, xe^{ t\beta}) . 
        \label{eqn:init_dist_3}
    \end{equation}
    If $\rho_0(x) = \theta(\sigma, x)$ for $\sigma \in \R$, the solution to equation \eqref{eqn:PDE} is $\rho(t, x) = \theta(t + \sigma, x)$.
    In all simulations using this choice of $f$ and $v$, we set $\tau = 0.0625$ and choose $\sigma$ so that $\theta(\sigma, 0) = 0.8\rho_{\infty}(0)$, where $\rho_{\infty}$ is the longtime behavior of equation $\eqref{eqn:PDE}$.

    Again, consider fast and slow diffusion equations, i.e. $f$ is given by equation \eqref{eqn1}, but now suppose $v = \nabla V$, where $V: \R^d \rightarrow \R$ is \emph{any} continuously differentiable function that is $\lambda$-convex for $\lambda > 0$.
    In Sections \ref{sec:diffusion_longtime}, \ref{sec:convergence_rate}, \ref{sec:height_constraint}, and $\ref{sec:2d_height_constraint}$, we are interested in the longtime solution to equation \eqref{eqn:PDE}, which is 
        \begin{equation}
            \rho_{\infty}(x) = \begin{cases}
                e^{Z - V(x)} & m = 1 \\
                \left(\frac{Z - V(x)}{m'}\right)_+^{m' - 1} & m \neq 1 .
            \end{cases}
            \label{eqn6}
        \end{equation}
    $Z$ is a constant chosen so that $\int \rho_{\infty}(x) dx = 1$ and $m'$ is the H\"older conjugate of $m$.

    Next, consider sandpile dynamics, where $f$ as given in equation $\eqref{eqn18}$, as discussed in Section \ref{sec:sandpile}.
    We define
    \[ \gamma(t, x) = \mathbbm{1}_{[L_1(t), L_2(t)]} + \frac{e^{- \frac{x^2}{4t^2}}}{(4\pi t^2)^{\frac{1}{2}}}\mathbbm{1}_{[L_2(t), R_1(t)]} + \mathbbm{1}_{[R_1(t), R_2(t)]}. \]
    The functions $L_1(t), L_2(t), R_1(t)$, and $R_2(t)$ are chosen so that $\gamma$ is symmetric about the $y$-axis and mass remains equal to one.
    If the initial solution to the PDE is given by $\rho_0(x) = \gamma(\tau, x)$, then the exact solution at time $t$ is given by $\rho(t, x) = \gamma(t + \tau, x)$. 
    In all simulations using this choice of $f$, we set $\tau = 0.1$.

\section{The Choice of Mollifier}

    \label{sec:moll}

    In all the simulations above, we use a Gaussian mollifier in the ODE system given by system \eqref{eqn:xdot}.
    In this section, we consider different mollifiers.
    For example, consider a variant of the function discussed by Carrillo et al. in \cite{Carrillo2024}:
    \begin{equation}
        \tag{\text{CESW}}
        \eta_2(x) = C_{\alpha}^1 \left (1 + \left |\frac{x}{C_{\alpha}^2} \right |^2 \right )^{-\alpha} .
    \end{equation}
    $C_{\alpha}^1$ and $C_{\alpha}^2$ are constants chosen so that $\varphi_1$ integrates to one and its second moment also equals one. 
    An important element in the analysis of Carrillo et al. on the convergence of the blob method to the continuum PDE \eqref{eqn:PDE} was the assumption $\alpha > \frac{d}{2} + \frac{1}{m}$.
    We also consider a mollifier with compact support \cite{Brezis2011}:
    \begin{equation}
        \tag{\text{Cpt.}}
        \label{eqn7}
        \eta_3(x) = \begin{cases}
            C_1 \exp \left ( \frac{1}{ \left | \frac{x}{C_2} \right |^2 - 1 } \right ) & \left | \frac{x}{C_2} \right | < 1 \\
            0 & \left | \frac{x}{C_2} \right | \geq 1 . 
        \end{cases}
    \end{equation}
    Again, $C_1$ is a constant chosen so $\varphi^2$ integrates to one, and $C_2$ is a constant chosen so that $\int x^2 \varphi^2(x) dx = 1 $.
    For each $\eta_i$, we define the mollifier $\varphi_{\varepsilon} = (\frac{1}{\varepsilon})^d \eta_i \left ( \frac{x}{\varepsilon} \right )$.
    
    In Figure \ref{fig20}, we compare the performance of each of these mollifiers.
    We seek to solve equation \eqref{eqn:PDE} in the case $f$ is given by equation \eqref{eqn1} and $v = 0$.
    We consider the case $m = \frac{3}{4}$ and $m = 2$.
    For $h \in \{ 0.005, 0.01, 0.02, 0.04 \}$, we iterate through each of our mollifiers, calculate the particle solution, and measure the error to the continuum PDE.
    In all cases, we calculate the particle solution using the FE method with time step $\Delta t = 10^{-5}$ up to final time $T = 0.1$. 
    Notice that for both choices of $m$, the Gaussian mollifier displays the most accuracy. 
    On the right-hand panel, the Gaussian and CESW mollifier with $\alpha = 20$ have almost the same amount of accuracy; their lines overlap each other almost entirely.
    
    \begin{figure}
        \centering
        \includegraphics[width=0.395\linewidth]{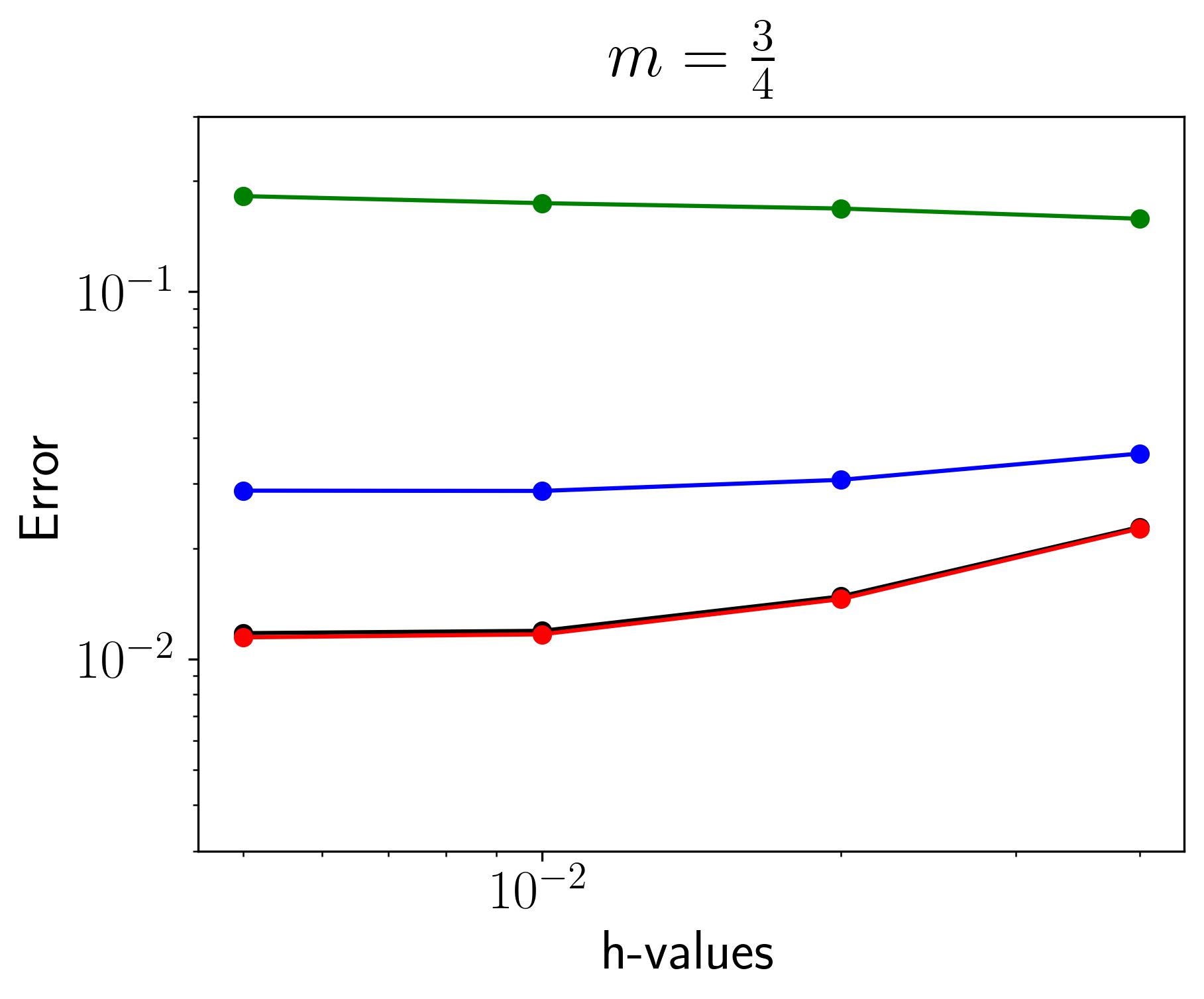}
        \includegraphics[width=0.52\linewidth]{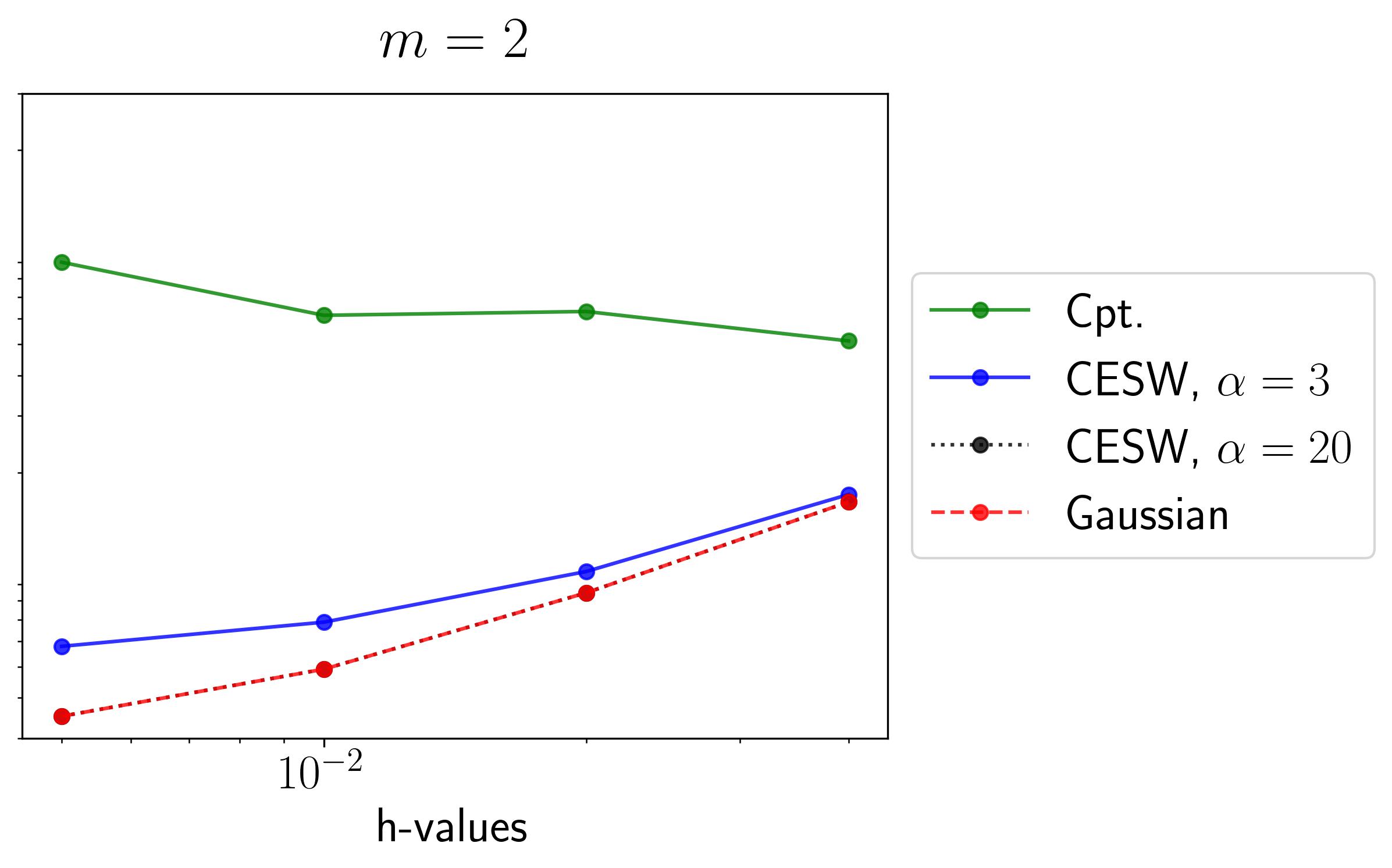}
        \caption{Error between exact and numerical solutions for different choices of mollifiers $\varphi_{\eps}$ in system \eqref{eqn:xdot}. We consider nonlinear diffusion with no external velocity for two choices of $m$.}
        \label{fig20}
    \end{figure}

\section*{Acknowledgments} The work of K. Craig and C. Murphy has been supported
in part by NSF DMS grant 2145900. 
The authors would like to thank Hector Ceniceros, Matt Jacobs, Qin Li, and Olga Turanova for their valuable feedback on this project. 
The second author is grateful to Jack Pfaffinger for many illuminating conversations that helped shape this work. 
The authors used Claude, Gemini, and ChatGPT to conduct literature reviews, formulate theorems and proofs, write code, and prepare citations. The authors assume responsibility for all content.

\newpage
\bibliographystyle{siamplain}
\bibliography{references}

\end{document}